\documentclass[a4paper,12pt]{amsart}

\usepackage{float}
\usepackage{color}
\usepackage{amsfonts, amsmath, amsthm, amssymb,nicefrac}
\usepackage[T1]{fontenc}
\usepackage[cp1250]{inputenc}
\usepackage{xcolor}
\usepackage{graphicx}
\usepackage{mathptmx}
\usepackage{helvet}
\usepackage{courier}
\usepackage{txfonts}
\usepackage{tikz} 
\usetikzlibrary{arrows}

\usepackage{verbatim}

\usepackage{epsfig,amscd,amsxtra}
\usepackage{type1cm}
\usepackage{graphics}
\usepackage[mathscr]{eucal}
\usepackage[all]{xy}
\usepackage{tikz-cd}

\newtheorem{theorem}{Theorem}[section]

\newtheorem{lemma}[theorem]{Lemma}

\newtheorem{corollary}[theorem]{Corollary}
\newtheorem{problem}[theorem]{Problem}
\newtheorem{observation}[theorem]{Observation}

\theoremstyle{definition}
\newtheorem{definition}[theorem]{Definition}
\newtheorem{example}[theorem]{Example}

\newcommand{\R}{\mathbb{R}}
\newcommand{\N}{\mathbb{N}}
\newcommand{\Z}{\mathbb{Z}}

\newcommand{\B}{\mathcal{B}}

\newcommand{\explicitSet}[1]{\left\lbrace #1 \right\rbrace}
\newcommand{\brackets}[1]{\left\langle #1 \right\rangle}
\newcommand{\set}[2]{\explicitSet{#1 \colon #2}}
\newcommand{\seq}[2]{\brackets{#1 \colon #2}}

\newcommand{\e}{\varepsilon}

\newcommand{\0}{\emptyset}

\newcommand{\sub}{\subseteq}
\newcommand{\rest}{\!\restriction\!}

\newcommand{\homeo}{\approx}

\newcommand{\closure}[1]{\overline{#1}}

\begin{document}

\def\joinrel{\mkern-3mu}
\newcommand{\varproj}{\displaystyle \lim_{\multimapinv\joinrel-\joinrel-}}

%%%%%%%%%%%%
\title{Endpoint-homogeneous fans}
%%%%%%%%%%%%

\author{Will Brian}
\address{
W. R. Brian\\
Department of Mathematics and Statistics\\
University of North Carolina at Charlotte\\
Charlotte, NC, USA}
\email{wbrian.math@gmail.com}
\urladdr{wrbrian.wordpress.com}

\author{Rene Gril Rogina}
\address{
R. Gril Rogina\\
Faculty of Sciences and Mathematics\\
University of Maribor\\
Maribor, Slovenia}
\email{rene.gril@student.um.si}

%%%%%%%%%%%%

\thanks{The first author is supported in part by NSF grant DMS-2154229. 
The second author is supported in part by the Slovenian Research Agency (research program J1-4632).}
%%%%%%%%%%%%

%%%%%%%%%%%%
\maketitle
%%%%%%%%%%%%

%%%%%%%%%%%%
\begin{abstract}
A fan $F$ is \emph{endpoint-homogeneous} if for any two endpoints $e,e'$ of $F$, there is a homeomorphism $h: F \to F$ such that $h(e) = e'$. 

We prove there are uncountably many distinct homeomorphism types of endpoint-homogeneous smooth fans. To do this, we associate to each such fan $F$ a topological invariant, in the form of a characteristic subset $EPG(F) \sub [0,1]$ describing how the endpoints of $F$ limit onto any given blade of $F$. We describe precisely all the uncountably many different $X \sub [0,1]$ that can arise as $EPG(F)$ for some endpoint-homogeneous smooth fan $F$. We also prove the existence of $\frac{1}{n}$-homogeneous smooth fans for all $n \geq 5$. 
\end{abstract}
%%%%%%%%%%%%

%%%%%%%%%%%%
\section{Introduction}
%%%%%%%%%%%%

A topological space $X$ is \emph{homogeneous} if for any $x,y \in X$, there is a homeomorphism $h: X \to X$ such that $h(x) = y$. 
This property expresses that $X$ has an especially rich group of self-homeomorphisms. 

Many continua fail to be homogeneous, but still have very rich self-homeomorphism groups. 
The class of \emph{fans} (which is defined precisely in the next section) contains many examples of this kind. 
A fan cannot be homogeneous. It has one ``top'' point, many endpoints, and many points that are neither. These three classes of points are topologically defined, and consequently are invariant under homeomorphisms. 
Nevertheless, some fans, such as the Cantor fan and the Lelek fan, have extremely rich homeomorphism groups (see \cite{AHJ}, \cite{BK}). 

A fan $F$ is called \emph{endpoint-homogeneous} if for any two endpoints $e,e'$ of $F$, there is a homeomorphism $h: F \to F$ such that $h(e) = e'$. Equivalently, $F$ is endpoint-homogeneous if for any two ``blades'' of the fan (arcs connecting the top of the fan to some endpoint), there is a self-homeomorphism of $F$ that restricts to a homeomorphism between these two blades. 

Let $F$ be an endpoint-homogeneous fan and $E$ its set of endpoints. 
Because the blades of $F$ are all topologically indistinguishable, there is a single set $EPG(F) \sub [0,1]$ such that for every blade $B$ of $F$, there is a homeomorphism $B \to [0,1]$ mapping $B \cap \closure{E \setminus B}$ onto $EPG(F)$. In other words, $EPG(F)$ is a topological invariant describing how the endpoints of $F$ limit onto any given blade of $F$. 

In this paper we work towards a classification of endpoint-homogeneous smooth fans. 
Our main results are:

\begin{enumerate}
\item A set $X \sub [0,1]$ is equal to the invariant $EPG(F)$ of some smooth endpoint-homogeneous fan if and only if:
\begin{itemize}
\item[$\circ$] $X = \0$, $\{1\}$, or $[0,1]$, or
\item[$\circ$] $X$ is a closed subset of $[0,1]$ with $0 \in X$ and $1 \notin X$, or
\item[$\circ$] $X$ is a closed subset of $[0,1]$ with $0,1 \in X$, and $1$ is an isolated point of $X$.
\end{itemize} 
\item In particular, it follows that there are uncountably many homeomorphism types of endpoint-homogeneous smooth fans. 
\end{enumerate}

\noindent We conjecture that if $F$ and $G$ are smooth fans and $EPG(F) = EPG(F) \neq \0$, then $F$ and $G$ are homeomorphic. The case $EPG(F) = \0$ is something of a degenerate case; there are countably many homeomorphism types of smooth fans for which the endpoints do not accumulate onto any blades (namely those fans with finitely many blades, also known as the $n$-ods, with $n \geq 3$). 
If true, this conjecture, combined with statement $(1)$ above, would give an exact classification of all endpoint-homogeneous smooth fans. 

We do not say much about non-smooth fans, but we do provide examples of non-smooth endpoint-homogeneous fans, and show 
\begin{enumerate}
\item[$(3)$] There are uncountably many distinct homeomorphism types of non-smooth endpoint-homogeneous fans, including some that violate statement $(1)$ above.
\end{enumerate}

Another natural weakening of homogeneity that can apply to fans, or dendroids more generally, is $\frac{1}{3}$-homogeneity. 
A space $X$ is \emph{$\frac{1}{n}$-homogeneous} if there is a partition of $X$ into $n$ sets (but no fewer) such that if $x,y \in X$ lie in the same partition piece, then there is a homeomorphism $h: X \to X$ with $h(x) = y$. Equivalently, $X$ is $\frac{1}{n}$-homogeneous if the homeomorphism group of $X$ has $n$ distinct orbits with respect to its natural action on $X$. 

As mentioned above, fans have $3$ distinct classes of points that are invariant under homeomorphisms: the top, the endpoints, and everything else (usually called the \emph{ordinary} points). Thus $\frac{1}{3}$-homogeneity is essentially the strongest weakening of homogeneity that can apply to fans. But it is perhaps too strong. 
In \cite{AHJ}, Acosta, Hoehn, and Ju\'arez completely classify all $\frac{1}{3}$-homogeneous smooth fans. 
Other than the simple $n$-ods (again something of a degenerate case), the $\frac{1}{3}$-homogeneous smooth fans come in precisely $4$ homeomorphism types: the Cantor fan, the Lelek fan, and two others. 
They also prove that no smooth fan is $\frac{1}{4}$-homogeneous. 
In addition to the results listed above, we show:
\begin{enumerate}
\item[$(4)$] For every $n \geq 5$, there is an endpoint-homogeneous, $\frac{1}{n}$-homogeneous smooth fan.
\end{enumerate} 

Note that endpoint-homogeneity is a strict weakening of $\frac{1}{3}$-homogeneity: we require only that the endpoints (but not the ordinary points) be an orbit of the homeomorphism group. 
This way of weakening $\frac{1}{3}$-homogeneity brings us from having only $4$ nontrivial examples to having uncountably many. 

It is worth mentioning that $\frac{1}{n}$-homogeneity is a broadly studied topic in continuum theory, not just with $n=3$ and not just with fans. One notable example is \cite{APJ}, which provides a complete classification of all $\frac{1}{3}$-homogeneous dendrites. The introduction of \cite{APJ} contains an extensive list of papers on the topic of $\frac{1}{n}$-homogeneity.

It is also worth mentioning that the opposite of endpoint-homogeneous fans was studied by Hern\'andez-Guti\'errez and Hoehn in \cite{HGH}. They define a fan $F$ to be \emph{endpoint-rigid} if every self-homeomorphism of $F$ restricts to the identity map on the set of endpoints of $F$. They construct many examples of such fans, including several where the set of endpoints of the fan (with the rest of the fan removed) is a homogeneous space.

The characteristic set $EPG(F)$, as described above, is not well-defined for all smooth fans, but only those for which no two blades can be distinguished by looking at limits of endpoints. (We do not know whether this latter class is actually larger than the class of smooth endpoint-homogeneous fans.) 
The fans with this property were first defined in \cite{Connecting}, where the fans $F$ such that $EPG(F) = X \sub [0,1]$ are called the $X$\emph{-endpoint generated} fans (or $X$-EPG fans). 
These arise in a classification of the Mahavier products obtained from a pair of lines meeting at the origin. (In all but a few cases, it is an $X$-EPG fan.) This classification is the main result of \cite{Connecting} and \cite{BGR}. 
The work in the present paper arose from the desire to classify those $X$ for which smooth $X$-EPG fans exist, followed by the realization that all known constructions of these fans are endpoint-homogeneous.

%%%%%%%%%%%%
\section{Definitions and notation}
%%%%%%%%%%%%

In this section we give definitions, notation, and preliminary results that are used throughout the paper. 

\begin{definition}
	A \emph{continuum} is a non-empty compact connected metric space.
	A \emph{subcontinuum} of some continuum $X$ is a subspace of $X$ that is itself a continuum.
\end{definition}

In what follows, we sometimes write $I$ to denote the unit interval $[0,1]$. Any subset of a continuum homeomorphic to $I$ is an \emph{arc}. If $X$ is a continuum and $A \sub X$ is an arc, then the \emph{endpoints} of $A$ are the two points $a,b \in A$ such that $A \setminus \{a\}$ and $A \setminus \{b\}$ are connected. 
A \emph{simple triod} is a continuum homeomorphic to $([-1,1]\times \{0\})\cup (\{0\}\times I)$, i.e., three copies of $I$ glued together at an endpoint. Similarly, a \emph{simple $n$-od} is $n$ copies of $I$ glued together at an endpoint. 
If $T$ is a simple triod (or a simple $n$-od), then the \emph{top point} of $T$ is the unique point $t \in T$ such that $T \setminus \{t\}$ has three connected components (or $n$ components for a simple $n$-od).

\begin{definition}
	$\ $
	\begin{enumerate}
		\item A continuum $X$ is \emph{unicoherent}, if whenever $A$ and $B$ are subcontinua of $X$ such that $X = A\cup B$, then $A\cap B$ is connected.
		\item A continuum $X$ is \emph{hereditarily unicoherent} if every subcontinuum of $X$ is unicoherent.
		\item A continuum $X$ is a \emph{dendroid}, if it is an arcwise connected, hereditarily unicoherent continuum.
		\item Let $X$ be a dendroid. A point $x\in X$ is an \emph{endpoint of $X$}, if it is an endpoint of every arc in $X$ that contains it. The set of all endpoints of $X$ is denoted $E(X)$.
		\item Let $X$ be a dendroid. A point $x \in X$ is a \emph{ramification point} of $X$ if there is a simple triod $T \sub X$ with top $x$. The set of all ramification points of $X$ is denoted $R(X)$. A point in $X \setminus (E(X) \cup R(X))$ is an \emph{ordinary point}; the set of these points is denoted $O(X)$.
		\item A \emph{fan} is a dendroid with exactly one ramification point $t$, which is called the \emph{top} of the fan.
		\item A fan $F$ is \emph{endpoint-homogeneous} if for any $e,e' \in F$, there is a homeomorphism $h: F \to F$ such that $h(e) = e'$.
\end{enumerate}
\end{definition}

\begin{observation}\label{homeo_end-points}
Let $X$ and $Y$ be dendroids and $h:X\to Y$ a homeomorphism. Then $h$ maps 
$E(X)$ onto $E(Y)$, $R(X)$ onto $R(Y)$, and $O(X)$ onto $O(Y)$. 
In particular, no fan is homogeneous.
\end{observation}

Two important examples of fans are the Cantor fan and the Lelek fan. 
The \emph{Cantor fan} is any continuum homeomorphic to $\bigcup_{x \in K} A_x$, where $K \subseteq I$ is the Cantor set, and for each $x \in K$, $A_x$ is the straight line segment in the plane from $(\nicefrac{1}{2},1)$ to $(x,0)$. More succinctly, the Cantor fan is the one-point compactification of $K \times [0,1)$, also known as the cone over $K$.

A \emph{Lelek fan} is any subcontinuum $F$ of a Cantor fan such that $F$ is a fan and $E(F)$ is dense in $F$. Charatonik proved in \cite{lelekUniqueChar} that this description uniquely specifies a continuum up to homeomorphism, i.e., any two Lelek fans are homeomorphic (see \cite{BK0}, \cite{lelekUniqueBO}, and \cite{HGH} for further characterizations).

\begin{center}
\begin{tikzpicture}[xscale=.5,yscale=.55]

\node at (4.5,7.2) {$\ $};

\draw[very thin] (4.5,7) -- (0,0);
\draw[ultra thin] (4.5,7) -- (1/27,0);
\draw[ultra thin] (4.5,7) -- (2/27,0);
\draw[very thin] (4.5,7) -- (1/9,0);
\draw[very thin] (4.5,7) -- (2/9,0);
\draw[ultra thin] (4.5,7) -- (7/27,0);
\draw[ultra thin] (4.5,7) -- (8/27,0);
\draw[very thin] (4.5,7) -- (1/3,0);
\draw[very thin] (4.5,7) -- (2/3,0);
\draw[ultra thin] (4.5,7) -- (19/27,0);
\draw[ultra thin] (4.5,7) -- (20/27,0);
\draw[very thin] (4.5,7) -- (7/9,0);
\draw[very thin] (4.5,7) -- (8/9,0);
\draw[ultra thin] (4.5,7) -- (25/27,0);
\draw[ultra thin] (4.5,7) -- (26/27,0);
\draw[very thin] (4.5,7) -- (1,0);
\draw[very thin] (4.5,7) -- (2,0);
\draw[ultra thin] (4.5,7) -- (55/27,0);
\draw[ultra thin] (4.5,7) -- (56/27,0);
\draw[very thin] (4.5,7) -- (19/9,0);
\draw[very thin] (4.5,7) -- (20/9,0);
\draw[ultra thin] (4.5,7) -- (61/27,0);
\draw[ultra thin] (4.5,7) -- (62/27,0);
\draw[very thin] (4.5,7) -- (7/3,0);
\draw[very thin] (4.5,7) -- (8/3,0);
\draw[ultra thin] (4.5,7) -- (73/27,0);
\draw[ultra thin] (4.5,7) -- (74/27,0);
\draw[very thin] (4.5,7) -- (25/9,0);
\draw[very thin] (4.5,7) -- (26/9,0);
\draw[ultra thin] (4.5,7) -- (79/27,0);
\draw[ultra thin] (4.5,7) -- (80/27,0);
\draw[very thin] (4.5,7) -- (3,0);
\draw[very thin] (4.5,7) -- (6,0);
\draw[ultra thin] (4.5,7) -- (163/27,0);
\draw[ultra thin] (4.5,7) -- (164/27,0);
\draw[very thin] (4.5,7) -- (55/9,0);
\draw[very thin] (4.5,7) -- (56/9,0);
\draw[ultra thin] (4.5,7) -- (169/27,0);
\draw[ultra thin] (4.5,7) -- (170/27,0);
\draw[very thin] (4.5,7) -- (19/3,0);
\draw[very thin] (4.5,7) -- (20/3,0);
\draw[ultra thin] (4.5,7) -- (181/27,0);
\draw[ultra thin] (4.5,7) -- (182/27,0);
\draw[very thin] (4.5,7) -- (61/9,0);
\draw[very thin] (4.5,7) -- (62/9,0);
\draw[ultra thin] (4.5,7) -- (187/27,0);
\draw[ultra thin] (4.5,7) -- (188/27,0);
\draw[very thin] (4.5,7) -- (7,0);
\draw[very thin] (4.5,7) -- (8,0);
\draw[ultra thin] (4.5,7) -- (217/27,0);
\draw[ultra thin] (4.5,7) -- (218/27,0);
\draw[very thin] (4.5,7) -- (73/9,0);
\draw[very thin] (4.5,7) -- (74/9,0);
\draw[ultra thin] (4.5,7) -- (223/27,0);
\draw[ultra thin] (4.5,7) -- (224/27,0);
\draw[very thin] (4.5,7) -- (25/3,0);
\draw[very thin] (4.5,7) -- (26/3,0);
\draw[ultra thin] (4.5,7) -- (235/27,0);
\draw[ultra thin] (4.5,7) -- (236/27,0);
\draw[very thin] (4.5,7) -- (79/9,0);
\draw[very thin] (4.5,7) -- (80/9,0);
\draw[ultra thin] (4.5,7) -- (241/27,0);
\draw[ultra thin] (4.5,7) -- (242/27,0);
\draw[very thin] (4.5,7) -- (9,0);

\begin{scope}[shift={(19,7)}]
\draw[very thin] (0,0) -- (-4.5*.31,-7*.31);
\draw[very thin] (0,0) -- (-4.5*.41+1/27*.41,-7*.41);
\draw[very thin] (0,0) -- (-4.5*.59+2/27*.59,-7*.59);
\draw[very thin] (0,0) -- (-4.5*.26+1/9*.26,-7*.26);
\draw[very thin] (0,0) -- (-4.5*.53+2/9*.53,-7*.53);
\draw[very thin] (0,0) -- (-4.5*.58+7/27*.58,-7*.58);
\draw[very thin] (0,0) -- (-4.5*.97+8/27*.97,-7*.97);
\draw[very thin] (0,0) -- (-4.5*.93+1/3*.93,-7*.93);
\draw[very thin] (0,0) -- (-4.5*.23+2/3*.23,-7*.23);
\draw[very thin] (0,0) -- (-4.5*.84+19/27*.84,-7*.84);
\draw[very thin] (0,0) -- (-4.5*.62+20/27*.62,-7*.62);
\draw[very thin] (0,0) -- (-4.5*.64+7/9*.64,-7*.64);
\draw[very thin] (0,0) -- (-4.5*.33+8/9*.33,-7*.33);
\draw[very thin] (0,0) -- (-4.5*.83+25/27*.83,-7*.83);
\draw[very thin] (0,0) -- (-4.5*.27+26/27*.27,-7*.27);
\draw[very thin] (0,0) -- (-4.5*.95+1*.95,-7*.95);
\draw[very thin] (0,0) -- (-4.5*.02+2*.02,-7*.02);
\draw[very thin] (0,0) -- (-4.5*.88+55/27*.88,-7*.88);
\draw[very thin] (0,0) -- (-4.5*.41+56/27*.41,-7*.41);
\draw[very thin] (0,0) -- (-4.5*.97+19/9*.97,-7*.97);
\draw[very thin] (0,0) -- (-4.5*.16+20/9*.16,-7*.16);
\draw[very thin] (0,0) -- (-4.5*.93+61/27*.93,-7*.93);
\draw[very thin] (0,0) -- (-4.5*.99+62/27*.99,-7*.99);
\draw[very thin] (0,0) -- (-4.5*.37+7/3*.37,-7*.37);
\draw[very thin] (0,0) -- (-4.5*.51+8/3*.51,-7*.51);
\draw[very thin] (0,0) -- (-4.5*.05+73/27*.05,-7*.05);
\draw[very thin] (0,0) -- (-4.5*.82+74/27*.82,-7*.82);
\draw[very thin] (0,0) -- (-4.5*.09+25/9*.09,-7*.09);
\draw[very thin] (0,0) -- (-4.5*.74+26/9*.74,-7*.74);
\draw[very thin] (0,0) -- (-4.5*.94+79/27*.94,-7*.94);
\draw[very thin] (0,0) -- (-4.5*.45+80/27*.45,-7*.45);
\draw[very thin] (0,0) -- (-4.5*.92+3*.92,-7*.92);
\draw[very thin] (0,0) -- (-4.5*.3+6*.3,-7*.3);
\draw[very thin] (0,0) -- (-4.5*.78+163/27*.78,-7*.78);
\draw[very thin] (0,0) -- (-4.5*.16+164/27*.16,-7*.16);
\draw[very thin] (0,0) -- (-4.5*.4+55/9*.4,-7*.4);
\draw[very thin] (0,0) -- (-4.5*.62+56/9*.62,-7*.62);
\draw[very thin] (0,0) -- (-4.5*.86+169/27*.86,-7*.86);
\draw[very thin] (0,0) -- (-4.5*.2+170/27*.2,-7*.2);
\draw[very thin] (0,0) -- (-4.5*.89+19/3*.89,-7*.89);
\draw[very thin] (0,0) -- (-4.5*.98+20/3*.98,-7*.98);
\draw[very thin] (0,0) -- (-4.5*.62+181/27*.62,-7*.62);
\draw[very thin] (0,0) -- (-4.5*.8+182/27*.8,-7*.8);
\draw[very thin] (0,0) -- (-4.5*.34+61/9*.34,-7*.34);
\draw[very thin] (0,0) -- (-4.5*.82+62/9*.82,-7*.82);
\draw[very thin] (0,0) -- (-4.5*.53+187/27*.53,-7*.53);
\draw[very thin] (0,0) -- (-4.5*.42+188/27*.42,-7*.42);
\draw[very thin] (0,0) -- (-4.5*.11+7*.11,-7*.11);
\draw[very thin] (0,0) -- (-4.5*.7+8*.7,-7*.7);
\draw[very thin] (0,0) -- (-4.5*.67+217/27*.67,-7*.67);
\draw[very thin] (0,0) -- (-4.5*.98+218/27*.98,-7*.98);
\draw[very thin] (0,0) -- (-4.5*.21+73/9*.21,-7*.21);
\draw[very thin] (0,0) -- (-4.5*.48+74/9*.48,-7*.48);
\draw[very thin] (0,0) -- (-4.5*.08+223/27*.08,-7*.08);
\draw[very thin] (0,0) -- (-4.5*.65+224/27*.65,-7*.65);
\draw[very thin] (0,0) -- (-4.5*.13+25/3*.13,-7*.13);
\draw[very thin] (0,0) -- (-4.5*.28+26/3*.28,-7*.28);
\draw[very thin] (0,0) -- (-4.5*.23+235/27*.23,-7*.23);
\draw[very thin] (0,0) -- (-4.5*.06+236/27*.06,-7*.06);
\draw[very thin] (0,0) -- (-4.5*.64+79/9*.64,-7*.64);
\draw[very thin] (0,0) -- (-4.5*.7+80/9*.7,-7*.7);
\draw[very thin] (0,0) -- (-4.5*.93+241/27*.93,-7*.93);
\draw[very thin] (0,0) -- (-4.5*.84+242/27*.84,-7*.84);
\draw[very thin] (0,0) -- (-4.5*.46+9*.46,-7*.46);

\node at (0,-8.2) {Lelek fan};
\end{scope}

\node at (4.5,-1.2) {Cantor fan};

\end{tikzpicture}
\end{center}

Observe that dendroids generally, and fans in particular, are \emph{uniquely arcwise connected}, which means that if $X$ is a dendroid and $x,y \in X$ with $x \neq y$, then there is exactly one arc in $X$ whose endpoints are $x$ and $y$. (Some such arc must exist, because $X$ is arcwise connected, but any such arc is unique because $X$ is hereditarily unicoherent.)

\begin{definition}
Let $X$ be a dendroid. Given $x,y \in X$ with $x \neq y$, let $[x,y]$ denote the unique arc in $X$ with endpoints $x$ and $y$. If $x = y$, let $[x,y] = \{x\}$.

If $t$ is the top of a fan $F$ and $e \in E(F)$, then $[t,e]$ is a \emph{blade} of $F$. The set of all blades of $F$ is denoted $\B(F)$.
\end{definition}

\begin{lemma}\label{lem:blades}
A fan $F$ is endpoint-homogeneous if and only if for any blades $B,B' \in \B(F)$, there is a homeomorphism $F \to F$ mapping $B$ onto $B'$.
\end{lemma}
\begin{proof}
For the ``only if'' direction, suppose $F$ is an endpoint-homogeneous fan. Given two blades $B$ and $B'$ of $F$, there are (by the definition of a blade) $e,e' \in E(F)$ such that $B = [t,e]$ and $B' = [t,e']$, where $t$ denotes the top of $F$. But then there is a homeomorphism $h: F \to F$ such that $h(e) = e'$. Because $h(t) = t$, this means $h$ maps $B = [t,e]$ onto $[h(t),h(e)] = [t,e'] = B'$. 

For the ``if'' direction, suppose that for any two blades $B,B'$ of $F$, there is a homeomorphism $F \to F$ mapping $B$ onto $B'$. Let $e,e' \in E(F)$, and fix a homeomorphism $h: F \to F$ mapping $[t,e]$ onto $[t,e']$, where $t$ denotes the top of $F$. Because $h(t) = t$, we must have $h(e) = e'$.
\end{proof}

\begin{definition}\label{def:smooth}
A fan $F$ with top $t$ is \emph{smooth} if for any convergent sequence $\seq{x_n}{n \in \N}$ of points in $F$,
$$\textstyle \lim_{n \to \infty} [t,x_n] \,=\, [t,\lim_{n \to \infty}x_n]$$
where the limit on the left is taken with respect to the Hausdorff metric.
\end{definition}

\vspace{1mm}

\begin{center}
\begin{tikzpicture}[scale=.7]
	\filldraw[black,  thick] (0,0) -- (0,3);
	
	\filldraw[black,  thick] (0,0) -- (-2,3);
	\filldraw[black,  thick] (-2,3) -- (0,5);
	\filldraw[black,  thick] (0,5) -- (2,3);
	\filldraw[black,  thick] (2,3) -- (2,0);
	
	\filldraw[black,  thick] (0,0) -- (-1,3);
	\filldraw[black,  thick] (-1,3) -- (0,4);
	\filldraw[black,  thick] (0,4) -- (1,3);
	\filldraw[black,  thick] (1,3) -- (1,0);
	
	\filldraw[black,  thick] (0,0) -- (-0.5,3);
	\filldraw[black,  thick] (-0.5,3) -- (0,3.5);
	\filldraw[black,  thick] (0,3.5) -- (0.5,3);
	\filldraw[black,  thick] (0.5,3) -- (0.5,0);
	
	\filldraw[black] (0.35,.2) circle (0.5pt);
	\filldraw[black] (0.25,.2) circle (0.5pt);
	\filldraw[black] (0.15,.2) circle (0.5pt);
	\filldraw[black] (0.35,1.2) circle (0.5pt);
	\filldraw[black] (0.25,1.2) circle (0.5pt);
	\filldraw[black] (0.15,1.2) circle (0.5pt);
	\filldraw[black] (0.35,2.2) circle (0.5pt);
	\filldraw[black] (0.25,2.2) circle (0.5pt);
	\filldraw[black] (0.15,2.2) circle (0.5pt);
	
\begin{scope}[shift={(-11,0)}]

\begin{scope}[rotate=90]
\begin{scope}[yscale=1]
\draw[thick] (0,0) -- (5,0) -- (5,-.45) -- (2.5,-.25);
\end{scope}\end{scope}
\begin{scope}[rotate=90*.7^1]
\begin{scope}[yscale=.79^1]
\draw[thick] (0,0) -- (5,0) -- (5,-.45) -- (2.5,-.25);
\end{scope}\end{scope}
\begin{scope}[rotate=90*.7^2]
\begin{scope}[yscale=.79^2]
\draw[thick] (0,0) -- (5,0) -- (5,-.45) -- (2.5,-.25);
\end{scope}\end{scope}
\begin{scope}[rotate=90*.7^3]
\begin{scope}[yscale=.79^3]
\draw[thick] (0,0) -- (5,0) -- (5,-.45) -- (2.5,-.25);
\end{scope}\end{scope}
\begin{scope}[rotate=90*.7^4]
\begin{scope}[yscale=.79^4]
\draw[thick] (0,0) -- (5,0) -- (5,-.45) -- (2.5,-.25);
\end{scope}\end{scope}
\begin{scope}[rotate=90*.7^5]
\begin{scope}[yscale=.79^5]
\draw[thick] (0,0) -- (5,0) -- (5,-.45) -- (2.5,-.25);
\end{scope}\end{scope}
\begin{scope}[rotate=90*.7^6]
\begin{scope}[yscale=.79^6]
\draw[thick] (0,0) -- (5,0) -- (5,-.45) -- (2.5,-.25);
\end{scope}\end{scope}
%\begin{scope}[rotate=90*.7^7]
%\begin{scope}[yscale=.79^7]
%\draw[thick] (0,0) -- (2,0) -- (2,-.25) -- (1,-.15);
%\end{scope}\end{scope}
\draw[thick] (-.02,0) -- (5,0);

\node at (2.25,.1) {\tiny $.$};
\node at (2.25,.2) {\tiny $.$};
\node at (2.25,.3) {\tiny $.$};

\node at (4.5,.2) {\tiny $.$};
\node at (4.5,.35) {\tiny $.$};
\node at (4.5,.5) {\tiny $.$};

\node at (0,-.5) {$\ $};
\end{scope}

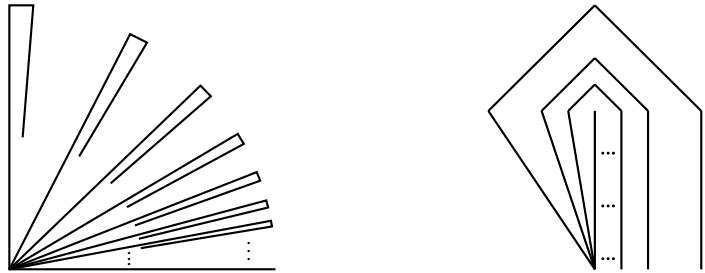
\node at (-4,-1.2) {Two non-smooth fans};

\end{tikzpicture}
\end{center}

We make repeated use of the following fact, a proof of which can be found in \cite[Theorem 9]{C}, \cite[Proposition 4]{CC}, or \cite[Corollary 4]{E}. 

\begin{theorem}\label{thm:CF}
A fan is smooth if and only if it is (homeomorphic to) a subcontinuum of the Cantor fan.
\end{theorem}

%We will make use in several places of Brouwer's characterization of the
%Cantor space: up to homeomorphism, the Cantor space is the unique nonempty zero-dimensional compact metrizable space containing no isolated points 
%(see, e.g., \cite[Theorem 7.14, page 109]{Nadler}).

%%%%%%%%%%%%
%\section{Endpoint-generated fans}
%%%%%%%%%%%%

The following definition first appeared in \cite{Connecting}. 

\begin{definition}
Let $F$ be a fan with top $t$. 
Given $X \sub [0,1]$, we say that $F$ is an \emph{$X$-endpoint generated} fan (or an $X$\emph{-EPG} fan) if for every $B \in \B(F)$, there is a homeomorphism $\phi: I \to B$ such that $\phi(0) = t$ and $\phi[X] = B \cap \closure{E(F) \setminus B}$.
\end{definition} 

For example, $n$-ods are $\0$-EPG fans, the Cantor fan is a $\{1\}$-EPG fan, and the Lelek fan is a $[0,1]$-EPG fan. 
In fact, Charatonik's characterization of the Lelek fan shows it is the only $[0,1]$-EPG smooth fan. 

\begin{theorem}\label{thm:EHimpliesEPG}
If $F$ is an endpoint-homogeneous fan, then there is some $X \sub [0,1]$ such that $F$ is an $X$-EPG fan.
\end{theorem}
\begin{proof}
Let $F$ be a fan with top $t$, fix some $B_0 \in \B(F)$, and let $\phi: I \to B$ be a homeomorphism with $\phi(0) = t$. 
Let $X = \phi^{-1}\big(B_0 \cap \closure{E(F) \setminus B_0}\big)$. 
Given any $B \in \B(F)$, there is a homeomorphism $h: F \to F$ that maps $B_0$ onto $B$. But then $h \circ \phi$ is a homeomorphism $I \to B$ with $h \circ \phi(t) = t$, and, because $h$ fixes $E(F)$, $h \circ \phi[X] = h[B_0 \cap \closure{E(F) \setminus B_0}] = B \cap \closure{E(F) \setminus B}$. Hence $F$ is $X$-EPG.
\end{proof}

This theorem shows that to understand endpoint-homogeneous fans, it is necessary first to understand $X$-EPG fans. 
This raises several questions:
\begin{itemize}
\item[$\circ$] For which $X \sub [0,1]$ is there an $X$-EPG fan?
\item[$\circ$] For which $X \sub [0,1]$ can such a fan be endpoint-homogeneous?
\item[$\circ$] If there is an $X$-EPG fan for some nonempty $X \sub [0,1]$, is it unique up to homeomorphism?
\end{itemize}
As mentioned in the introduction, a significant part of this paper is devoted to fully answering the first two of these questions. 
We record for now two basic observations in this direction. 

\begin{observation}\label{S_F_closed}
If there is an $X$-EPG fan, then $X$ is closed.
\end{observation}

\begin{observation}\label{obs:EE}
Let $X$ and $Y$ be closed subsets of $I$. The following statements are equivalent:
\begin{enumerate}
\item There is an order-preserving homeomorphism $h:I\to I$ such that $h[X] = Y$.
\item Some $X$-EPG fan is a $Y$-EPG fan.
\item Every $X$-EPG fan is a $Y$-EPG fan.
\end{enumerate}
\end{observation}

In light of this observation, we define the following topological invariant.

\begin{definition}
Two subsets of $[0,1]$ are \emph{equivalently embedded} if there is an order-preserving homeomorphism $I \to I$ mapping one onto the other. This is an equivalence relation on subsets of $I$. 

If $X \sub I$ and $F$ is an $X$-EPG fan, we define $EPG(F)$ to be the equivalence class of $X$ under the ``equivalently embedded'' relation.
\end{definition}

Abusing notation, we will write $EPG(F) = X$ when $F$ is an $X$-EPG fan. But it is important to bear in mind that the property of being $X$-EPG does not depend specifically on $X$, but rather on how $X$ is embedded in $[0,1]$. 

Note that homeomorphic subsets of $I$ are not necessarily equivalently embedded. 
For example, $X = \{0\}$, $Y = \{\nicefrac{1}{2}\}$, and $Z = \{1\}$ are all homeomorphic to each other, but not equivalently embedded in $I$. As we shall see, there are $X$-EPG and $Z$-EPG smooth fans, but they look very different (the former have countably many blades and the latter uncountably many, for example); and there are no $Y$-EPG smooth fans at all. 

%We state the main problems of this paper. These can be found originally in \cite{Connecting}.
%
%\begin{problem}
%	Does there exist a fan which is $\left\{\frac{1}{2}\right\}$-end-point-generated?
%\end{problem}
%
%\begin{problem}
%	Characterize all non-empty closed subsets $X$ of $[0,1]$ for which there exists a fan $F$ such that $F$ is $X$-end-point-generated or find a non-trivial family of such sets $X$.
%\end{problem}

We make use in several places of Brouwer's characterization of the
Cantor space: up to homeomorphism, the Cantor space is the unique nonempty zero-dimensional compact metrizable space containing no isolated points 
(see, e.g., \cite[Theorem 7.14, page 109]{Nadler}).

\section{Smooth fans from combs}\label{section_gen}

This section presents a general way of describing smooth fans, which we later use to build endpoint-homogeneous, $X$-EPG smooth fans for every $X$ such that an $X$-EPG smooth fan exists. 
This technique has been used elsewhere, e.g. by Aarts and Oversteegen to study the Lelek fan in \cite{AO}. 
We present it here nonetheless, in order to keep our paper reasonably self-contained. 
Several of the proofs in this section have been omitted on the grounds that they are routine exercises in general topology.

\begin{definition}
Let $\pi_1$ and $\pi_2$ denote the natural coordinate projection maps $I \times I \to I$; that is, $\pi_1(x,y) = x$ and $\pi_2(x,y) = y$ whenever $(x,y) \in I^2$. 

Furthermore, if $C \sub I^2$ and $y \in I$, then $C_y = \set{x}{(x,y) \in C} = \pi_2^{-1}(y) \cap C$.
\end{definition}

\begin{definition}\label{def:comb}
Let $K \sub [0,1]$ denote the Cantor set. A \emph{comb over $K$} (or, in context, simply a \emph{comb}) is a closed subset $C \sub I^2$ such that
\begin{enumerate}
	\item $C_0 = I$,
	\item $C_y$ is a subset of the Cantor set for all $y \in (0,1]$, 
	\item if $y_1<y_2$ then $C_{y_1} \supseteq C_{y_2}$, and
	\item $|C_y| > 2$ for some $y \in (0,1]$. 
\end{enumerate}
Given a comb $C$, let $F_C$ denote the space obtained by collapsing $[0,1] \times \{0\}$ to a point. This space is called the \emph{fan of the comb $C$}.
\end{definition} 

Formally, we take $F_C = \{t\} \cup C \setminus [0,1] \times \{0\}$. In other words, the points of $F_C$ are literally equal to the points of $C$ with the $X$-axis removed, plus one extra point $t$, the top of the fan. The natural quotient map is
$$q_C(z) \,=\, \begin{cases}
z &\text{ if } \pi_2(z) > 0, \\
t &\text{ if } \pi_2(z) = 0.
\end{cases}$$

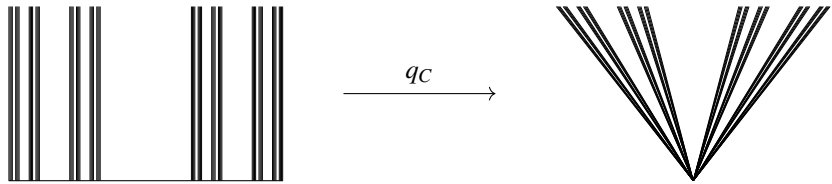
\begin{center}
\begin{tikzpicture}[xscale=.4,yscale=.33]

\node at (4.5,.2) {$\ $};
\node at (4.5,-7.2) {$\ $};

\draw[->] (-7,-3.5) -- (-2,-3.5);
\node at (-4.5,-2.75) {\footnotesize $q_C$};

\draw[very thin] (4.5,-7) -- (0,0);
\draw[ultra thin] (4.5,-7) -- (1/27,0);
\draw[ultra thin] (4.5,-7) -- (2/27,0);
\draw[very thin] (4.5,-7) -- (1/9,0);
\draw[very thin] (4.5,-7) -- (2/9,0);
\draw[ultra thin] (4.5,-7) -- (7/27,0);
\draw[ultra thin] (4.5,-7) -- (8/27,0);
\draw[very thin] (4.5,-7) -- (1/3,0);
\draw[very thin] (4.5,-7) -- (2/3,0);
\draw[ultra thin] (4.5,-7) -- (19/27,0);
\draw[ultra thin] (4.5,-7) -- (20/27,0);
\draw[very thin] (4.5,-7) -- (7/9,0);
\draw[very thin] (4.5,-7) -- (8/9,0);
\draw[ultra thin] (4.5,-7) -- (25/27,0);
\draw[ultra thin] (4.5,-7) -- (26/27,0);
\draw[very thin] (4.5,-7) -- (1,0);
\draw[very thin] (4.5,-7) -- (2,0);
\draw[ultra thin] (4.5,-7) -- (55/27,0);
\draw[ultra thin] (4.5,-7) -- (56/27,0);
\draw[very thin] (4.5,-7) -- (19/9,0);
\draw[very thin] (4.5,-7) -- (20/9,0);
\draw[ultra thin] (4.5,-7) -- (61/27,0);
\draw[ultra thin] (4.5,-7) -- (62/27,0);
\draw[very thin] (4.5,-7) -- (7/3,0);
\draw[very thin] (4.5,-7) -- (8/3,0);
\draw[ultra thin] (4.5,-7) -- (73/27,0);
\draw[ultra thin] (4.5,-7) -- (74/27,0);
\draw[very thin] (4.5,-7) -- (25/9,0);
\draw[very thin] (4.5,-7) -- (26/9,0);
\draw[ultra thin] (4.5,-7) -- (79/27,0);
\draw[ultra thin] (4.5,-7) -- (80/27,0);
\draw[very thin] (4.5,-7) -- (3,0);
\draw[very thin] (4.5,-7) -- (6,0);
\draw[ultra thin] (4.5,-7) -- (163/27,0);
\draw[ultra thin] (4.5,-7) -- (164/27,0);
\draw[very thin] (4.5,-7) -- (55/9,0);
\draw[very thin] (4.5,-7) -- (56/9,0);
\draw[ultra thin] (4.5,-7) -- (169/27,0);
\draw[ultra thin] (4.5,-7) -- (170/27,0);
\draw[very thin] (4.5,-7) -- (19/3,0);
\draw[very thin] (4.5,-7) -- (20/3,0);
\draw[ultra thin] (4.5,-7) -- (181/27,0);
\draw[ultra thin] (4.5,-7) -- (182/27,0);
\draw[very thin] (4.5,-7) -- (61/9,0);
\draw[very thin] (4.5,-7) -- (62/9,0);
\draw[ultra thin] (4.5,-7) -- (187/27,0);
\draw[ultra thin] (4.5,-7) -- (188/27,0);
\draw[very thin] (4.5,-7) -- (7,0);
\draw[very thin] (4.5,-7) -- (8,0);
\draw[ultra thin] (4.5,-7) -- (217/27,0);
\draw[ultra thin] (4.5,-7) -- (218/27,0);
\draw[very thin] (4.5,-7) -- (73/9,0);
\draw[very thin] (4.5,-7) -- (74/9,0);
\draw[ultra thin] (4.5,-7) -- (223/27,0);
\draw[ultra thin] (4.5,-7) -- (224/27,0);
\draw[very thin] (4.5,-7) -- (25/3,0);
\draw[very thin] (4.5,-7) -- (26/3,0);
\draw[ultra thin] (4.5,-7) -- (235/27,0);
\draw[ultra thin] (4.5,-7) -- (236/27,0);
\draw[very thin] (4.5,-7) -- (79/9,0);
\draw[very thin] (4.5,-7) -- (80/9,0);
\draw[ultra thin] (4.5,-7) -- (241/27,0);
\draw[ultra thin] (4.5,-7) -- (242/27,0);
\draw[very thin] (4.5,-7) -- (9,0);

\begin{scope}[shift={(-18,0)}]
\draw[very thin] (0,-7) -- (0,0);
\draw[ultra thin] (1/27,-7) -- (1/27,0);
\draw[ultra thin] (2/27,-7) -- (2/27,0);
\draw[very thin] (1/9,-7) -- (1/9,0);
\draw[very thin] (2/9,-7) -- (2/9,0);
\draw[ultra thin] (7/27,-7) -- (7/27,0);
\draw[ultra thin] (8/27,-7) -- (8/27,0);
\draw[very thin] (1/3,-7) -- (1/3,0);
\draw[very thin] (2/3,-7) -- (2/3,0);
\draw[ultra thin] (19/27,-7) -- (19/27,0);
\draw[ultra thin] (20/27,-7) -- (20/27,0);
\draw[very thin] (7/9,-7) -- (7/9,0);
\draw[very thin] (8/9,-7) -- (8/9,0);
\draw[ultra thin] (25/27,-7) -- (25/27,0);
\draw[ultra thin] (26/27,-7) -- (26/27,0);
\draw[very thin] (1,-7) -- (1,0);
\draw[very thin] (2,-7) -- (2,0);
\draw[ultra thin] (55/27,-7) -- (55/27,0);
\draw[ultra thin] (56/27,-7) -- (56/27,0);
\draw[very thin] (19/9,-7) -- (19/9,0);
\draw[very thin] (20/9,-7) -- (20/9,0);
\draw[ultra thin] (61/27,-7) -- (61/27,0);
\draw[ultra thin] (62/27,-7) -- (62/27,0);
\draw[very thin] (7/3,-7) -- (7/3,0);
\draw[very thin] (8/3,-7) -- (8/3,0);
\draw[ultra thin] (73/27,-7) -- (73/27,0);
\draw[ultra thin] (74/27,-7) -- (74/27,0);
\draw[very thin] (25/9,-7) -- (25/9,0);
\draw[very thin] (26/9,-7) -- (26/9,0);
\draw[ultra thin] (79/27,-7) -- (79/27,0);
\draw[ultra thin] (80/27,-7) -- (80/27,0);
\draw[very thin] (3,-7) -- (3,0);
\draw[very thin] (6,-7) -- (6,0);
\draw[ultra thin] (163/27,-7) -- (163/27,0);
\draw[ultra thin] (164/27,-7) -- (164/27,0);
\draw[very thin] (55/9,-7) -- (55/9,0);
\draw[very thin] (56/9,-7) -- (56/9,0);
\draw[ultra thin] (169/27,-7) -- (169/27,0);
\draw[ultra thin] (170/27,-7) -- (170/27,0);
\draw[very thin] (19/3,-7) -- (19/3,0);
\draw[very thin] (20/3,-7) -- (20/3,0);
\draw[ultra thin] (181/27,-7) -- (181/27,0);
\draw[ultra thin] (182/27,-7) -- (182/27,0);
\draw[very thin] (61/9,-7) -- (61/9,0);
\draw[very thin] (62/9,-7) -- (62/9,0);
\draw[ultra thin] (187/27,-7) -- (187/27,0);
\draw[ultra thin] (188/27,-7) -- (188/27,0);
\draw[very thin] (7,-7) -- (7,0);
\draw[very thin] (8,-7) -- (8,0);
\draw[ultra thin] (217/27,-7) -- (217/27,0);
\draw[ultra thin] (218/27,-7) -- (218/27,0);
\draw[very thin] (73/9,-7) -- (73/9,0);
\draw[very thin] (74/9,-7) -- (74/9,0);
\draw[ultra thin] (223/27,-7) -- (223/27,0);
\draw[ultra thin] (224/27,-7) -- (224/27,0);
\draw[very thin] (25/3,-7) -- (25/3,0);
\draw[very thin] (26/3,-7) -- (26/3,0);
\draw[ultra thin] (235/27,-7) -- (235/27,0);
\draw[ultra thin] (236/27,-7) -- (236/27,0);
\draw[very thin] (79/9,-7) -- (79/9,0);
\draw[very thin] (80/9,-7) -- (80/9,0);
\draw[ultra thin] (241/27,-7) -- (241/27,0);
\draw[ultra thin] (242/27,-7) -- (242/27,0);
\draw[very thin] (9,-7) -- (9,0);

\draw (-.008,-7) -- (9.008,-7);
\end{scope}

\end{tikzpicture}
\end{center}

\begin{lemma}\label{lem:Ck}
There is a largest comb, namely $C_K = [0,1] \times \{0\} \cup K \times (0,1]$, and the fan of this comb is (homeomorphic to) the Cantor fan. 
\end{lemma}
\begin{proof}
It is clear from the definition above that $C_K$ is a comb, and that every other comb is a subset of $C_K$. Furthermore, the fan of $C_K$ is the cone over $K$, i.e., the Cantor fan.
\end{proof}

\begin{theorem}\label{thm:FC}
If $C$ is a comb, then $F_C$ is a smooth fan. Conversely, if $F$ is a smooth fan then there is a comb $C \sub I^2$ such that $F_C \homeo F$.
\end{theorem}
\begin{proof}
To prove the first assertion of the theorem, let $C$ be a comb. Then $C$ is a closed subset of $C_K$ (the maximal comb described in the previous lemma), and this implies that $F_C$ is a closed subset of the Cantor fan. Thus if $F_C$ is a fan, then it must be a smooth fan.

To check that $F_C$ is a fan, first let us check that it is a continuum. As a closed subspace of the Cantor fan, it is clearly compact and metrizable. 
We claim $F_C$ is also path-connected. Let $z \in F_C$ with $z \neq t$. Note that $z = (x,y) \in C$ for some $y \in (0,1]$. By condition $(3)$ in the definition of a comb, $\{x\} \times [0,y] \sub C$. In $F_C$, this set is an arc with endpoints $t$ and $z$. This means $F_C$ is path-connected, because if $z,z' \in F_C$, then there is an arc connecting $z$ to $t$ and an arc connecting $t$ to $z'$. 

Observe that $F_C$ is hereditarily unicoherent, because it is a closed subset of the Cantor fan, which is hereditarily unicoherent. Together with the previous paragraph, this means $F_C$ is a dendroid.

By requirement $(4)$ in the definition of a comb, there is some $y > 0$ with $|C_y| > 2$. If $z_1$, $z_2$, and $z_3$ are distinct points in $C_y$, then 
$$[0,1] \times \{0\} \cup \{z_1\} \times (0,1] \cup \{z_2\} \times (0,1] \cup \{z_3\} \times (0,1]$$
collapses to a triod in $F_C$ with top $t$. Hence $t$ is a ramification point of $F_C$. There are no other ramification points of $F_C$, because if $t'$ were a ramification point of $F_C$ (i.e., it is the top of a triod $T \sub F_C$), then it would be a ramification point of the Cantor fan containing $F_C$; but the Cantor fan has only one ramification point.

Thus $F_C$ is a fan, and as we noted earlier, this means it is a smooth fan. 

For the second assertion of the theorem, suppose $F$ is a smooth fan. By Theorem~\ref{thm:CF}, $F$ is a subcontinuum of the Cantor fan, which we may (and do) view as the fan of the comb $C_K$ described in the previous lemma. Let $C = (F \cap C_K \setminus \{t\}) \cup [0,1] \times \{0\}$ (where $t$ denotes the top of $C_K$). It is not difficult to check that $C$ is a comb over $K$, and the fan of this comb is $F$.
\end{proof}

Observe that every comb is in fact a dendroid. In particular, we may (and do) use the notation $[z,z']$ to denote the unique arc in a comb $C$ containing $z$ and $z'$ as endpoints.

\begin{definition}
Let $C$ be a comb, and suppose that $(x,y) \in C$ for some $y > 0$, but $(x,y') \notin C$ for any $y' > y$. Then $[(x,0),(x,y)]$ is called a \emph{blade} of $C$, specifically the \emph{blade of $C$ at $x$}, and $(x,y)$ is called the \emph{tip} of the blade. If $C$ has a blade at some $x \in K$, then the blade at $x$ is denoted $B_x$.

The set of all blades of $C$ is denoted $\B(C)$, and the set of all points that are the tip of some blade is denoted $\mathcal E(C)$. 
\end{definition}

Observe that $|\B(C)| \geq 3$, for any comb $C$, by part $(4)$ of the definition of a comb. 
Of course combs have teeth, not blades, but the following lemma shows why we have chosen this terminology.

\begin{lemma}\label{setA_endpoints_gen}
Let $C \subseteq I^2$ be a comb. Then
\begin{enumerate}
\item If $B \in \B(C)$ then $q_C[B] \in \B(F_C)$, and if $B' \in \B(F_C)$ then there is some $B \in \B(C)$ with $q_C[B] = B'$.
\item $\mathcal E(C) = E(F_C)$.
\end{enumerate}
\end{lemma}
%\begin{proof}
%If $y'\in(y,1]$ and $(x,y')\in C$, then $\{x\}\times [0,y']$ is an arc in $C$ that contains $(x,y)$, but not as an endpoint, so $(x,y) \notin E(C)$. 
%On the other hand, suppose $(x,y')\notin C$ whenever $y' \in (y,1]$. 
%Recall that $C \sub C_K$ by Lemma~\ref{lem:Ck}. The connected component of $(x,y)$ in $C_K \setminus [0,1] \times \{0\}$ is $\{x\} \times (0,1]$, therefore the connected component of $(x,y)$ in $C$ is $\{x\} \times (0,y]$. It follows that any arc in $C$ containing $(x,y)$ contains it as an endpoint, so $(x,y) \in E(C)$.
%
%The ``furthermore'' part of the lemma follows from the first part and the definition of $F_C$.
%\end{proof}

In other words, blades of $C$ correspond naturally to blades of $F_C$, and the tips of those blades correspond naturally (in fact, are equal to) the endpoints of $F_C$. 
Note that $\mathcal E(C)$ and $E(C)$ may not exactly be equal, but they are nearly so: $\mathcal E(C) \sub E(C) \sub \mathcal E(C) \cup \{(0,0),(1,0)\}$. In other words, it is possible that $(0,0)$ and $(1,0)$ are endpoints of $C$, but these two points are the only possible difference between $E(C)$ and $\mathcal E(C)$. Whether these points are in $E(C)$ depends on whether $C$ has blades at $0$ or $1$. 

\begin{observation}\label{obs:Blades}
$\pi_1[\mathcal E(C)] \,=\, \set{x \in I}{C \text{ has a blade at }x}$.
\end{observation}

We close this section with a lemma roughly stating that in order for $F_C$ to be an $X$-EPG fan (for some given $X \sub [0,1]$), the comb $C$ must have an analogous property. 

\begin{lemma}\label{setA_fancyA_EPG}
	Let $C \subseteq I^2$ be a comb, and let $X \sub I$.
	\begin{enumerate}
		\item If $0 \notin X$, then $F_C$ is an $X$-EPG fan if and only if $\closure{E(C)} \cap I \times \{0\} = \0$, and for each $B \in \B(C)$ there is a homeomorphism $\phi_B: I \to B$ such that $\phi$ maps $1$ to the tip of $B$, and $\phi[X] = B \cap \closure{E(C) \setminus B}$.
		\item If $0 \in X$, then $F_C$ is an $X$-EPG fan if and only if $\closure{E(C)} \cap I \times \{0\} \neq \0$, and for each $B \in \B(C)$ there is a homeomorphism $\phi_B: I \to B$ such that $\phi$ maps $1$ to the tip of $B$, and $\phi[X \setminus \{0\}] = B \cap \closure{E(C) \setminus B} \setminus (I \times \{0\})$.
	\end{enumerate}
\end{lemma}

%%%%%%%%%%%%
\section{Sets that do not work}
%%%%%%%%%%%%

In this section and the next, we exactly characterize the sets $X \sub [0,1]$ such that there is an $X$-EPG smooth fan. As mentioned in the introduction, the complete list is:
\begin{itemize}
\item[$\circ$] $X = \0$, $\{1\}$, or $[0,1]$, or
\item[$\circ$] $X$ is a closed subset of $[0,1]$ with $0 \in X$ and $1 \notin X$, or
\item[$\circ$] $X$ is a closed subset of $[0,1]$ with $0,1 \in X$, and $1$ is an isolated point of $X$ (i.e., a set of the second type plus the point $1$).
\end{itemize} 
In this section we prove that there is no $X$-EPG fan for any $X \sub [0,1]$ not appearing on this list. In the next section we construct examples of $X$-EPG fans for every $X$ that does appear on this list. 

The following two lemmas rely on two well-known facts from general topology: $(1)$ any relatively discrete subset of a compact metric space is countable, and $(2)$ any nonempty, closed subset of the Cantor set $K$ with no isolated points is homeomorphic to $K$.

\begin{theorem}\label{thm:Countable}
Let $X \sub [0,1]$, and suppose that $F$ is an $X$-EPG fan. \linebreak If $1 \notin X$, then $E(F)$ is countable.
\end{theorem}
\begin{proof}
Note that $F$ is a compact metric space. Because $F$ is $X$-EPG and $1\notin X$, no member of $E(F)$ is a limit point of $E(F)$; i.e., $E(F)$ is a relatively discrete subset of $F$. Hence $E(F)$ is countable.
\end{proof}

\begin{lemma}\label{lem:HomeoCantor}
Let $C$ be a comb, and suppose $F_C$ is an $X$-EPG smooth fan for some $X \sub [0,1]$. If $0 \notin X$ and $X \neq \0$, then $\pi_1[\mathcal E(C)]$ is homeomorphic to the Cantor space.
\end{lemma}
\begin{proof}
Observe that every blade of $C$ contains exactly one member of $\mathcal E(C)$. 
Consequently, $\pi_1$ restricts to an injection on $\mathcal E(C)$. 
We claim that $\pi_1[\mathcal E(C)]$ is closed and has no isolated points. 

To see that it is closed, fix a sequence $\seq{x_n}{n \in \N}$ of points in $\pi_1[\mathcal E(C)]$. For each $n$, there is some $y_n$ such that $(x_n,y_n) \in \mathcal E(C)$. Because $C$ is compact, $\seq{(x_n,y_n)}{n \in \N}$ has a subsequence converging to some $(x,y) \in C$. 
Because $0 \notin X$, Lemma~\ref{setA_fancyA_EPG} implies $y > 0$. But then $C$ has a blade at $x$, and the tip of that blade is in $\mathcal E(C)$, hence $x \in \pi_1[\mathcal E(C)]$. 
Thus we have a subsequence of $\seq{x_n}{n \in \N}$ converging to $x \in \pi_1[\mathcal E(C)]$. 
This shows that every sequence in $\pi_1[\mathcal E(C)]$ has a convergent subsequence, which means $\pi_1[\mathcal E(C)]$ is closed.

To see that $\pi_1[\mathcal E(C)]$ has no isolated points, fix some $x \in \pi_1[\mathcal E(C)]$, and fix $y > 0$ with $(x,y) \in \mathcal E(C)$. 
Because $X \neq \0$, there is (by Lemma~\ref{setA_fancyA_EPG}) some sequence of distinct points $\seq{(x_n,y_n)}{n \in \N}$ in $\mathcal E(C)$ converging to $(x,y')$ for some $y' \leq y$. Because $\pi_1$ is continuous and (as noted above) injective on $\mathcal E(C)$, this means $\seq{x_n}{n \in \N}$ is a sequence of distinct points in $\pi_1[\mathcal E(C)]$ converging to $\pi_1((x,y')) = x$. Hence $x$ is not isolated in $\pi_1[\mathcal E(C)]$. 

Thus $\pi_1[\mathcal E(C)]$ is a closed, nonempty subset of $K$ (the Cantor set) with no isolated points. By Brouwer's characterization of the Cantor space, this implies $\pi_1[\mathcal E(C)]$ is homeomorphic to the Cantor space.
\end{proof}

\begin{theorem}\label{thm:Uncountable}
Let $X \sub [0,1]$, $X \neq \0$, and suppose $F$ is an $X$-EPG smooth fan. If $0 \notin X$, then $E(F)$ is uncountable.
\end{theorem}
\begin{proof}
Using Theorem \ref{thm:FC}, let $C$ be a comb such that $F \homeo F_C$. By the previous lemma, $\pi_1[\mathcal E(C)]$ is homeomorphic to the Cantor space, which means $\mathcal E(C)$ is uncountable. But $E(F_C) = \mathcal E(C)$, so $E(F)$ is uncountable.
\end{proof}

\begin{corollary}\label{cor:4kinds}
If $X \sub [0,1]$, and there is an $X$-EPG smooth fan, then either $X = \0$, or else $0 \in X$ or $1 \in X$ (or both).
\end{corollary}
\begin{proof}
By the preceding lemmas, if $F$ were an $X$-EPG smooth fan for any nonempty $X \not\ni 0,1$, then $E(F)$ would be both countable and uncountable.
\end{proof}

This theorem answers Problem 3.41 from \cite{Connecting}, which asked whether a smooth fan can be $\{\nicefrac{1}{2}\}$-endpoint-generated. 

Note that there are examples of $X$-EPG fans for sets $X$ satisfying each of the four possibilities from Corollary~\ref{cor:4kinds}: 
\begin{enumerate}
	\item $X = \0$: The simple $n$-od, for any $n \geq 3$, is $\0$-EPG.
	\item $0,1 \in X$: The Lelek fan is $[0,1]$-EPG.
	\item $1 \in X$: The Cantor fan is $\{1\}$-EPG.
	\item $0 \in X$: The star (i.e., the one point compactification of countably many copies of $(0,1]$, drawn below) is $\{0\}$-EPG.
\end{enumerate}

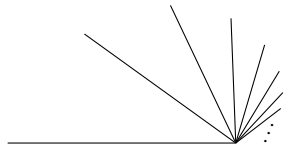
\vspace{1mm}

\begin{center}
\begin{tikzpicture}[scale=1.5]

\draw (0,0) -- (-2,0);
\begin{scope}[rotate=144]
\draw (0,0) -- (1.64,0);
\end{scope}
\begin{scope}[rotate=115.2]
\draw (0,0) -- (1.345,0);
\end{scope}
\begin{scope}[rotate=92.16]
\draw (0,0) -- (1.103,0);
\end{scope}
\begin{scope}[rotate=73.728]
\draw (0,0) -- (.904,0);
\end{scope}
\begin{scope}[rotate=58.98]
\draw (0,0) -- (.741,0);
\end{scope}
\begin{scope}[rotate=47.186]
\draw (0,0) -- (.608,0);
\end{scope}
\begin{scope}[rotate=37.749]
\draw (0,0) -- (.499,0);
\end{scope}

\node at (.26,.02) {\tiny $.$};
\node at (.29,.09) {\tiny $.$};
\node at (.32,.16) {\tiny $.$};

\end{tikzpicture}
\end{center}

\vspace{1mm}

%We show next that the Cantor fan is in fact the only example of an $X$-EPG smooth fan with $0 \notin X$ and $1 \in X$. 

\begin{theorem}\label{thm:BCT1}
Suppose $X \sub [0,1]$ and there is an $X$-EPG smooth fan. If $0 \notin X$ and $1 \in X$, then $X = \{1\}$.
\end{theorem}
\begin{proof}
Aiming for a contradiction, suppose that $F$ is an $X$-EPG smooth fan for some $X$ with $0 \notin X$, $1 \in X$, and $X \neq \{1\}$. 
Applying Theorem~\ref{thm:FC}, let $C$ be a comb such that $F_C \homeo F$.

For each $x$ such that $C$ has a blade at $x$, define 
\begin{align*}
e_x &\,=\, \max \set{y}{(x,y) \in B_x \cap \closure{\mathcal E(C) \setminus B_x}} \\
m_x &\,=\, \min \set{y}{(x,y) \in B_x \cap \closure{\mathcal E(C) \setminus B_x}}.
\end{align*}
Because $1 \in X$, the highest point at which $\closure{\mathcal E(C) \setminus B}$ hits a blade $B$ is the tip of $B$. In other words, $e_x$ is the $Y$-coordinate of the tip of $B_x$, and $(x,e_x) \in \mathcal E(C)$ (hence the choice of the letter $e$). 

For each $n \in \N$, let 
$$L_{\nicefrac{1}{n}} \,=\, \set{x}{e_x-m_x \geq \nicefrac{1}{n}}.$$ 
Lemma~\ref{setA_fancyA_EPG} implies that for each $B \in \B(C)$, $B \cap \closure{\mathcal E(C) \setminus B}$ contains at least two points: the tip of $B$ (because $1 \in X$) and at least one other point of $B$ (because $X \neq \{1\}$). Thus $e_x-m_x > 0$ whenever $C$ has a blade at $x$. Therefore 
$$\textstyle \bigcup_{n \in \N}L_{\nicefrac{1}{n}} \,=\, \set{x}{C \text{ has a blade at }x} \,=\, \pi_1[\mathcal E(C)].$$
Because $0 \notin X$ and $X \neq \0$, Lemma~\ref{lem:HomeoCantor} asserts that $\pi_1[\mathcal E(C)]$ is homeomorphic to the Cantor space. 

We claim the sets $L_{\nicefrac{1}{n}}$ are all closed. To see this, fix $n \in \N$ and suppose $\seq{x_k}{k \in \N}$ is a sequence of points in $L_{\nicefrac{1}{n}}$ converging to some point $x$; we wish to show $x \in L_{\nicefrac{1}{n}}$. Because $\seq{(x_k,e_{x_k})}{k \in \N}$ is a sequence of points in the compact set $\closure{\mathcal E(C)}$, some subsequence $\seq{(x_{k_i},e_{x_{k_i}})}{i \in \N}$ of this sequence converges to a point $(x,e) \in \closure{\mathcal E(C)}$. (Note that this point must have first coordinate $x$, because $x = \lim_{k \to \infty}x_k$.) 
Again, $\seq{(x_{k_i},m_{x_{k_i}})}{i \in \N}$ is a sequence of points in $\closure{\mathcal E(C)}$, so some subsequence of this sequence converges to a point $(x,m) \in \closure{\mathcal E(C)}$. 
Because $e_{x_k}-m_{x_k} \geq \nicefrac{1}{n}$ for all $k$, we have $e-m \geq \nicefrac{1}{n}$ as well. 
Because $(x,m),(x,e) \in \closure{\mathcal E(C)}$, this implies $x \in L_{\nicefrac{1}{n}}$ as claimed.

Because $\pi_1[\mathcal E(C)]$ is homeomorphic to the Cantor space, the Baire Category Theorem asserts that it is not a countable union of closed sets nowhere dense in $\pi_1[\mathcal E(C)]$. 
But $\pi_1[\mathcal E(C)] = \bigcup_{n \in \N}L_{\nicefrac{1}{n}}$, and each $L_{\nicefrac{1}{n}}$ is closed by the previous paragraph. 
It follows that there is some $n \in \N$ such that $L_{\nicefrac{1}{n}}$ is not nowhere dense in $\pi_1[\mathcal E(C)]$. Because $L_{\nicefrac{1}{n}}$ is closed in $\pi_1[\mathcal E(C)]$, this means $L_{\nicefrac{1}{n}}$ must have nonempty interior in $\pi_1[\mathcal E(C)]$.

Fix some nonempty, relatively open $U \sub \pi_1[\mathcal E(C)]$ such that $U \subset L_{\nicefrac{1}{n}}$. 
Let $m = \inf \set{m_x}{x \in U}$, and fix some particular $x \in U$ such that $m_x < m+\nicefrac{1}{3n}$. 
On the one hand, because $(x,m_x) \in \closure{\mathcal E(C)}$, there is some $x' \in U$ such that $(x',e_{x'})$ is within $\nicefrac{1}{3n}$ of $(x,m_x)$; in particular, $e_{x'} \leq m_x + \nicefrac{1}{3n} < m + \nicefrac{2}{3n}$. 
But $x' \in U \sub L_{\nicefrac{1}{n}}$, so this implies 
$m_{x'} \leq e_{x'}-\nicefrac{1}{n} < m - \nicefrac{1}{3n},$
which contradicts the definition of $m$.
\end{proof}

\begin{corollary}\label{cor:Cantor}
Up to homeomorphism, the Cantor fan is the only $\{1\}$-EPG smooth fan. 
Furthermore, the Cantor fan is the only $X$-EPG smooth fan for a set $X$ with $0 \notin X$ and $1 \in X$.
\end{corollary}
\begin{proof}
If $F$ is a $\{1\}$-EPG smooth fan, then $E(F)$ is a closed subset of $F$ with no isolated points, and containing no arcs. By Brouwer's characterization of the Cantor space, this implies $E(F)$ is homeomorphic to the Cantor space. 
By \cite[Theorem 2.2]{HGH}, the Cantor fan is, up to homeomorphism, the unique smooth fan in which $E(F)$ is homeomorphic to the Cantor space. 
This proves the first statement of the corollary, and the second statement follows by combining this with the previous theorem.
\end{proof}

The proof of the following theorem is similar to the proof of Theorem~\ref{thm:BCT1}, insofar as it employs the Baire Category Theorem in a similar way.

\begin{theorem}
Suppose $X \sub [0,1]$ with $1 \in X$, and there is an $X$-EPG smooth fan. Then either $1$ is an isolated point of $X$, or $X = [0,1]$.
\end{theorem}
\begin{proof}
Aiming for a contradiction, suppose that $F$ is an $X$-EPG smooth fan for some $X \sub [0,1]$ such that $1 \in X$, but $X \neq [0,1]$ and $1$ is not an isolated point of $X$. 
By the previous theorem, if $0 \notin X$ then $X = \{1\}$ and we are done. Thus we may and do assume that $0 \in X$. 
Applying Theorem~\ref{thm:FC}, let $C$ be a comb such that $F_C \homeo F$.

Fix some $x_0$ such that $C$ has a blade at $x_0$. By Lemma~\ref{setA_fancyA_EPG}, there is a homeomorphism $h: I \to B_{x_0}$ mapping $X \setminus \{0\}$ onto $(B_{x_0} \setminus \{x_0\}) \cap \closure{\mathcal E(C) \setminus B_{x_0}}$. 
In particular, because $0,1 \in X$ but $X \neq [0,1]$, there is a point of $(0,1)$ not in $X$, and the image of this point under $h$ is a point of $B_{x_0}$ not in $\closure{\mathcal E(C)}$. Fix $y_0 > 0$ such that $(x_0,y_0) \in B_{x_0} \setminus \closure{\mathcal E(C)}$. 

Recall that $C$ is a comb over the Cantor set $K$, so $C \sub I \times \{0\} \cup K \times I$. 
Because $(x_0,y_0) \notin \closure{\mathcal E(C)}$, there is some basic clopen $V \sub K$ with $x_0 \in V$ and there is some $\e > 0$ such that 
$$\big( V \times (y_0-\e,y_0+\e) \big) \cap \closure{\mathcal E(C)} \,=\, \0.$$

\begin{center}
\begin{tikzpicture}[xscale=1.2]
	\draw[black!30, dashed] (-0.5,1) rectangle (0.5,2);
	\filldraw[black!30, dashed, opacity=0.5] (-0.5,1) rectangle (0.5,2);

	\draw[thick] (-2,0) -- (2,0);
	\draw[thick] (0,0) -- (0,3.5);
	\filldraw (0,0) circle (1pt) node [anchor=north]{$x_0$};
	\filldraw (0,1.5) circle (1pt) node [anchor=west]{\small $(x_0,y_0)$};	
\end{tikzpicture}
\end{center}

Because $(x_0,y_0) \in B_{x_0}$, the tip of $B_{x_0}$ has $Y$-coordinate $\geq y_0$. But the tip of $B_{x_0}$ is in $\mathcal E(C)$, so this means its $Y$-coordinate is actually $\geq y_0+\e$; i.e., the tip of $B_{x_0}$ sits above this rectangle. 
Let $D$ denote the set of all $x$ such that there is a blade at $x$ with its tip sitting above this rectangle: i.e., 
$$D \,=\, \set{x \in V}{ \text{there is some $y \geq y_0+\e$ such that }(x,y) \in C }.$$
If $(x,y_0) \in C$ for some $x$, then the tip of $B_x$ has $Y$-coordinate $\geq\! y_0+\e$, because $\big( V \times (y_0-\e,y_0+\e) \big) \cap \closure{\mathcal E(C)} = \0$. Thus
$$D \,=\, \set{x \in V}{(x,y_0) \in C} \,=\, C_{y_0} \cap V.$$
In particular, this makes it clear that $D$ is a compact subset of $K$. 
Let 
$$E \,=\, \set{(x,y)\in \mathcal E(C) }{ x \in D } \,=\, \mathcal E(C) \cap \big(V \times [y_0,1]\big).$$
Because $1 \in X$, every member of $E$ is a limit of points in $\mathcal E(C)$. 
Because $V$ is clopen and $\big(V \times (y_0-\e,y_0+\e)\big) \cap \mathcal E(C) = \0$, this implies $E$ has no isolated points. 
Because $D = \pi_1[E]$ and $\pi_1$ restricts to an injection on $E$, this implies $D$ has no isolated points. 
Thus $D$ is a closed subset of the Cantor set $K$ with no isolated points. Hence $D$ is homeomorphic to the Cantor space.

For each $x\in D$, define
\begin{align*}
e_x &\,=\, \max \set{y}{(x,y) \in B_x \cap \closure{\mathcal E(C) \setminus B_x}} \\
m_x &\,=\, \min \set{y \geq y_0}{(x,y) \in B_x \cap \closure{\mathcal E(C) \setminus B_x}}.
\end{align*}
Because $1 \in X$, $(x,e_x)$ is the tip of $B_x$. 
Because $1$ is not isolated in $X$, Lemma~\ref{setA_fancyA_EPG} says that for each $x \in D$, $(x,e_x)$ is not isolated in $B_x \cap \closure{\mathcal E(C) \setminus B_x}$. If $x \in D$, then $e_x > y_0$ (by the definition of $D$), and this implies that the set $\set{y \geq y_0}{(x,y) \in B_x \cap \closure{\mathcal E(C) \setminus B_x}}$ contains points other than $e_x$. Hence $m_x \neq e_x$, or equivalently $e_x-m_x > 0$, for all $x \in D$. 

Following the proof of the previous theorem, define 
$$L_{\nicefrac{1}{n}} \,=\, \set{x}{e_x-m_x \geq \nicefrac{1}{n}}$$
for each $n \in \N$.  
By the previous paragraph, $D = \bigcup_{n \in \N}L_{\nicefrac{1}{n}}$.

We claim that each of the sets $L_{\nicefrac{1}{n}}$ is closed. To see this, fix $n \in \N$ and suppose $\seq{x_k}{k \in \N}$ is a sequence of points in $L_{\nicefrac{1}{n}}$ converging to some point $x$; we wish to show $x \in L_{\nicefrac{1}{n}}$. Because $\seq{(x_k,e_{x_k})}{k \in \N}$ is a sequence of points in the compact set $\big( V \times [y_0,1] \big) \cap \closure{\mathcal E(C)}$, some subsequence $\seq{(x_{k_i},e_{x_{k_i}})}{i \in \N}$ of this sequence converges to a point $(x,e) \in \big( V \times [y_0,1] \big) \cap \closure{\mathcal E(C)}$. (Note that this point must have first coordinate $x$, because $x = \lim_{k \to \infty}x_k$.) 
Again, $\seq{(x_{k_i},m_{x_{k_i}})}{i \in \N}$ is a sequence of points in $\big(V \times [y_0,1]\big) \cap \closure{\mathcal E(C)}$, so a subsequence of this sequence converges to a point $(x,m) \in (V \times [y_0,1]) \cap \closure{\mathcal E(C)}$. 
Because $e_{x_k}-m_{x_k} \geq \nicefrac{1}{n}$ for all $k$, we have $e-m \geq \nicefrac{1}{n}$ as well. 
Because $(x,m),(x,e) \in \big(V \times [y_0,1]\big) \cap \closure{\mathcal E(C)}$, this implies $x \in L_{\nicefrac{1}{n}}$ as claimed.

Because $D$ is homeomorphic to the Cantor space, the Baire Category Theorem asserts that it is not a countable union of closed sets nowhere dense in $D$. 
It follows that there is some $n \in \N$ such that $L_{\nicefrac{1}{n}}$ is not nowhere dense in $D$, and therefore (since it is closed) must have nonempty interior in $D$.

Fix a nonempty open $U \sub D$ such that $U \subset L_{\nicefrac{1}{n}}$. 
Let $m = \inf \set{m_x}{x \in U}$, and fix some particular $x \in U$ such that $m_x < m+\nicefrac{1}{3n}$. 
On the one hand, because $(x,m_x) \in \closure{\mathcal E(C)}$, there is some $x' \in U$ such that $(x',e_{x'})$ is within $\nicefrac{1}{3n}$ of $(x,m_x)$; in particular, $e_{x'} \leq m_x + \nicefrac{1}{3n} < m + \nicefrac{2}{3n}$. 
But $x' \in U \sub L_{\nicefrac{1}{n}}$, so this implies 
$m_{x'} \leq e_{x'}-\nicefrac{1}{n} < m - \nicefrac{1}{3n},$
contradicting our choice of $m$.
\end{proof}

The results of this section are summarized as follows:

\begin{corollary}\label{cor:list0}
Suppose $X \sub [0,1]$ and there is an $X$-EPG fan. Then one of the following must hold:
\begin{enumerate}
\item $X = \0$, $\{1\}$, or $[0,1]$, or
\item $X$ is a closed subset of $[0,1]$ with $0 \in X$ and $1 \notin X$, or
\item $X$ is a closed subset of $[0,1]$ with $0,1 \in X$, and $1$ is an isolated point of $X$.
\end{enumerate} 
\end{corollary}

%%%%%%%%%%%%
\section{Constructing endpoint-homogeneous fans}
%%%%%%%%%%%%

In the previous section, we proved that for certain $X \sub [0,1]$ there can be no $X$-EPG smooth fan. In this section, we show that for all $X \sub [0,1]$ not ruled out by the results of the previous section, there is an $X$-EPG smooth fan. In other words, the goal of this section is to construct examples of $X$-EPG fans for every $X \sub [0,1]$ listed in Corollary~\ref{cor:list0}. Furthermore, we prove that all the examples we construct are in fact endpoint-homogeneous.

If $X = \0$, $\{1\}$, or $[0,1]$, as in the first bullet point, then we have examples of $X$-EPG fans already:

\begin{center}
\begin{tikzpicture}[xscale=.4,yscale=.4]

\node at (4.5,7.2) {$\ $};

\draw[very thin] (4.5,7) -- (0,0);
\draw[ultra thin] (4.5,7) -- (1/27,0);
\draw[ultra thin] (4.5,7) -- (2/27,0);
\draw[very thin] (4.5,7) -- (1/9,0);
\draw[very thin] (4.5,7) -- (2/9,0);
\draw[ultra thin] (4.5,7) -- (7/27,0);
\draw[ultra thin] (4.5,7) -- (8/27,0);
\draw[very thin] (4.5,7) -- (1/3,0);
\draw[very thin] (4.5,7) -- (2/3,0);
\draw[ultra thin] (4.5,7) -- (19/27,0);
\draw[ultra thin] (4.5,7) -- (20/27,0);
\draw[very thin] (4.5,7) -- (7/9,0);
\draw[very thin] (4.5,7) -- (8/9,0);
\draw[ultra thin] (4.5,7) -- (25/27,0);
\draw[ultra thin] (4.5,7) -- (26/27,0);
\draw[very thin] (4.5,7) -- (1,0);
\draw[very thin] (4.5,7) -- (2,0);
\draw[ultra thin] (4.5,7) -- (55/27,0);
\draw[ultra thin] (4.5,7) -- (56/27,0);
\draw[very thin] (4.5,7) -- (19/9,0);
\draw[very thin] (4.5,7) -- (20/9,0);
\draw[ultra thin] (4.5,7) -- (61/27,0);
\draw[ultra thin] (4.5,7) -- (62/27,0);
\draw[very thin] (4.5,7) -- (7/3,0);
\draw[very thin] (4.5,7) -- (8/3,0);
\draw[ultra thin] (4.5,7) -- (73/27,0);
\draw[ultra thin] (4.5,7) -- (74/27,0);
\draw[very thin] (4.5,7) -- (25/9,0);
\draw[very thin] (4.5,7) -- (26/9,0);
\draw[ultra thin] (4.5,7) -- (79/27,0);
\draw[ultra thin] (4.5,7) -- (80/27,0);
\draw[very thin] (4.5,7) -- (3,0);
\draw[very thin] (4.5,7) -- (6,0);
\draw[ultra thin] (4.5,7) -- (163/27,0);
\draw[ultra thin] (4.5,7) -- (164/27,0);
\draw[very thin] (4.5,7) -- (55/9,0);
\draw[very thin] (4.5,7) -- (56/9,0);
\draw[ultra thin] (4.5,7) -- (169/27,0);
\draw[ultra thin] (4.5,7) -- (170/27,0);
\draw[very thin] (4.5,7) -- (19/3,0);
\draw[very thin] (4.5,7) -- (20/3,0);
\draw[ultra thin] (4.5,7) -- (181/27,0);
\draw[ultra thin] (4.5,7) -- (182/27,0);
\draw[very thin] (4.5,7) -- (61/9,0);
\draw[very thin] (4.5,7) -- (62/9,0);
\draw[ultra thin] (4.5,7) -- (187/27,0);
\draw[ultra thin] (4.5,7) -- (188/27,0);
\draw[very thin] (4.5,7) -- (7,0);
\draw[very thin] (4.5,7) -- (8,0);
\draw[ultra thin] (4.5,7) -- (217/27,0);
\draw[ultra thin] (4.5,7) -- (218/27,0);
\draw[very thin] (4.5,7) -- (73/9,0);
\draw[very thin] (4.5,7) -- (74/9,0);
\draw[ultra thin] (4.5,7) -- (223/27,0);
\draw[ultra thin] (4.5,7) -- (224/27,0);
\draw[very thin] (4.5,7) -- (25/3,0);
\draw[very thin] (4.5,7) -- (26/3,0);
\draw[ultra thin] (4.5,7) -- (235/27,0);
\draw[ultra thin] (4.5,7) -- (236/27,0);
\draw[very thin] (4.5,7) -- (79/9,0);
\draw[very thin] (4.5,7) -- (80/9,0);
\draw[ultra thin] (4.5,7) -- (241/27,0);
\draw[ultra thin] (4.5,7) -- (242/27,0);
\draw[very thin] (4.5,7) -- (9,0);

\begin{scope}[shift={(21,7)}]
\draw[very thin] (0,0) -- (-4.5*.31,-7*.31);
\draw[very thin] (0,0) -- (-4.5*.41+1/27*.41,-7*.41);
\draw[very thin] (0,0) -- (-4.5*.59+2/27*.59,-7*.59);
\draw[very thin] (0,0) -- (-4.5*.26+1/9*.26,-7*.26);
\draw[very thin] (0,0) -- (-4.5*.53+2/9*.53,-7*.53);
\draw[very thin] (0,0) -- (-4.5*.58+7/27*.58,-7*.58);
\draw[very thin] (0,0) -- (-4.5*.97+8/27*.97,-7*.97);
\draw[very thin] (0,0) -- (-4.5*.93+1/3*.93,-7*.93);
\draw[very thin] (0,0) -- (-4.5*.23+2/3*.23,-7*.23);
\draw[very thin] (0,0) -- (-4.5*.84+19/27*.84,-7*.84);
\draw[very thin] (0,0) -- (-4.5*.62+20/27*.62,-7*.62);
\draw[very thin] (0,0) -- (-4.5*.64+7/9*.64,-7*.64);
\draw[very thin] (0,0) -- (-4.5*.33+8/9*.33,-7*.33);
\draw[very thin] (0,0) -- (-4.5*.83+25/27*.83,-7*.83);
\draw[very thin] (0,0) -- (-4.5*.27+26/27*.27,-7*.27);
\draw[very thin] (0,0) -- (-4.5*.95+1*.95,-7*.95);
\draw[very thin] (0,0) -- (-4.5*.02+2*.02,-7*.02);
\draw[very thin] (0,0) -- (-4.5*.88+55/27*.88,-7*.88);
\draw[very thin] (0,0) -- (-4.5*.41+56/27*.41,-7*.41);
\draw[very thin] (0,0) -- (-4.5*.97+19/9*.97,-7*.97);
\draw[very thin] (0,0) -- (-4.5*.16+20/9*.16,-7*.16);
\draw[very thin] (0,0) -- (-4.5*.93+61/27*.93,-7*.93);
\draw[very thin] (0,0) -- (-4.5*.99+62/27*.99,-7*.99);
\draw[very thin] (0,0) -- (-4.5*.37+7/3*.37,-7*.37);
\draw[very thin] (0,0) -- (-4.5*.51+8/3*.51,-7*.51);
\draw[very thin] (0,0) -- (-4.5*.05+73/27*.05,-7*.05);
\draw[very thin] (0,0) -- (-4.5*.82+74/27*.82,-7*.82);
\draw[very thin] (0,0) -- (-4.5*.09+25/9*.09,-7*.09);
\draw[very thin] (0,0) -- (-4.5*.74+26/9*.74,-7*.74);
\draw[very thin] (0,0) -- (-4.5*.94+79/27*.94,-7*.94);
\draw[very thin] (0,0) -- (-4.5*.45+80/27*.45,-7*.45);
\draw[very thin] (0,0) -- (-4.5*.92+3*.92,-7*.92);
\draw[very thin] (0,0) -- (-4.5*.3+6*.3,-7*.3);
\draw[very thin] (0,0) -- (-4.5*.78+163/27*.78,-7*.78);
\draw[very thin] (0,0) -- (-4.5*.16+164/27*.16,-7*.16);
\draw[very thin] (0,0) -- (-4.5*.4+55/9*.4,-7*.4);
\draw[very thin] (0,0) -- (-4.5*.62+56/9*.62,-7*.62);
\draw[very thin] (0,0) -- (-4.5*.86+169/27*.86,-7*.86);
\draw[very thin] (0,0) -- (-4.5*.2+170/27*.2,-7*.2);
\draw[very thin] (0,0) -- (-4.5*.89+19/3*.89,-7*.89);
\draw[very thin] (0,0) -- (-4.5*.98+20/3*.98,-7*.98);
\draw[very thin] (0,0) -- (-4.5*.62+181/27*.62,-7*.62);
\draw[very thin] (0,0) -- (-4.5*.8+182/27*.8,-7*.8);
\draw[very thin] (0,0) -- (-4.5*.34+61/9*.34,-7*.34);
\draw[very thin] (0,0) -- (-4.5*.82+62/9*.82,-7*.82);
\draw[very thin] (0,0) -- (-4.5*.53+187/27*.53,-7*.53);
\draw[very thin] (0,0) -- (-4.5*.42+188/27*.42,-7*.42);
\draw[very thin] (0,0) -- (-4.5*.11+7*.11,-7*.11);
\draw[very thin] (0,0) -- (-4.5*.7+8*.7,-7*.7);
\draw[very thin] (0,0) -- (-4.5*.67+217/27*.67,-7*.67);
\draw[very thin] (0,0) -- (-4.5*.98+218/27*.98,-7*.98);
\draw[very thin] (0,0) -- (-4.5*.21+73/9*.21,-7*.21);
\draw[very thin] (0,0) -- (-4.5*.48+74/9*.48,-7*.48);
\draw[very thin] (0,0) -- (-4.5*.08+223/27*.08,-7*.08);
\draw[very thin] (0,0) -- (-4.5*.65+224/27*.65,-7*.65);
\draw[very thin] (0,0) -- (-4.5*.13+25/3*.13,-7*.13);
\draw[very thin] (0,0) -- (-4.5*.28+26/3*.28,-7*.28);
\draw[very thin] (0,0) -- (-4.5*.23+235/27*.23,-7*.23);
\draw[very thin] (0,0) -- (-4.5*.06+236/27*.06,-7*.06);
\draw[very thin] (0,0) -- (-4.5*.64+79/9*.64,-7*.64);
\draw[very thin] (0,0) -- (-4.5*.7+80/9*.7,-7*.7);
\draw[very thin] (0,0) -- (-4.5*.93+241/27*.93,-7*.93);
\draw[very thin] (0,0) -- (-4.5*.84+242/27*.84,-7*.84);
\draw[very thin] (0,0) -- (-4.5*.46+9*.46,-7*.46);

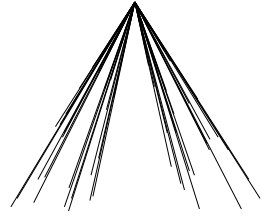
\node at (0,-8.2) {$[0,1]$-EPG: the Lelek fan};
\end{scope}

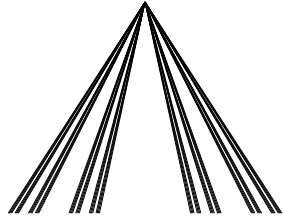
\node at (4.5,-1.2) {$\{1\}$-EPG: the Cantor fan};

\node at (12.75,-3.5) {$\0$-EPG: the simple $n$-ods};

\begin{scope}[shift={(2,-7)}]

\begin{scope}[shift={(0,-.5)}]
\draw[thick] (0,0) -- (0,2);
\begin{scope}[rotate=120]
\draw[thick] (0,0) -- (0,2);
\end{scope}
\begin{scope}[rotate=240]
\draw[thick] (0,0) -- (0,2);
\end{scope}
\end{scope}

\begin{scope}[shift={(6.25,0)}]
\draw[thick] (0,0) -- (2,0);
\begin{scope}[rotate=90]
\draw[thick] (0,0) -- (2,0);
\end{scope}
\begin{scope}[rotate=180]
\draw[thick] (0,0) -- (2,0);
\end{scope}
\begin{scope}[rotate=270]
\draw[thick] (0,0) -- (2,0);
\end{scope}
\end{scope}

\begin{scope}[shift={(12.5,-.2)}]
\draw[thick] (0,0) -- (0,2);
\begin{scope}[rotate=72]
\draw[thick] (0,0) -- (0,2);
\end{scope}
\begin{scope}[rotate=144]
\draw[thick] (0,0) -- (0,2);
\end{scope}
\begin{scope}[rotate=216]
\draw[thick] (0,0) -- (0,2);
\end{scope}
\begin{scope}[rotate=288]
\draw[thick] (0,0) -- (0,2);
\end{scope}
\end{scope}

\begin{scope}[shift={(18.75,0)}]
\draw[thick] (0,0) -- (0,2);
\begin{scope}[rotate=60]
\draw[thick] (0,0) -- (0,2);
\end{scope}
\begin{scope}[rotate=120]
\draw[thick] (0,0) -- (0,2);
\end{scope}
\begin{scope}[rotate=180]
\draw[thick] (0,0) -- (0,2);
\end{scope}
\begin{scope}[rotate=240]
\draw[thick] (0,0) -- (0,2);
\end{scope}
\begin{scope}[rotate=300]
\draw[thick] (0,0) -- (0,2);
\end{scope}
\end{scope}

\node at (23,0) {\large $\dots$};

\end{scope}

\end{tikzpicture}
\end{center}

\noindent All of these examples are endpoint-homogeneous. In fact, as mentioned in the introduction, they are all $\frac{1}{3}$-homogeneous \cite{AHJ}. These examples are also unique up to homeomorphism:

\begin{observation}
Up to homeomorphism, the Cantor fan is the only $\{1\}$-EPG smooth fan, 
the Lelek fan is the only $[0,1]$-EPG smooth fan, and 
the simple $n$-ods are the only $\0$-EPG smooth fans.
\end{observation}
\begin{proof}
For the Cantor fan, this follows from Corollary~\ref{cor:Cantor}. 
For the Lelek fan, this follows from Charatonik's characterization of the Lelek fan in \cite{lelekUniqueChar}. 
For the $n$-ods, note that if $F$ is a $\0$-EPG fan, then $E(F)$ has no accumulation points. Because $F$ is compact, this implies $E(F)$ is finite, which implies $F$ is a simple $n$-od.
\end{proof}

It remains to find $X$-EPG, endpoint-homogeneous fans for sets of type $(2)$ and $(3)$ in the list from Corollary~\ref{cor:list0}. 
Recall from Theorems~\ref{thm:Countable} and \ref{thm:Uncountable} that the $X$-EPG fans with $X$ of type $(2)$ should have countably many blades, while the $X$-EPG fans with $X$ of type $(3)$ should have uncountably many blades. We cover the case of countably many blades first.

\begin{theorem}\label{thm:EPG}
If $X$ is a closed subset of $[0,1]$ with $1 \notin X$, then there is an $X$-EPG, endpoint-homogeneous smooth fan.
\end{theorem}
\begin{proof}
If $X = \0$ or $X = \{0\}$, then we already have examples of $X$-EPG smooth fans: the simple $n$-ods for $X = \0$, and the star for $X = \{0\}$ (see the comments following Corollary~\ref{cor:4kinds}). Furthermore, it is clear that these examples are endpoint-homogeneous. 
Thus, for the remainder of the proof, fix some closed $X \sub [0,1]$ such that $0 \in X$, $1 \notin X$, and $X \cap (0,1) \neq \0$. 

Let $D$ be a countable dense subset of $X \setminus \{0\}$ (which is nonempty, by the assumption of the previous paragraph), and let $\seq{y_n}{n \in \N}$ be an infinite sequence such that $y_n \in D$ for every $n$, and for every $d \in D$ there are infinitely many $n$ with $y_n = d$. 

We now describe the construction of a comb over $K$ such that the corresponding fan will be $X$-EPG and endpoint-homogeneous. To begin, let $L^0 = \{0\}\times [0,1]$. Let $M = \max (X)$, and note that $M < 1$. Then:
\begin{enumerate}
	\item For each $n_1\in\mathbb{N}$, let
	$$
	L_{n_1}^1 = \left\{\frac{2}{3^{n_1}}\right\}\times \big[ 0,y_{n_1} \big]
	.$$
	\item For each $n_1$, $n_2\in\mathbb{N}$, let
	$$
	L_{n_1,n_2}^2 = \left\{\frac{2}{3^{n_1}}+\frac{2}{3^{n_1+n_2}}\right\}\times \big[ 0,M^{n_1}y_{n_1}y_{n_2} \big]
	.$$
	\item For each $n_1$, $n_2$, $n_3\in\mathbb{N}$, let
	$$
	L_{n_1,n_2,n_3}^3 = \left\{\frac{2}{3^{n_1}}+\frac{2}{3^{n_1+n_2}}+
	\frac{2}{3^{n_1+n_2+n_3}}\right\}\times \big[ 0,M^{n_1+n_2}y_{n_1}y_{n_2}y_{n_3} \big]
	.$$
	\item More generally, for any given $k \in \N$ and for $n_1$, $n_2,\ldots ,n_k\in\mathbb{N}$, let 
	$$
	L_{n_1,n_2,\ldots,n_k}^k = \left\{2\cdot\sum_{i=1}^k \left(\frac{1}{3}\right)^{\sum_{j=1}^in_j}\right\}\times \left[0 \,,\, M^{\sum_{i=1}^{k-1}n_i}\cdot \prod_{i=1}^ky_{n_i}\right]
	.$$
\end{enumerate}
For each $N \in \N$, let $C^N$ be the result of the first $N$ stages of this construction, 
together with $I \times \{0\}$: i.e., 
$$C^N \,=\, I \times \{0\} \cup L^0 \cup \bigcup \set{L^k_{n_1,n_2,\dots,n_k}}{k \leq N \text{ and } (n_1,n_2,\dots,n_k) \in \N^k}.$$
Finally, let $C = \bigcup_{N \in \N}C^N$. 
This is illustrated below for $X = \{ 0,\nicefrac{1}{2},\nicefrac{3}{4}\}$.

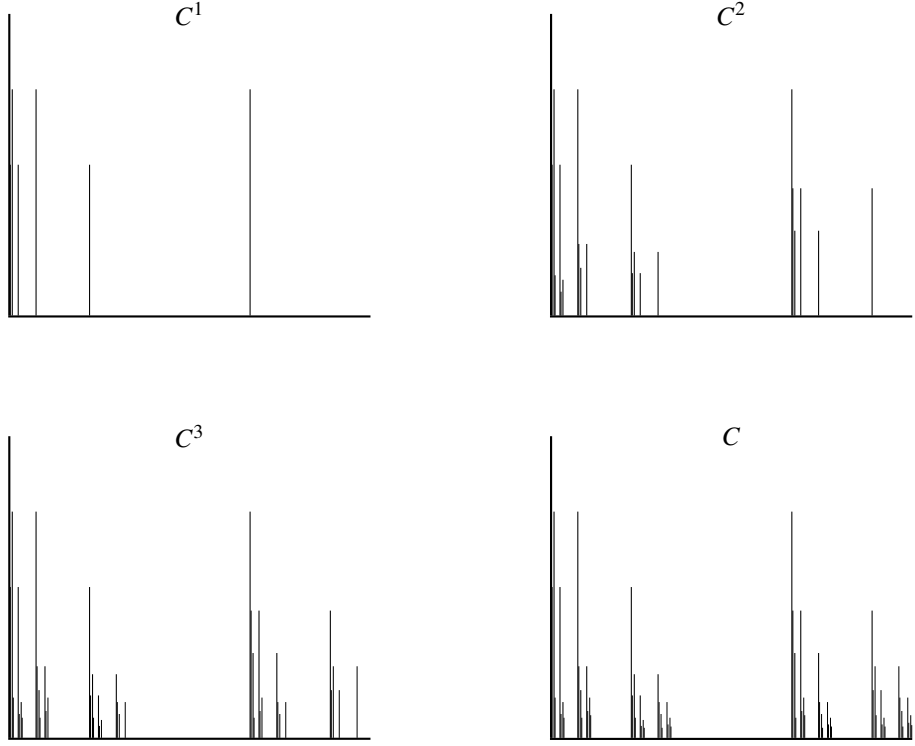
\begin{center}
\begin{tikzpicture}[xscale=4.75,yscale=4]

\begin{scope}[shift={(0,0)}]

\draw[thick] (-.003,-.003) -- (1,-.003);

\draw[thick] (0,0) -- (0,1);

\draw (2/3,0) -- (2/3,.75);
\draw (2/9,0) -- (2/9,.5);
\draw (2/27,0) -- (2/27,.75);
\draw (2/81,0) -- (2/81,.5);
\draw (2/243,0) -- (2/243,.75);
\draw (2/729,0) -- (2/729,.5);

\node at (.5,1) {\footnotesize $C^1$};

\end{scope}

\begin{scope}[shift={(1.5,0)}]

\draw[thick] (-.003,-.003) -- (1,-.003);

\draw[thick] (0,0) -- (0,1);

\draw (2/3,0) -- (2/3,.75);
\draw (2/9,0) -- (2/9,.5);
\draw (2/27,0) -- (2/27,.75);
\draw (2/81,0) -- (2/81,.5);
\draw (2/243,0) -- (2/243,.75);
\draw (2/729,0) -- (2/729,.5);

\draw[thin] (2/3+2/9,0) -- (2/3+2/9,.75*.75*.75);
\draw[thin] (2/3+2/27,0) -- (2/3+2/27,.75*.75*.5);
\draw[thin] (2/3+2/81,0) -- (2/3+2/81,.75*.75*.75);
\draw[thin] (2/3+2/243,0) -- (2/3+2/243,.75*.75*.5);
\draw[thin] (2/3+2/729,0) -- (2/3+2/729,.75*.75*.75);

\draw[thin] (2/9+2/27,0) -- (2/9+2/27,.75^2*.5*.75);
\draw[thin] (2/9+2/81,0) -- (2/9+2/81,.75^2*.5*.5);
\draw[thin] (2/9+2/243,0) -- (2/9+2/243,.75^2*.5*.75);
\draw[thin] (2/9+2/729,0) -- (2/9+2/729,.75^2*.5*.5);

\draw[thin] (2/27+2/81,0) -- (2/27+2/81,.75^3*.75*.75);
\draw[thin] (2/27+2/243,0) -- (2/27+2/243,.75^3*.75*.5);
\draw[thin] (2/27+2/729,0) -- (2/27+2/729,.75^3*.75*.75);

\draw[thin] (2/81+2/243,0) -- (2/81+2/243,.75^4*.5*.75);
\draw[thin] (2/81+2/729,0) -- (2/81+2/729,.75^4*.5*.5);

\draw[thin] (2/243+2/729,0) -- (2/243+2/729,.75^5*.75*.75);

\node at (.5,1) {\footnotesize $C^2$};

\end{scope}

\begin{scope}[shift={(0,-1.4)}]

\draw[thick] (-.003,-.003) -- (1,-.003);

\draw[thick] (0,0) -- (0,1);

\draw (2/3,0) -- (2/3,.75);
\draw (2/9,0) -- (2/9,.5);
\draw (2/27,0) -- (2/27,.75);
\draw (2/81,0) -- (2/81,.5);
\draw (2/243,0) -- (2/243,.75);
\draw (2/729,0) -- (2/729,.5);

\draw[thin] (2/3+2/9,0) -- (2/3+2/9,.75*.75*.75);
\draw[thin] (2/3+2/27,0) -- (2/3+2/27,.75*.75*.5);
\draw[thin] (2/3+2/81,0) -- (2/3+2/81,.75*.75*.75);
\draw[thin] (2/3+2/243,0) -- (2/3+2/243,.75*.75*.5);
\draw[thin] (2/3+2/729,0) -- (2/3+2/729,.75*.75*.75);

\draw[thin] (2/9+2/27,0) -- (2/9+2/27,.75^2*.5*.75);
\draw[thin] (2/9+2/81,0) -- (2/9+2/81,.75^2*.5*.5);
\draw[thin] (2/9+2/243,0) -- (2/9+2/243,.75^2*.5*.75);
\draw[thin] (2/9+2/729,0) -- (2/9+2/729,.75^2*.5*.5);

\draw[thin] (2/27+2/81,0) -- (2/27+2/81,.75^3*.75*.75);
\draw[thin] (2/27+2/243,0) -- (2/27+2/243,.75^3*.75*.5);
\draw[thin] (2/27+2/729,0) -- (2/27+2/729,.75^3*.75*.75);

\draw[thin] (2/81+2/243,0) -- (2/81+2/243,.75^4*.5*.75);
\draw[thin] (2/81+2/729,0) -- (2/81+2/729,.75^4*.5*.5);

\draw[thin] (2/243+2/729,0) -- (2/243+2/729,.75^5*.75*.75);

\draw[very thin] (2/3+2/9+2/27,0) -- (2/3+2/9+2/27,.75^2*.75*.75*.75);
\draw[very thin] (2/3+2/9+2/81,0) -- (2/3+2/9+2/81,.75^2*.75*.75*.5);
\draw[very thin] (2/3+2/9+2/243,0) -- (2/3+2/9+2/243,.75^2*.75*.75*.75);
\draw[very thin] (2/3+2/9+2/729,0) -- (2/3+2/9+2/729,.75^2*.75*.75*.5);

\draw[very thin] (2/3+2/27+2/81,0) -- (2/3+2/27+2/81,.75^3*.75*.5*.75);
\draw[very thin] (2/3+2/27+2/243,0) -- (2/3+2/27+2/243,.75^3*.75*.5*.5);
\draw[very thin] (2/3+2/27+2/729,0) -- (2/3+2/27+2/729,.75^3*.75*.5*.75);

\draw[very thin] (2/3+2/81+2/243,0) -- (2/3+2/81+2/243,.75^4*.75*.75*.75);
\draw[very thin] (2/3+2/81+2/729,0) -- (2/3+2/81+2/729,.75^4*.75*.75*.5);

\draw[very thin] (2/3+2/243+2/729,0) -- (2/3+2/243+2/729,.75^5*.75*.5*.75);

\draw[very thin] (2/9+2/27+2/81,0) -- (2/9+2/27+2/81,.75^3*.5*.75*.75);
\draw[very thin] (2/9+2/27+2/243,0) -- (2/9+2/27+2/243,.75^3*.5*.75*.5);
\draw[very thin] (2/9+2/27+2/729,0) -- (2/9+2/27+2/729,.75^3*.5*.75*.75);

\draw[very thin] (2/9+2/81+2/243,0) -- (2/9+2/81+2/243,.75^4*.5*.5*.75);
\draw[very thin] (2/9+2/81+2/729,0) -- (2/9+2/81+2/729,.75^4*.5*.5*.5);

\draw[very thin] (2/9+2/243+2/729,0) -- (2/9+2/243+2/729,.75^5*.5*.75*.75);

\draw[very thin] (2/27+2/81+2/243,0) -- (2/27+2/81+2/243,.75^4*.75*.75*.75);
\draw[very thin] (2/27+2/81+2/729,0) -- (2/27+2/81+2/729,.75^4*.75*.75*.5);

\draw[very thin] (2/27+2/243+2/729,0) -- (2/27+2/243+2/729,.75^5*.75*.5*.75);

\draw[very thin] (2/81+2/243+2/729,0) -- (2/81+2/243+2/729,.75^5*.5*.75*.75);

\node at (.5,1) {\footnotesize $C^3$};

\end{scope}

\begin{scope}[shift={(1.5,-1.4)}]

\draw[thick] (-.003,-.0035) -- (1,-.0035);

\draw[thick] (0,0) -- (0,1);

\draw (2/3,0) -- (2/3,.75);
\draw (2/9,0) -- (2/9,.5);
\draw (2/27,0) -- (2/27,.75);
\draw (2/81,0) -- (2/81,.5);
\draw (2/243,0) -- (2/243,.75);
\draw (2/729,0) -- (2/729,.5);

\draw[thin] (2/3+2/9,0) -- (2/3+2/9,.75*.75*.75);
\draw[thin] (2/3+2/27,0) -- (2/3+2/27,.75*.75*.5);
\draw[thin] (2/3+2/81,0) -- (2/3+2/81,.75*.75*.75);
\draw[thin] (2/3+2/243,0) -- (2/3+2/243,.75*.75*.5);
\draw[thin] (2/3+2/729,0) -- (2/3+2/729,.75*.75*.75);

\draw[thin] (2/9+2/27,0) -- (2/9+2/27,.75^2*.5*.75);
\draw[thin] (2/9+2/81,0) -- (2/9+2/81,.75^2*.5*.5);
\draw[thin] (2/9+2/243,0) -- (2/9+2/243,.75^2*.5*.75);
\draw[thin] (2/9+2/729,0) -- (2/9+2/729,.75^2*.5*.5);

\draw[thin] (2/27+2/81,0) -- (2/27+2/81,.75^3*.75*.75);
\draw[thin] (2/27+2/243,0) -- (2/27+2/243,.75^3*.75*.5);
\draw[thin] (2/27+2/729,0) -- (2/27+2/729,.75^3*.75*.75);

\draw[thin] (2/81+2/243,0) -- (2/81+2/243,.75^4*.5*.75);
\draw[thin] (2/81+2/729,0) -- (2/81+2/729,.75^4*.5*.5);

\draw[thin] (2/243+2/729,0) -- (2/243+2/729,.75^5*.75*.75);

\draw[very thin] (2/3+2/9+2/27,0) -- (2/3+2/9+2/27,.75^2*.75*.75*.75);
\draw[very thin] (2/3+2/9+2/81,0) -- (2/3+2/9+2/81,.75^2*.75*.75*.5);
\draw[very thin] (2/3+2/9+2/243,0) -- (2/3+2/9+2/243,.75^2*.75*.75*.75);
\draw[very thin] (2/3+2/9+2/729,0) -- (2/3+2/9+2/729,.75^2*.75*.75*.5);

\draw[very thin] (2/3+2/27+2/81,0) -- (2/3+2/27+2/81,.75^3*.75*.5*.75);
\draw[very thin] (2/3+2/27+2/243,0) -- (2/3+2/27+2/243,.75^3*.75*.5*.5);
\draw[very thin] (2/3+2/27+2/729,0) -- (2/3+2/27+2/729,.75^3*.75*.5*.75);

\draw[very thin] (2/3+2/81+2/243,0) -- (2/3+2/81+2/243,.75^4*.75*.75*.75);
\draw[very thin] (2/3+2/81+2/729,0) -- (2/3+2/81+2/729,.75^4*.75*.75*.5);

\draw[very thin] (2/3+2/243+2/729,0) -- (2/3+2/243+2/729,.75^5*.75*.5*.75);

\draw[very thin] (2/9+2/27+2/81,0) -- (2/9+2/27+2/81,.75^3*.5*.75*.75);
\draw[very thin] (2/9+2/27+2/243,0) -- (2/9+2/27+2/243,.75^3*.5*.75*.5);
\draw[very thin] (2/9+2/27+2/729,0) -- (2/9+2/27+2/729,.75^3*.5*.75*.75);

\draw[very thin] (2/9+2/81+2/243,0) -- (2/9+2/81+2/243,.75^4*.5*.5*.75);
\draw[very thin] (2/9+2/81+2/729,0) -- (2/9+2/81+2/729,.75^4*.5*.5*.5);

\draw[very thin] (2/9+2/243+2/729,0) -- (2/9+2/243+2/729,.75^5*.5*.75*.75);

\draw[very thin] (2/27+2/81+2/243,0) -- (2/27+2/81+2/243,.75^4*.75*.75*.75);
\draw[very thin] (2/27+2/81+2/729,0) -- (2/27+2/81+2/729,.75^4*.75*.75*.5);

\draw[very thin] (2/27+2/243+2/729,0) -- (2/27+2/243+2/729,.75^5*.75*.5*.75);

\draw[very thin] (2/81+2/243+2/729,0) -- (2/81+2/243+2/729,.75^5*.5*.75*.75);

\draw[very thin] (2/3+2/9+2/27+2/81,0) -- (2/3+2/9+2/27+2/81,.75^3*.75*.75*.75*.75);
\draw[very thin] (2/3+2/9+2/27+2/243,0) -- (2/3+2/9+2/27+2/243,.75^3*.75*.75*.75*.5);
\draw[very thin] (2/3+2/9+2/27+2/729,0) -- (2/3+2/9+2/27+2/729,.75^3*.75*.75*.75*.75);

\draw[very thin] (2/3+2/9+2/27+2/81+2/243,0) -- (2/3+2/9+2/27+2/81+2/243,.75^4*.75*.75*.75*.75*.75);
\draw[very thin] (2/3+2/9+2/27+2/81+2/729,0) -- (2/3+2/9+2/27+2/81+2/729,.75^4*.75*.75*.75*.75*.5);

\draw[very thin] (2/3+2/9+2/27+2/243+2/729,0) -- (2/3+2/9+2/27+2/243+2/729,.75^5*.75*.75*.75*.5*.75);

\draw[very thin] (2/3+2/9+2/27+2/81+2/243+2/729,0) -- (2/3+2/9+2/27+2/81+2/243+2/729,.75^5*.75*.75*.75*.75*.75*.75);

\draw[very thin] (2/3+2/27+2/81+2/243,0) -- (2/3+2/27+2/81+2/243,.75^4*.75*.5*.75*.75);
\draw[very thin] (2/3+2/27+2/81+2/729,0) -- (2/3+2/27+2/81+2/729,.75^4*.75*.5*.75*.5);
\draw[very thin] (2/3+2/27+2/81+2/243+2/729,0) -- (2/3+2/27+2/81+2/243+2/729,.75^5*.75*.5*.75*.75*.75);
\draw[very thin] (2/3+2/27+2/243+2/729,0) -- (2/3+2/27+2/243+2/729,.75^5*.75*.5*.5*.75);

\draw[very thin] (2/3+2/81+2/243+2/729,0) -- (2/3+2/81+2/243+2/729,.75^5*.75*.75*.75*.75);

\draw[very thin] (2/27+2/81+2/243+2/729,0) -- (2/27+2/81+2/243+2/729,.75^5*.75*.75*.75*.75);

\draw[very thin] (2/9+2/81+2/243+2/729,0) -- (2/9+2/81+2/243+2/729,.75^5*.5*.5*.75*.75);

\draw[very thin] (2/9+2/27+2/243+2/729,0) -- (2/9+2/27+2/243+2/729,.75^5*.5*.75*.5*.75);
\draw[very thin] (2/9+2/27+2/81+2/729,0) -- (2/9+2/27+2/81+2/729,.75^4*.5*.75*.75*.5);
\draw[very thin] (2/9+2/27+2/81+2/243+2/729,0) -- (2/9+2/27+2/81+2/243+2/729,.75^5*.5*.75*.75*.75*.75);
\draw[very thin] (2/9+2/27+2/81+2/243,0) -- (2/9+2/27+2/81+2/243,.75^4*.5*.75*.75*.75);

\draw[very thin] (2/3+2/9+2/243+2/729,0) -- (2/3+2/9+2/243+2/729,.75^5*.75*.75*.75*.75);
\draw[very thin] (2/3+2/9+2/81+2/243+2/729,0) -- (2/3+2/9+2/81+2/243+2/729,.75^5*.75*.75*.5*.75*.75);
\draw[very thin] (2/3+2/9+2/81+2/243,0) -- (2/3+2/9+2/81+2/243,.75^4*.75*.75*.5*.75);
\draw[very thin] (2/3+2/9+2/81+2/729,0) -- (2/3+2/9+2/81+2/729,.75^4*.75*.75*.5*.5);

%%%%

\node at (.5,1) {\footnotesize $C$};

\end{scope}

\end{tikzpicture}
\end{center}

Observe that the set of all $x \in I$ where this construction places a vertical arc is exactly 
$$\textstyle \set{\sum_{i \in F} \frac{2}{3^i}}{F \text{ is a finite subset of }\N}.$$
To put it more simply, this is just the set of all left-hand endpoints of the Cantor set $K$. This is a countable dense subset of $K$. 

Let us check that $C$ is closed. Let $\seq{z_n}{n \in \N}$ be a sequence of points in $C$; we wish to find a convergent subsequence. 
If $\liminf_{n \to \infty} \pi_2(z_n) = 0$, then it is not difficult to see that some subsequence converges to a point of $I \times \{0\}$. So let us assume $\liminf_{n \to \infty} \pi_2(z_n) = \e > 0$. The exponent on $M$ in the definition of $L^k_{n_1,n_2,\dots,n_k}$ is at least $k-1$. Because $M < 0$, this means there is some $N \in \N$ such that $M^N < \e$, which implies  
$\set{z_n}{n \in \N} \sub C^N.$
But it is not difficult to see that $C^N$ is compact. So in this case, some subsequence of the $z_n$ converges to a point of $C^N$.

Thus $C$ is closed, and it is easy to see that $C$ satisfies every other part of Definition~\ref{def:comb}. Hence $C$ is a comb, and collapsing the bottom line of $C$ to a point yields a smooth fan $F_C$. 
To finish the proof, we must show that this fan is an endpoint-homogeneous $X$-EPG fan. 

For each vertical arc $L^k_{n_1,\dots,n_k}$ produced by this construction, let 
$$x^k_{n_1,\dots,n_k} \,=\,\textstyle 2\cdot\sum_{i=1}^k \left(\frac{1}{3}\right)^{\sum_{j=1}^in_j},$$
i.e., $x^k_{n_1,\dots,n_k}$ is the $X$-coordinate of the vertical arc $L^k_{n_1,\dots,n_k}$. Let 
$$e_{n_1,\ldots,n_k}^k \,=\, \textstyle M^{\sum_{i=1}^{k-1}n_i}\cdot \prod_{i=1}^ky_{n_i},$$
or in other words, $e^k_{n_1,\dots,n_k}$ is the $Y$-coordinate of the tip of the blade at $x^k_{n_1,\dots,n_k}$. Note that $(x^k_{n_1,\dots,n_k},e^k_{n_1,\dots,n_k}) \in \mathcal E(C)$. 
Let
$$U^k_{n_1,\dots,n_k} \,=\, \textstyle \Big[ x^k_{n_1,\dots,n_k} \,,\, x^k_{n_1,\dots,n_k}+\left( \frac{1}{3} \right)^{\sum_{i=1}^k n_i} \Big] \cap K.$$
Because $x^k_{n_1,\dots,n_k}$ is the left-hand endpoint of an interval in the Cantor set $K$, this is simply the basic clopen subset of $K$ with left-hand endpoint $x^k_{n_1,\dots,n_k}$ and with width $\left( \frac{1}{3} \right)^{\sum_{i=1}^k n_i}$. This particular width for $U^k_{n_1,\dots,n_k}$ is chosen precisely so as to make the following observation true.

\vspace{3mm}

\noindent \textbf{Observation:} Let $x^k_{n_1,\dots,n_k}$ and $x^\ell_{m_1,\dots,m_\ell}$ be two points at which $C$ has a blade. Then $x^\ell_{m_1,\dots,m_\ell} \in U^k_{n_1,\dots,n_k}$ if and only if 
$\ell \geq k$ and $(n_1,n_2,\dots,n_k) = (m_1,m_2,\dots,m_k)$ (but $m_{k+1},\dots,m_\ell$ may be any natural numbers). 

\vspace{3mm}

Let $C_{>0} = \bigcup_{y > 0}C_y = C \cap (0,1]$. For each point $x^k_{n_1,\dots,n_k}$ at which $C$ has a blade, let $V^k_{n_1,\dots,n_k} = \big( U^k_{n_1,\dots,n_k} \times (0,1] \big) \cap C_{>0}$. 

\vspace{3mm}

\noindent \textbf{Claim:} Let $x^k_{n_1,\dots,n_k}$ be a point at which $C$ has a blade. There is a homeomorphism $h^k_{n_1,\dots,n_k}: V^k_{n_1,\dots,n_k} \to C_{>0}$ mapping the tip of $B_{x^k_{n_1,\dots,n_k}}$ to $(1,0)$. 

\vspace{3mm}

\noindent \emph{Proof of Claim:} 
This claim is proved by exploiting the fractal-like self-similarities of $C_{>0}$. The idea is simply that $V^k_{n_1,\dots,n_k}$ is just a translated, dilated copy of $C_{>0}$, but where the dilation is not quite even: the left-most blade of $V^k_{n_1,\dots,n_k}$ is dilated by a larger factor than every other blade.

To make this precise, define
\begin{align*}
m^k_{n_1,\dots,n_k}  \,=\, \max \set{y}{(x^k_{n_1,\dots,n_k},y) \in \closure{\mathcal E(C) \setminus B_{x^k_{n_1,\dots,n_k}}}} 
 \textstyle \,=\, M^{\sum_{i=1}^{k}n_i}\cdot M \cdot \prod_{i=1}^ky_{n_i}.
\end{align*}
In other words, this is the $Y$-coordinate of the highest point on $B_{x^k_{n_1,\dots,n_k}}$ (the left-most blade in $V^k_{n_1,\dots,n_k}$) that is in the closure of the other blades. 
Define 
\begin{align*} h^k_{n_1,\dots,n_k}(x,y) &\,=\, \textstyle \Bigg( 3^{\sum_{i=1}^k n_i} \Big( x-x^k_{n_1,\dots,n_k} \Big) \,,\, \frac{1}{M^{\sum_{i=1}^{k}n_i}\cdot \prod_{i=1}^ky_{n_i}}y \Bigg) \\
&\,=\, \textstyle \left( 3^{\sum_{i=1}^k n_i} \Big( x-x^k_{n_1,\dots,n_k} \Big) \,,\, \frac{M}{m^k_{n_1,\dots,n_k}}y \right)
\end{align*}
for all $(x,y) \in V^k_{n_1,\dots,n_k}$ such that $y \leq m^k_{n_1,\dots,n_k}$. This condition, that $y \leq m^k_{n_1,\dots,n_k}$, is true of every point of $V^k_{n_1,\dots,n_k}$ except for those in the topmost part of the blade $B_{x^k_{n_1,\dots,n_k}}$, specifically the points in $\{x^k_{n_1,\dots,n_k}\} \times (m^k_{n_1,\dots,n_k},e^k_{n_1,\dots,n_k}]$. Define $h^k_{n_1,\dots,n_k}$ on this part of this blade simply by mapping it linearly onto $\{0\} \times (M,1]$. That is, define
$$h^k_{n_1,\dots,n_k}(x,y) \,=\, \textstyle \left( 0 \,,\, M + \frac{1-M}{e^k_{n_1,\dots,n_k}-m^k_{n_1,\dots,n_k}} \Big( y-m^k_{n_1,\dots,n_k} \Big) \right)$$
if $x = x^k_{n_1,\dots,n_k}$ and $y > m^k_{n_1,\dots,n_k}$.

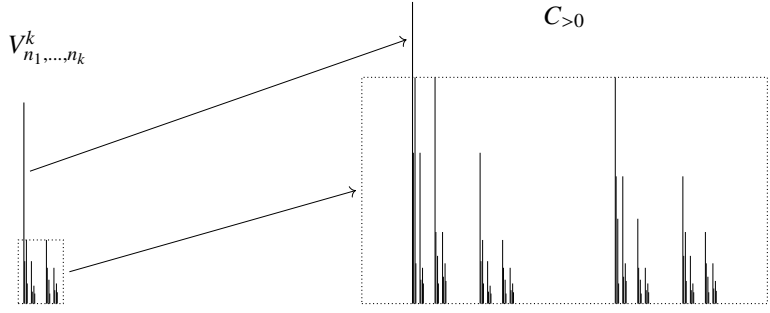
\begin{center}
\begin{tikzpicture}[scale=4]

\node at (.5,.95) {\footnotesize $C_{>0}$};

\draw[thin] (0,0) -- (0,1);

\draw[very thin] (2/3,0) -- (2/3,.75);
\draw[very thin] (2/9,0) -- (2/9,.5);
\draw[very thin] (2/27,0) -- (2/27,.75);
\draw[very thin] (2/81,0) -- (2/81,.5);
\draw[very thin] (2/243,0) -- (2/243,.75);
\draw[very thin] (2/729,0) -- (2/729,.5);

\draw[very thin] (2/3+2/9,0) -- (2/3+2/9,.75*.75*.75);
\draw[very thin] (2/3+2/27,0) -- (2/3+2/27,.75*.75*.5);
\draw[very thin] (2/3+2/81,0) -- (2/3+2/81,.75*.75*.75);
\draw[very thin] (2/3+2/243,0) -- (2/3+2/243,.75*.75*.5);
\draw[very thin] (2/3+2/729,0) -- (2/3+2/729,.75*.75*.75);

\draw[very thin] (2/9+2/27,0) -- (2/9+2/27,.75^2*.5*.75);
\draw[very thin] (2/9+2/81,0) -- (2/9+2/81,.75^2*.5*.5);
\draw[very thin] (2/9+2/243,0) -- (2/9+2/243,.75^2*.5*.75);
\draw[very thin] (2/9+2/729,0) -- (2/9+2/729,.75^2*.5*.5);

\draw[very thin] (2/27+2/81,0) -- (2/27+2/81,.75^3*.75*.75);
\draw[very thin] (2/27+2/243,0) -- (2/27+2/243,.75^3*.75*.5);
\draw[very thin] (2/27+2/729,0) -- (2/27+2/729,.75^3*.75*.75);

\draw[very thin] (2/81+2/243,0) -- (2/81+2/243,.75^4*.5*.75);
\draw[very thin] (2/81+2/729,0) -- (2/81+2/729,.75^4*.5*.5);

\draw[very thin] (2/243+2/729,0) -- (2/243+2/729,.75^5*.75*.75);

\draw[very thin] (2/3+2/9+2/27,0) -- (2/3+2/9+2/27,.75^2*.75*.75*.75);
\draw[very thin] (2/3+2/9+2/81,0) -- (2/3+2/9+2/81,.75^2*.75*.75*.5);
\draw[very thin] (2/3+2/9+2/243,0) -- (2/3+2/9+2/243,.75^2*.75*.75*.75);
\draw[very thin] (2/3+2/9+2/729,0) -- (2/3+2/9+2/729,.75^2*.75*.75*.5);

\draw[very thin] (2/3+2/27+2/81,0) -- (2/3+2/27+2/81,.75^3*.75*.5*.75);
\draw[very thin] (2/3+2/27+2/243,0) -- (2/3+2/27+2/243,.75^3*.75*.5*.5);
\draw[very thin] (2/3+2/27+2/729,0) -- (2/3+2/27+2/729,.75^3*.75*.5*.75);

\draw[very thin] (2/3+2/81+2/243,0) -- (2/3+2/81+2/243,.75^4*.75*.75*.75);
\draw[very thin] (2/3+2/81+2/729,0) -- (2/3+2/81+2/729,.75^4*.75*.75*.5);

\draw[very thin] (2/3+2/243+2/729,0) -- (2/3+2/243+2/729,.75^5*.75*.5*.75);

\draw[very thin] (2/9+2/27+2/81,0) -- (2/9+2/27+2/81,.75^3*.5*.75*.75);
\draw[very thin] (2/9+2/27+2/243,0) -- (2/9+2/27+2/243,.75^3*.5*.75*.5);
\draw[very thin] (2/9+2/27+2/729,0) -- (2/9+2/27+2/729,.75^3*.5*.75*.75);

\draw[very thin] (2/9+2/81+2/243,0) -- (2/9+2/81+2/243,.75^4*.5*.5*.75);
\draw[very thin] (2/9+2/81+2/729,0) -- (2/9+2/81+2/729,.75^4*.5*.5*.5);

\draw[very thin] (2/9+2/243+2/729,0) -- (2/9+2/243+2/729,.75^5*.5*.75*.75);

\draw[very thin] (2/27+2/81+2/243,0) -- (2/27+2/81+2/243,.75^4*.75*.75*.75);
\draw[very thin] (2/27+2/81+2/729,0) -- (2/27+2/81+2/729,.75^4*.75*.75*.5);

\draw[very thin] (2/27+2/243+2/729,0) -- (2/27+2/243+2/729,.75^5*.75*.5*.75);

\draw[very thin] (2/81+2/243+2/729,0) -- (2/81+2/243+2/729,.75^5*.5*.75*.75);

\draw[very thin] (2/3+2/9+2/27+2/81,0) -- (2/3+2/9+2/27+2/81,.75^3*.75*.75*.75*.75);
\draw[very thin] (2/3+2/9+2/27+2/243,0) -- (2/3+2/9+2/27+2/243,.75^3*.75*.75*.75*.5);
\draw[very thin] (2/3+2/9+2/27+2/729,0) -- (2/3+2/9+2/27+2/729,.75^3*.75*.75*.75*.75);

\draw[very thin] (2/3+2/9+2/27+2/81+2/243,0) -- (2/3+2/9+2/27+2/81+2/243,.75^4*.75*.75*.75*.75*.75);
\draw[very thin] (2/3+2/9+2/27+2/81+2/729,0) -- (2/3+2/9+2/27+2/81+2/729,.75^4*.75*.75*.75*.75*.5);

\draw[very thin] (2/3+2/9+2/27+2/243+2/729,0) -- (2/3+2/9+2/27+2/243+2/729,.75^5*.75*.75*.75*.5*.75);

\draw[very thin] (2/3+2/9+2/27+2/81+2/243+2/729,0) -- (2/3+2/9+2/27+2/81+2/243+2/729,.75^5*.75*.75*.75*.75*.75*.75);

\draw[very thin] (2/3+2/27+2/81+2/243,0) -- (2/3+2/27+2/81+2/243,.75^4*.75*.5*.75*.75);
\draw[very thin] (2/3+2/27+2/81+2/729,0) -- (2/3+2/27+2/81+2/729,.75^4*.75*.5*.75*.5);
\draw[very thin] (2/3+2/27+2/81+2/243+2/729,0) -- (2/3+2/27+2/81+2/243+2/729,.75^5*.75*.5*.75*.75*.75);
\draw[very thin] (2/3+2/27+2/243+2/729,0) -- (2/3+2/27+2/243+2/729,.75^5*.75*.5*.5*.75);

\draw[very thin] (2/3+2/81+2/243+2/729,0) -- (2/3+2/81+2/243+2/729,.75^5*.75*.75*.75*.75);

\draw[very thin] (2/27+2/81+2/243+2/729,0) -- (2/27+2/81+2/243+2/729,.75^5*.75*.75*.75*.75);

\draw[very thin] (2/9+2/81+2/243+2/729,0) -- (2/9+2/81+2/243+2/729,.75^5*.5*.5*.75*.75);

\draw[very thin] (2/9+2/27+2/243+2/729,0) -- (2/9+2/27+2/243+2/729,.75^5*.5*.75*.5*.75);
\draw[very thin] (2/9+2/27+2/81+2/729,0) -- (2/9+2/27+2/81+2/729,.75^4*.5*.75*.75*.5);
\draw[very thin] (2/9+2/27+2/81+2/243+2/729,0) -- (2/9+2/27+2/81+2/243+2/729,.75^5*.5*.75*.75*.75*.75);
\draw[very thin] (2/9+2/27+2/81+2/243,0) -- (2/9+2/27+2/81+2/243,.75^4*.5*.75*.75*.75);

\draw[very thin] (2/3+2/9+2/243+2/729,0) -- (2/3+2/9+2/243+2/729,.75^5*.75*.75*.75*.75);
\draw[very thin] (2/3+2/9+2/81+2/243+2/729,0) -- (2/3+2/9+2/81+2/243+2/729,.75^5*.75*.75*.5*.75*.75);
\draw[very thin] (2/3+2/9+2/81+2/243,0) -- (2/3+2/9+2/81+2/243,.75^4*.75*.75*.5*.75);
\draw[very thin] (2/3+2/9+2/81+2/729,0) -- (2/3+2/9+2/81+2/729,.75^4*.75*.75*.5*.5);

\begin{scope}[shift={(-1.5,0)}]

\node at (.3,.85) {\footnotesize $V^k_{n_1,\dots,n_k}$};

\draw[thin] (2/9,0) -- (2/9,2/3);

\draw[very thin] (2/9+2/27,0) -- (2/9+2/27,.75^2*.5*.75);
\draw[very thin] (2/9+2/81,0) -- (2/9+2/81,.75^2*.5*.5);
\draw[very thin] (2/9+2/243,0) -- (2/9+2/243,.75^2*.5*.75);
\draw[very thin] (2/9+2/729,0) -- (2/9+2/729,.75^2*.5*.5);

\draw[very thin] (2/9+2/27+2/81,0) -- (2/9+2/27+2/81,.75^3*.5*.75*.75);
\draw[very thin] (2/9+2/27+2/243,0) -- (2/9+2/27+2/243,.75^3*.5*.75*.5);
\draw[very thin] (2/9+2/27+2/729,0) -- (2/9+2/27+2/729,.75^3*.5*.75*.75);

\draw[very thin] (2/9+2/81+2/243,0) -- (2/9+2/81+2/243,.75^4*.5*.5*.75);
\draw[very thin] (2/9+2/81+2/729,0) -- (2/9+2/81+2/729,.75^4*.5*.5*.5);

\draw[very thin] (2/9+2/243+2/729,0) -- (2/9+2/243+2/729,.75^5*.5*.75*.75);

\draw[very thin] (2/9+2/81+2/243+2/729,0) -- (2/9+2/81+2/243+2/729,.75^5*.5*.5*.75*.75);

\draw[very thin] (2/9+2/27+2/243+2/729,0) -- (2/9+2/27+2/243+2/729,.75^5*.5*.75*.5*.75);
\draw[very thin] (2/9+2/27+2/81+2/729,0) -- (2/9+2/27+2/81+2/729,.75^4*.5*.75*.75*.5);
\draw[very thin] (2/9+2/27+2/81+2/243+2/729,0) -- (2/9+2/27+2/81+2/243+2/729,.75^5*.5*.75*.75*.75*.75);
\draw[very thin] (2/9+2/27+2/81+2/243,0) -- (2/9+2/27+2/81+2/243,.75^4*.5*.75*.75*.75);

\end{scope}

\draw[densely dotted] (-1/6,0) rectangle (7/6,.75);
\draw[densely dotted] (-1.5+2/9-1/54,0) rectangle (-1.5+1/3+1/54,.75^3*.5);

\draw[->,very thin] (-1.5+1/3+1/54+.02,.75^3*.5^2) -- (-1/6-.02,.375);

\draw[->,very thin] (-1.5+2/9+.02,1/3+.75^3*.5^2) -- (-.02,.5+.375);

\end{tikzpicture}
\end{center}

Observe that $h_{x^k_{n_1,\dots,n_k}}$ maps the leftmost blade in in $V^k_{n_1,\dots,n_k}$, namely $B_{x^k_{n_1,\dots,n_k}}$, homeomorphically onto $\{0\} \times [0,1]$ via a piecewise-linear homeomorphism. 

If $B_{x^\ell_{m_1,\dots,m_\ell}}$ is any other blade in $V^k_{n_1,\dots,n_k}$, then as observed above, 
$k \leq \ell$ and $(n_1,n_2,\dots,n_k) = (m_1,m_2,\dots,m_k)$. 
But then 
\begin{align*}
h^k_{n_1,\dots,n_k}(x^\ell_{m_1,\dots,m_\ell},e^\ell_{m_1,\dots,m_\ell}) &\,=\, \textstyle \Bigg( 3^{\sum_{i=1}^k n_i} \Big( x^\ell_{m_1,\dots,m_\ell} - x^k_{n_1,\dots,n_k} \Big) \,,\, \frac{M}{m^k_{n_1,\dots,n_k}} e^\ell_{m_1,\dots,m_k} \Bigg) \\
&\,=\, \textstyle \Bigg( 3^{\sum_{i=1}^k n_i} \bigg( 2\cdot\sum_{i=1}^\ell \left(\frac{1}{3}\right)^{\sum_{j=1}^im_j} \,-\, 2\cdot\sum_{i=1}^k \left(\frac{1}{3}\right)^{\sum_{j=1}^in_j} \bigg) \,,\,  \\
& \textstyle \ \ \qquad \qquad \qquad \frac{M}{M^{\sum_{i=1}^{k}n_i}\cdot M \cdot \prod_{i=1}^ky_{n_i}}M^{\sum_{i=1}^{\ell-1}m_i}\cdot \prod_{i=1}^\ell y_{m_i} \Bigg) \\
&\,=\, \textstyle \Bigg( 2\cdot\sum_{i=k+1}^\ell \left(\frac{1}{3}\right)^{\sum_{j=k+1}^im_j} \,,\, 
M^{\sum_{i=k+1}^{\ell-1}m_i}\cdot \prod_{i=k+1}^\ell y_{m_i} \Bigg) \\
&\,=\, \Big( x^{\ell-k}_{m_{k+1},\dots,m_{\ell}},e^{\ell-k}_{m_{k+1},\dots,m_{\ell}} \Big).
\end{align*}
In other words, $h^k_{n_1,\dots,n_k}$ maps the tip of the blade at $x^\ell_{m_1,\dots,m_\ell}$ to the tip of the blade at $x^{\ell-k}_{m_{k+1},\dots,m_{\ell}}$. Because $h^k_{n_1,\dots,n_k}$ is just a scaling in the $Y$-coordinate, this means that $h^k_{n_1,\dots,n_k}$ maps the blade $B_{x^\ell_{m_1,\dots,m_\ell}}$ in $V^k_{n_1,\dots,n_k}$ to the blade $B_{x^{\ell-k}_{m_{k+1},\dots,m_{\ell}}}$ in $C_{>0}$. 
Thus every blade of $V^k_{n_1,\dots,n_k}$ maps to some blade of $C$, no two blades of $V^k_{n_1,\dots,n_k}$ map to the same blade of $C$, and every blade of $C$ has some blade of $V^k_{n_1,\dots,n_k}$ mapping onto it. 
Therefore $h^k_{n_1,\dots,n_k}$ is a bijection. 
And $h^k_{n_1,\dots,n_k}$ is clearly continuous in both directions, so $h^k_{n_1,\dots,n_k}$ is a homeomorphism as claimed. 
Furthermore, $h^k_{n_1,\dots,n_k}$ maps the endpoint of $B_{x^k_{n_1,\dots,n_k}}$ to $(0,1)$.
\hfill $\dashv$
\vspace{3mm}

Using this claim, we now show $F_C$ is endpoint-homogeneous. 
Suppose $\bar x_1 = x^k_{n_1,\dots,n_k}$ and $\bar x_2 = x^\ell_{m_1,\dots,m_\ell}$ are two different points at which $C$ has a blade. By the claim above, there is a homeomorphism 
$$h \,=\, \big( h^\ell_{m_1,\dots,m_\ell} \big)^{-1} \circ h^k_{n_1,\dots,n_k}:\, V^k_{n_1,\dots,n_k} \to V^\ell_{m_1,\dots,m_\ell}$$ 
that maps the tip of $B_{\bar x_1}$ to the tip of $B_{\bar x_2}$. 
Because $h$ is a homeomorphism and $\bar x_1 \neq \bar x_2$, there is a clopen neighborhood $W$ of $B_{\bar x_1}$ such that $W \cap h[W] = \0$. Shrinking $W$ if necessary, we may (and do) assume that $W$ has the form $V \times (0,1]$ for some basic clopen $V \sub K$. 
This implies $h[W]$ is also a clopen subset of $K$. 
Now define a homeomorphism $H: F_C \to F_C$ as follows:
$$H(z) \,=\, \begin{cases}
h(z) &\text{ if } z \in W, \\
h^{-1}(z) &\text{ if } z \in h[W], \\
z & \text{ otherwise.}
\end{cases}$$
Because $W$ and $h[W]$ are both relatively clopen subsets of $F_C \setminus \{t\}$, it is clear that $H \rest (F_C \setminus \{t\})$ is a homeomorphism $F_C \setminus \{t\} \to F_C \setminus \{t\}$. But then this homeomorphism extends (uniquely) to the one-point compactification of $F_C \setminus \{t\}$, namely $F_C$, by mapping $t$ to $t$; this extension is $H$. So $H$ is a homeomorphism $F_C \to F_C$ mapping the tip of the blade at $\bar x_1$ to the tip of the blade at $\bar x_2$. 
Thus, given two arbitrary points $\bar x_1$ and $\bar x_2$ at which $C$ has a blade, we have found a homeomorphism $F_C \to F_C$ mapping the corresponding endpoints to one another. 
Hence $F_C$ is endpoint-homogeneous.

It remains only to check that $F_C$ is $X$-EPG. Because we already know $F_C$ is endpoint-homogeneous (so all blades are topologically indistinguishable in the sense of Lemma~\ref{lem:blades}), it suffices to show for a single $B \in \B(F_C)$ that there is a homeomorphism $\phi: I \to B$ such that $\phi(0) = t$ and $\phi[X] = \closure{E(F_C) \setminus B}$. 
Let $B = \{0\} \times I$, the leftmost blade of $C$, and let $\phi$ be the natural homeomorphism $\phi(z) = (0,z)$. 

Because the tip of the blade at $x^k_{n_1,\dots,n_k}$ has height $\leq\! M^k$, it is clear that $\closure{\mathcal E(C)}$ contains points of $I \times \{0\}$ (in $C$), which means $t \in \closure{E(F_c) \setminus B}$ (in $F_C$). Hence $\phi(0) = t \in \closure{E(F_c) \setminus B}$.

Next let $x \in X \setminus \{0\}$. 
By our choice of the sequence $\seq{y_n}{n \in \N}$ in the construction, there is a subsequence $\seq{y_{n_k}}{k \in \N}$ such that $\lim_{k \to \infty}y_{n_k} = x$. But then
$\seq{(x^1_{n_k},e^1_{n_k})}{k \in \N} = \seq{(x^1_{n_k},y_{n_k})}{k \in \N}$ 
is a sequence of points in $\mathcal E(C) \setminus B_0$ converging to $(0,x) = \phi(x)$. 
Combined with the previous paragraph, this shows $\phi[X] \sub \closure{E(F_C) \setminus B}$.

For the reverse inclusion, fix some $z \in B \cap \closure{E(F_C) \setminus B}$. Either $z = t$, or else $z = (0,y_z)$ for some $y_z \in [0,1]$. Let $\seq{(\bar x_k,\bar e_k)}{k \in \N}$ be a sequence in $E(F_C) \setminus B = \mathcal E(C) \setminus B$ converging to $z$. 
This implies $\lim_{k \to \infty} \bar x_k = 0$, and we would like to show 
$\lim_{k \to \infty} \bar e_k \in X$, since this would mean 
$z = (0,\lim_{k \to \infty} \bar e_k) \in \phi[X]$. 
We consider two cases. 

For the first case, suppose there are infinitely many $k$ such that $\bar x_k = x^1_{n_k}$ for some $n_k$. Then by construction, $(\bar x_k,\bar e_k) = (x^1_{n_k},y_{n_k})$ for infinitely many $k$. 
But then $\lim_{k \to \infty} \bar y_n \in \closure{\set{y_n}{n \in \N}} = X$. 

For the second case, suppose there are only finitely many $k$ such that $\bar x_k = x^1_{n_k}$. 
So, for all sufficiently large $k$, $\bar x_k = x^\ell_{n^k_1,n^k_2,\dots,n^k_\ell}$ for some $\ell > 1$ and some $n^k_1,n^k_2,\dots,n^k_\ell \in \N$. 
But $\bar x_k = x^\ell_{n^k_1,n^k_2,\dots,n^k_\ell} \geq \frac{2}{3^{n^k_1}}$ by construction. Because $\lim_{k \to \infty} \bar x_k = 0$, this implies $\lim_{k \to \infty}n^k_1 = \infty$. 
But also $\bar e_k \leq M^{n^1_k}$ by construction, so this implies $\lim_{k \to \infty} \bar e_k = 0 \in X$. 

In either case $z \in \phi[X]$, so $\phi[X] \supseteq B \cap \closure{E(F_C) \setminus B}$ as claimed. Hence $\phi[X] = B \cap \closure{E(F_C) \setminus B}$, and $F_C$ is an $X$-EPG fan.
\end{proof}

This concludes the proof that for sets $X$ of type $(1)$ or $(2)$ in Corollary~\ref{cor:list0}, there is an endpoint-homogeneous, $X$-EPG smooth fan. 
It remains to show that the same is true for $X$ of type $(3)$. 
As it turns out, this is a relatively easy consequence of the previous theorem.

\begin{theorem}\label{thm:add1}
Suppose $X$ is a closed subset of $[0,1]$ with $0,1 \in X$, and $1$ is an isolated point of $X$. 
Then there is an endpoint-homogeneous, $X$-EPG smooth fan. 
\end{theorem}
\begin{proof}
Because $1$ is an isolated point of $X$, $Y = X \setminus \{1\}$ is a closed subset of $[0,1]$ with $1 \notin Y$. By the previous theorem, there is an endpoint-homogeneous, $Y$-EPG smooth fan $F_Y$. Let $t_0$ denote the top of $F_Y$. 

Let $F$ denote the quotient space obtained from $F_Y \times K$ (where $K$ denotes the Cantor set) by collapsing $\{t_0\} \times K$ to a single point. (Equivalently, one may view $F$ as the one-point compactification of $(F_Y \setminus \{t_0\}) \times K$.) 
Let $t$ denote the top of $F$. 
For convenience, we identify $F \setminus \{t\}$ with $(F_Y \setminus \{t_0\}) \times K$.

Because $F_Y$ is a smooth fan, it is a subcontinuum of the Cantor fan $F_K$ (with top $t_K$). This means $F$ is a subset of the quotient space obtained from $F_K \times K$ by collapsing $\{t_K\} \times K$ to a single point (or, equivalently, the one-point compactification of $(F_K \setminus \{t_K\}) \times K$). This latter space is homeomorphic to $F_K$, because $K \times K \homeo K$. Thus $F$ is a subset of a Cantor fan. 

It is straightforward to check that $F$ is path-connected, compact, and contains a copy of $F_Y$. Combined with the previous paragraph, this means $F$ is a smooth fan. Furthermore, it is clear that $E(F) = E(F_Y) \times K$.

Let $e,e' \in E(F)$. Then $e = (e_1,e_2)$ and $e' = (e_1',e_2')$, where $e_1,e_1' \in E(F_Y)$ and $e_2,e_2' \in K$. Because $F_Y$ is endpoint-homogeneous, there is a homeomorphism $h_1: F_Y \to F_Y$ with $h_1(e_1) = e_1'$. Because $K$ is homogeneous, there is a homeomorphism $h_2: K \to K$ with $h_2(e_2) = e_2'$. But then
$$H(z) \,=\, \begin{cases}
(h_1(x),h_2(y)) &\text{ if } z = (x,y) \in (F_Y \setminus \{t_0\}) \times K, \\
t &\text{ if } z=t
\end{cases}$$
is a homeomorphism $F \to F$ with $H(e) = e'$. So $F$ is endpoint-homogeneous. 

Lastly, we must show that $F$ is an $X$-EPG fan. 
Fix $B \in \B(F)$. 
Observe that $B = B_0 \times \{c\}$ for some $B_0 \in \B(F_Y)$ and $c \in K$, provided that we are willing to abuse notation slightly and identify $(t_0,c)$ with $t$. 
Because $F_Y$ is $Y$-EPG, there is a homeomorphism $\phi: I \to B$ with $\phi(0) = t$ such that $\phi[Y] = B \cap \closure{E(F_Y) \setminus B}$. We claim that $\phi[Y] = B \cap \closure{E(F) \setminus B}$.

On the one hand, $F_Y \times \{c\}$ is a homeomorphic copy of $F_Y$ sitting inside of $F$, and the endpoints of this copy of $F_Y$ are all endpoints of $F$. 
Therefore 
$$B \cap \closure{E(F) \setminus B} \,\supseteq\, B \cap \closure{E(F_Y \times \{c\}) \setminus B} \,=\, \phi[Y].$$
But also $\phi(1)$, the endpoint of $B$, is in $\closure{E(F) \setminus B}$, because it is a limit point of $\{\phi(1)\} \times K \sub E(F)$. Hence $B \cap \closure{E(F) \setminus B} \supseteq \phi[Y] \cup \{\phi(1)\} = \phi[X]$. 

On the other hand, suppose $x \in B$ and $x \notin \phi[X]$. Fix $x_0 \in B_0$ such that $x = (x_0,c)$. Then $x_0 \notin \phi[Y]$, and $x_0$ is not the endpoint of $B_0$ (because $x \notin \phi[X]$), so there is an open neighborhood $U \sub F_Y$ with $x_0 \in U$ such that $U \cap E(F_Y) = \0$. But then $U \times K$ is an open neighborhood of $(x_0,c)$ in $F$ such that $(U \times K) \cap E(F) = (U \times K) \cap (E(F_Y) \times K) = \0$. Hence $x \notin \closure{E(F) \setminus B}$.
\end{proof}

\begin{corollary}\label{cor:list1}
Let $X \sub [0,1]$. The following are equivalent:
\begin{enumerate}
\item There is an endpoint-homogeneous, $X$-EPG smooth fan.
\item There is an $X$-EPG smooth fan.
\item $X$ satisfies one of the conditions listed in Corollary~\ref{cor:list0}.
\end{enumerate}
\end{corollary}

\begin{theorem}
There are uncountably many distinct homeomorphism types of endpoint-homogeneous smooth fans.
\end{theorem}
\begin{proof}
In light of the previous Corollary~\ref{cor:list1} and Observation~\ref{obs:EE}, it suffices to show that there is an uncountable family of subsets of $[0,1]$ such that every member of the family satisfies the three statements in Corollary~\ref{cor:list1}, but no two of them are equivalently embedded in $[0,1]$. 

For each $n \in \N$, fix some $a_n,b_n$ such that $\frac{1}{n+1} < a_n < b_n < \frac{1}{n}$. 
Then, for each $A \sub \N$, let 
$$X_A \,=\, \textstyle \{0\} \cup \set{\frac{1}{n}}{n \in \N} \cup \set{[a_n,b_n] \vphantom{f^f}}{n \in A}.$$
The family of sets $\set{X_A}{A \sub \N}$ has the required properties.
\end{proof}

%%%%%%%%%%%%
\section{$\frac{1}{n}$-homogeneity}
%%%%%%%%%%%%

Recall from the introduction that a space $X$ is \emph{$\frac{1}{n}$-homogeneous} if there is a partition of $X$ into $n$ sets (but no fewer) such that if $x,y \in X$ lie in the same partition piece, then there is a homeomorphism $h: X \to X$ with $h(x) = y$. 

In \cite{AHJ}, Acosta, Hoehn, and Ju\'arez completely classify all $\frac{1}{3}$-homogeneous smooth fans: they are the simple $n$-ods, the Cantor fan, the Lelek fan, and the $\{0\}$-EPG and $\{0,1\}$-EPG fans constructed in the previous section.

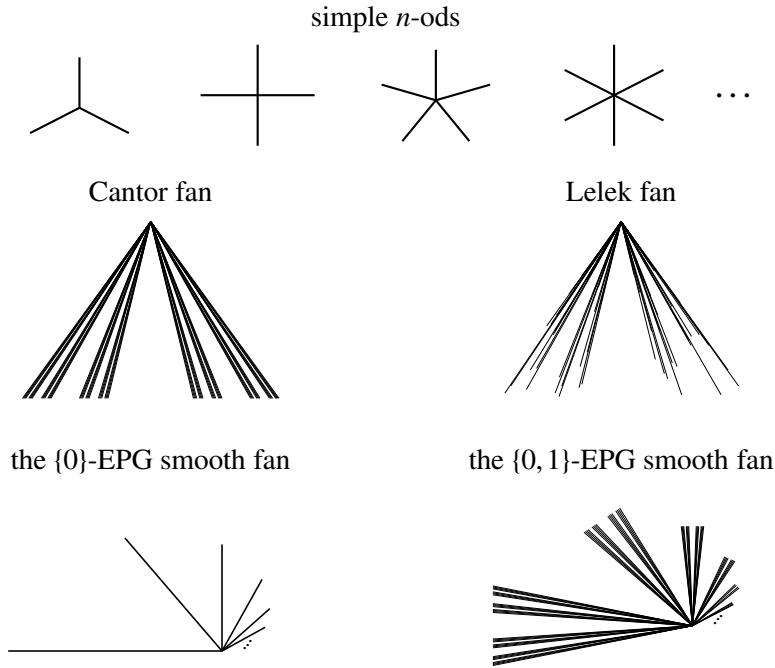
\begin{center}
\begin{tikzpicture}[xscale=.375,yscale=.335]

\draw[very thin] (4.5,7) -- (0,0);
\draw[ultra thin] (4.5,7) -- (1/27,0);
\draw[ultra thin] (4.5,7) -- (2/27,0);
\draw[very thin] (4.5,7) -- (1/9,0);
\draw[very thin] (4.5,7) -- (2/9,0);
\draw[ultra thin] (4.5,7) -- (7/27,0);
\draw[ultra thin] (4.5,7) -- (8/27,0);
\draw[very thin] (4.5,7) -- (1/3,0);
\draw[very thin] (4.5,7) -- (2/3,0);
\draw[ultra thin] (4.5,7) -- (19/27,0);
\draw[ultra thin] (4.5,7) -- (20/27,0);
\draw[very thin] (4.5,7) -- (7/9,0);
\draw[very thin] (4.5,7) -- (8/9,0);
\draw[ultra thin] (4.5,7) -- (25/27,0);
\draw[ultra thin] (4.5,7) -- (26/27,0);
\draw[very thin] (4.5,7) -- (1,0);
\draw[very thin] (4.5,7) -- (2,0);
\draw[ultra thin] (4.5,7) -- (55/27,0);
\draw[ultra thin] (4.5,7) -- (56/27,0);
\draw[very thin] (4.5,7) -- (19/9,0);
\draw[very thin] (4.5,7) -- (20/9,0);
\draw[ultra thin] (4.5,7) -- (61/27,0);
\draw[ultra thin] (4.5,7) -- (62/27,0);
\draw[very thin] (4.5,7) -- (7/3,0);
\draw[very thin] (4.5,7) -- (8/3,0);
\draw[ultra thin] (4.5,7) -- (73/27,0);
\draw[ultra thin] (4.5,7) -- (74/27,0);
\draw[very thin] (4.5,7) -- (25/9,0);
\draw[very thin] (4.5,7) -- (26/9,0);
\draw[ultra thin] (4.5,7) -- (79/27,0);
\draw[ultra thin] (4.5,7) -- (80/27,0);
\draw[very thin] (4.5,7) -- (3,0);
\draw[very thin] (4.5,7) -- (6,0);
\draw[ultra thin] (4.5,7) -- (163/27,0);
\draw[ultra thin] (4.5,7) -- (164/27,0);
\draw[very thin] (4.5,7) -- (55/9,0);
\draw[very thin] (4.5,7) -- (56/9,0);
\draw[ultra thin] (4.5,7) -- (169/27,0);
\draw[ultra thin] (4.5,7) -- (170/27,0);
\draw[very thin] (4.5,7) -- (19/3,0);
\draw[very thin] (4.5,7) -- (20/3,0);
\draw[ultra thin] (4.5,7) -- (181/27,0);
\draw[ultra thin] (4.5,7) -- (182/27,0);
\draw[very thin] (4.5,7) -- (61/9,0);
\draw[very thin] (4.5,7) -- (62/9,0);
\draw[ultra thin] (4.5,7) -- (187/27,0);
\draw[ultra thin] (4.5,7) -- (188/27,0);
\draw[very thin] (4.5,7) -- (7,0);
\draw[very thin] (4.5,7) -- (8,0);
\draw[ultra thin] (4.5,7) -- (217/27,0);
\draw[ultra thin] (4.5,7) -- (218/27,0);
\draw[very thin] (4.5,7) -- (73/9,0);
\draw[very thin] (4.5,7) -- (74/9,0);
\draw[ultra thin] (4.5,7) -- (223/27,0);
\draw[ultra thin] (4.5,7) -- (224/27,0);
\draw[very thin] (4.5,7) -- (25/3,0);
\draw[very thin] (4.5,7) -- (26/3,0);
\draw[ultra thin] (4.5,7) -- (235/27,0);
\draw[ultra thin] (4.5,7) -- (236/27,0);
\draw[very thin] (4.5,7) -- (79/9,0);
\draw[very thin] (4.5,7) -- (80/9,0);
\draw[ultra thin] (4.5,7) -- (241/27,0);
\draw[ultra thin] (4.5,7) -- (242/27,0);
\draw[very thin] (4.5,7) -- (9,0);

\begin{scope}[shift={(21,7)}]
\draw[very thin] (0,0) -- (-4.5*.31,-7*.31);
\draw[very thin] (0,0) -- (-4.5*.41+1/27*.41,-7*.41);
\draw[very thin] (0,0) -- (-4.5*.59+2/27*.59,-7*.59);
\draw[very thin] (0,0) -- (-4.5*.26+1/9*.26,-7*.26);
\draw[very thin] (0,0) -- (-4.5*.53+2/9*.53,-7*.53);
\draw[very thin] (0,0) -- (-4.5*.58+7/27*.58,-7*.58);
\draw[very thin] (0,0) -- (-4.5*.97+8/27*.97,-7*.97);
\draw[very thin] (0,0) -- (-4.5*.93+1/3*.93,-7*.93);
\draw[very thin] (0,0) -- (-4.5*.23+2/3*.23,-7*.23);
\draw[very thin] (0,0) -- (-4.5*.84+19/27*.84,-7*.84);
\draw[very thin] (0,0) -- (-4.5*.62+20/27*.62,-7*.62);
\draw[very thin] (0,0) -- (-4.5*.64+7/9*.64,-7*.64);
\draw[very thin] (0,0) -- (-4.5*.33+8/9*.33,-7*.33);
\draw[very thin] (0,0) -- (-4.5*.83+25/27*.83,-7*.83);
\draw[very thin] (0,0) -- (-4.5*.27+26/27*.27,-7*.27);
\draw[very thin] (0,0) -- (-4.5*.95+1*.95,-7*.95);
\draw[very thin] (0,0) -- (-4.5*.02+2*.02,-7*.02);
\draw[very thin] (0,0) -- (-4.5*.88+55/27*.88,-7*.88);
\draw[very thin] (0,0) -- (-4.5*.41+56/27*.41,-7*.41);
\draw[very thin] (0,0) -- (-4.5*.97+19/9*.97,-7*.97);
\draw[very thin] (0,0) -- (-4.5*.16+20/9*.16,-7*.16);
\draw[very thin] (0,0) -- (-4.5*.93+61/27*.93,-7*.93);
\draw[very thin] (0,0) -- (-4.5*.99+62/27*.99,-7*.99);
\draw[very thin] (0,0) -- (-4.5*.37+7/3*.37,-7*.37);
\draw[very thin] (0,0) -- (-4.5*.51+8/3*.51,-7*.51);
\draw[very thin] (0,0) -- (-4.5*.05+73/27*.05,-7*.05);
\draw[very thin] (0,0) -- (-4.5*.82+74/27*.82,-7*.82);
\draw[very thin] (0,0) -- (-4.5*.09+25/9*.09,-7*.09);
\draw[very thin] (0,0) -- (-4.5*.74+26/9*.74,-7*.74);
\draw[very thin] (0,0) -- (-4.5*.94+79/27*.94,-7*.94);
\draw[very thin] (0,0) -- (-4.5*.45+80/27*.45,-7*.45);
\draw[very thin] (0,0) -- (-4.5*.92+3*.92,-7*.92);
\draw[very thin] (0,0) -- (-4.5*.3+6*.3,-7*.3);
\draw[very thin] (0,0) -- (-4.5*.78+163/27*.78,-7*.78);
\draw[very thin] (0,0) -- (-4.5*.16+164/27*.16,-7*.16);
\draw[very thin] (0,0) -- (-4.5*.4+55/9*.4,-7*.4);
\draw[very thin] (0,0) -- (-4.5*.62+56/9*.62,-7*.62);
\draw[very thin] (0,0) -- (-4.5*.86+169/27*.86,-7*.86);
\draw[very thin] (0,0) -- (-4.5*.2+170/27*.2,-7*.2);
\draw[very thin] (0,0) -- (-4.5*.89+19/3*.89,-7*.89);
\draw[very thin] (0,0) -- (-4.5*.98+20/3*.98,-7*.98);
\draw[very thin] (0,0) -- (-4.5*.62+181/27*.62,-7*.62);
\draw[very thin] (0,0) -- (-4.5*.8+182/27*.8,-7*.8);
\draw[very thin] (0,0) -- (-4.5*.34+61/9*.34,-7*.34);
\draw[very thin] (0,0) -- (-4.5*.82+62/9*.82,-7*.82);
\draw[very thin] (0,0) -- (-4.5*.53+187/27*.53,-7*.53);
\draw[very thin] (0,0) -- (-4.5*.42+188/27*.42,-7*.42);
\draw[very thin] (0,0) -- (-4.5*.11+7*.11,-7*.11);
\draw[very thin] (0,0) -- (-4.5*.7+8*.7,-7*.7);
\draw[very thin] (0,0) -- (-4.5*.67+217/27*.67,-7*.67);
\draw[very thin] (0,0) -- (-4.5*.98+218/27*.98,-7*.98);
\draw[very thin] (0,0) -- (-4.5*.21+73/9*.21,-7*.21);
\draw[very thin] (0,0) -- (-4.5*.48+74/9*.48,-7*.48);
\draw[very thin] (0,0) -- (-4.5*.08+223/27*.08,-7*.08);
\draw[very thin] (0,0) -- (-4.5*.65+224/27*.65,-7*.65);
\draw[very thin] (0,0) -- (-4.5*.13+25/3*.13,-7*.13);
\draw[very thin] (0,0) -- (-4.5*.28+26/3*.28,-7*.28);
\draw[very thin] (0,0) -- (-4.5*.23+235/27*.23,-7*.23);
\draw[very thin] (0,0) -- (-4.5*.06+236/27*.06,-7*.06);
\draw[very thin] (0,0) -- (-4.5*.64+79/9*.64,-7*.64);
\draw[very thin] (0,0) -- (-4.5*.7+80/9*.7,-7*.7);
\draw[very thin] (0,0) -- (-4.5*.93+241/27*.93,-7*.93);
\draw[very thin] (0,0) -- (-4.5*.84+242/27*.84,-7*.84);
\draw[very thin] (0,0) -- (-4.5*.46+9*.46,-7*.46);

\end{scope}

\node at (4.5,8.2) {\small Cantor fan};

\node at (21,8.2) {\small Lelek fan};

\node at (12.75,15) {\small simple $n$-ods};

\begin{scope}[shift={(2,12)}]

\begin{scope}[shift={(0,-.5)}]
\draw[thick] (0,0) -- (0,2);
\begin{scope}[rotate=120]
\draw[thick] (0,0) -- (0,2);
\end{scope}
\begin{scope}[rotate=240]
\draw[thick] (0,0) -- (0,2);
\end{scope}
\end{scope}

\begin{scope}[shift={(6.25,0)}]
\draw[thick] (0,0) -- (2,0);
\begin{scope}[rotate=90]
\draw[thick] (0,0) -- (2,0);
\end{scope}
\begin{scope}[rotate=180]
\draw[thick] (0,0) -- (2,0);
\end{scope}
\begin{scope}[rotate=270]
\draw[thick] (0,0) -- (2,0);
\end{scope}
\end{scope}

\begin{scope}[shift={(12.5,-.2)}]
\draw[thick] (0,0) -- (0,2);
\begin{scope}[rotate=72]
\draw[thick] (0,0) -- (0,2);
\end{scope}
\begin{scope}[rotate=144]
\draw[thick] (0,0) -- (0,2);
\end{scope}
\begin{scope}[rotate=216]
\draw[thick] (0,0) -- (0,2);
\end{scope}
\begin{scope}[rotate=288]
\draw[thick] (0,0) -- (0,2);
\end{scope}
\end{scope}

\begin{scope}[shift={(18.75,0)}]
\draw[thick] (0,0) -- (0,2);
\begin{scope}[rotate=60]
\draw[thick] (0,0) -- (0,2);
\end{scope}
\begin{scope}[rotate=120]
\draw[thick] (0,0) -- (0,2);
\end{scope}
\begin{scope}[rotate=180]
\draw[thick] (0,0) -- (0,2);
\end{scope}
\begin{scope}[rotate=240]
\draw[thick] (0,0) -- (0,2);
\end{scope}
\begin{scope}[rotate=300]
\draw[thick] (0,0) -- (0,2);
\end{scope}
\end{scope}

\node at (23,0) {\large $\dots$};

\end{scope}

\begin{scope}[shift={(7,-10)}]

\begin{scope}[rotate=90]
\draw[semithick] (0,0) -- (0,7.5); \end{scope}
\begin{scope}[rotate=37.28]
\draw[semithick] (0,0) -- (0,7.5*.75); \end{scope}
\begin{scope}[rotate=0]
\draw[semithick] (0,0) -- (0,7.5*.75^2); \end{scope}
\begin{scope}[rotate=-26.36]
\draw[semithick] (0,0) -- (0,7.5*.75^3); \end{scope}
\begin{scope}[rotate=-45]
\draw[semithick] (0,0) -- (0,7.5*.75^4); \end{scope}
\begin{scope}[rotate=-58.18]
\draw[semithick] (0,0) -- (0,7.5*.75^5); \end{scope}

\node at (.8,.1) {\tiny $.$};
\node at (.9,.25) {\tiny $.$};
\node at (1,.4) {\tiny $.$};

\node at (-2.5,7.5) {\small the $\{0\}$-EPG smooth fan};

\end{scope}

\begin{scope}[shift={(23.5,-9)}]

\begin{scope}[rotate=90]
\begin{scope}[xscale=.35,yscale=1]
\draw[very thin] (0,0) -- (-4.5+0,7);
\draw[very thin] (0,0) -- (-4.5+1/9,7);
\draw[very thin] (0,0) -- (-4.5+2/9,7);
\draw[very thin] (0,0) -- (-4.5+1/3,7);
\draw[very thin] (0,0) -- (-4.5+2/3,7);
\draw[very thin] (0,0) -- (-4.5+7/9,7);
\draw[very thin] (0,0) -- (-4.5+8/9,7);
\draw[very thin] (0,0) -- (-4.5+1,7);
\draw[very thin] (0,0) -- (-4.5+2,7);
\draw[very thin] (0,0) -- (-4.5+19/9,7);
\draw[very thin] (0,0) -- (-4.5+20/9,7);
\draw[very thin] (0,0) -- (-4.5+7/3,7);
\draw[very thin] (0,0) -- (-4.5+8/3,7);
\draw[very thin] (0,0) -- (-4.5+25/9,7);
\draw[very thin] (0,0) -- (-4.5+26/9,7);
\draw[very thin] (0,0) -- (-4.5+3,7);
\draw[very thin] (0,0) -- (-4.5+6,7);
\draw[very thin] (0,0) -- (-4.5+55/9,7);
\draw[very thin] (0,0) -- (-4.5+56/9,7);
\draw[very thin] (0,0) -- (-4.5+19/3,7);
\draw[very thin] (0,0) -- (-4.5+20/3,7);
\draw[very thin] (0,0) -- (-4.5+61/9,7);
\draw[very thin] (0,0) -- (-4.5+62/9,7);
\draw[very thin] (0,0) -- (-4.5+7,7);
\draw[very thin] (0,0) -- (-4.5+8,7);
\draw[very thin] (0,0) -- (-4.5+73/9,7);
\draw[very thin] (0,0) -- (-4.5+74/9,7);
\draw[very thin] (0,0) -- (-4.5+25/3,7);
\draw[very thin] (0,0) -- (-4.5+26/3,7);
\draw[very thin] (0,0) -- (-4.5+79/9,7);
\draw[very thin] (0,0) -- (-4.5+80/9,7);
\draw[very thin] (0,0) -- (-4.5+9,7);
\end{scope}
\end{scope}

\begin{scope}[rotate=37.28]
\begin{scope}[xscale=.35*.5,yscale=.75]
\draw[very thin] (0,0) -- (-4.5+0,7);
\draw[very thin] (0,0) -- (-4.5+1/3,7);
\draw[very thin] (0,0) -- (-4.5+2/3,7);
\draw[very thin] (0,0) -- (-4.5+1,7);
\draw[very thin] (0,0) -- (-4.5+2,7);
\draw[very thin] (0,0) -- (-4.5+7/3,7);
\draw[very thin] (0,0) -- (-4.5+8/3,7);
\draw[very thin] (0,0) -- (-4.5+3,7);
\draw[very thin] (0,0) -- (-4.5+6,7);
\draw[very thin] (0,0) -- (-4.5+19/3,7);
\draw[very thin] (0,0) -- (-4.5+20/3,7);
\draw[very thin] (0,0) -- (-4.5+7,7);
\draw[very thin] (0,0) -- (-4.5+8,7);
\draw[very thin] (0,0) -- (-4.5+25/3,7);
\draw[very thin] (0,0) -- (-4.5+26/3,7);
\draw[very thin] (0,0) -- (-4.5+9,7);
\end{scope}
\end{scope}

\begin{scope}[rotate=0]
\begin{scope}[xscale=.35*.5^2,yscale=.75^2]
\draw[ultra thin] (0,0) -- (-4.5+0,7);
\draw[ultra thin] (0,0) -- (-4.5+1/3,7);
\draw[ultra thin] (0,0) -- (-4.5+2/3,7);
\draw[ultra thin] (0,0) -- (-4.5+1,7);
\draw[ultra thin] (0,0) -- (-4.5+2,7);
\draw[ultra thin] (0,0) -- (-4.5+7/3,7);
\draw[ultra thin] (0,0) -- (-4.5+8/3,7);
\draw[ultra thin] (0,0) -- (-4.5+3,7);
\draw[ultra thin] (0,0) -- (-4.5+6,7);
\draw[ultra thin] (0,0) -- (-4.5+19/3,7);
\draw[ultra thin] (0,0) -- (-4.5+20/3,7);
\draw[ultra thin] (0,0) -- (-4.5+7,7);
\draw[ultra thin] (0,0) -- (-4.5+8,7);
\draw[ultra thin] (0,0) -- (-4.5+25/3,7);
\draw[ultra thin] (0,0) -- (-4.5+26/3,7);
\draw[ultra thin] (0,0) -- (-4.5+9,7);
\end{scope}
\end{scope}

\begin{scope}[rotate=-26.36]
\begin{scope}[xscale=.35*.5^3,yscale=.75^3]
\draw[ultra thin] (0,0) -- (-4.5+0,7);
\draw[ultra thin] (0,0) -- (-4.5+1,7);
\draw[ultra thin] (0,0) -- (-4.5+2,7);
\draw[ultra thin] (0,0) -- (-4.5+3,7);
\draw[ultra thin] (0,0) -- (-4.5+6,7);
\draw[ultra thin] (0,0) -- (-4.5+7,7);
\draw[ultra thin] (0,0) -- (-4.5+8,7);
\draw[ultra thin] (0,0) -- (-4.5+9,7);
\end{scope}
\end{scope}

\begin{scope}[rotate=-45]
\begin{scope}[xscale=.35*.5^4,yscale=.75^4]
\draw[ultra thin] (0,0) -- (-4.5+0,7);
\draw[ultra thin] (0,0) -- (-4.5+2.5,7);
\draw[ultra thin] (0,0) -- (-4.5+6.5,7);
\draw[ultra thin] (0,0) -- (-4.5+9,7);
\end{scope}
\end{scope}

\begin{scope}[rotate=-58.18]
\begin{scope}[xscale=.35*.5^5,yscale=.75^5]
\draw[ultra thin] (0,0) -- (-4.5+0,7);
\draw[ultra thin] (0,0) -- (-4.5+2.5,7);
\draw[ultra thin] (0,0) -- (-4.5+6.5,7);
\draw[ultra thin] (0,0) -- (-4.5+9,7);
\end{scope}
\end{scope}

\node at (.8,.1) {\tiny $.$};
\node at (.9,.25) {\tiny $.$};
\node at (1,.4) {\tiny $.$};

\node at (-2.5,6.5) {\small the $\{0,1\}$-EPG smooth fan};

\end{scope}

\end{tikzpicture}
\end{center}
\vspace{2mm}

Acosta, Hoehn, and Ju\'arez also prove in \cite{AHJ} (see Theorem 4.1) that no smooth fan is $\frac{1}{4}$-homogeneous. 
The main result of this section is that for every $n \geq 5$, there is an endpoint-homogeneous, $\frac{1}{n}$-homogeneous smooth fan. 
No new fans are constructed though: rather, we show that for certain choices of $X$ (depending on $n$), the $X$-EPG fans constructed in the proof of Theorem~\ref{thm:EPG} are $\frac{1}{n}$-homogeneous. 
The sets $X$ used for odd $n$ and for even $n$ are rather different. The odd case is simpler, and we handle it first.

\begin{theorem}
There is an endpoint-homogeneous, $\frac{1}{n}$-homogeneous smooth fan for all odd $n \geq 5$.
\end{theorem}
\begin{proof}
Let $n$ be an odd number $\geq\!5$, and let $m = \frac{n-3}{2}$. Fix some $a_1,a_2,\dots,a_m$ such that $0 < a_1 < \dots < a_m < 1$, and let $X = \{0,a_1,\dots,a_m\}$. We claim that the $X$-EPG fan constructed in the proof of Theorem~\ref{thm:EPG} is $\frac{1}{n}$-homogeneous. 
Let $C$ denote the comb of this fan, as constructed in the proof of Theorem~\ref{thm:EPG}, and let $F_C$ denote the fan itself. 
We know already that $F_C$ is an endpoint-homogeneous smooth fan, so we need only show it is $\frac{1}{n}$-homogeneous. 

Note that $X$ naturally divides $[0,1]$ into $n$ sets:
\begin{align*}
T &\,=\, \{0\}, \\
E &\,=\, \{1\}, \\
A_1 &\,=\, \{a_1\}, \ A_2 \,=\, \{a_2\}, \dots, \ A_m \,=\, \{a_m\}, \\  
D_1 &\,=\, (0,a_1), \ D_2 \,=\, (a_1,a_2), \dots, \ D_m \,=\, (a_{m-1,}a_m), \ D_{m+1} \,=\, (a_m,1). 
\end{align*}

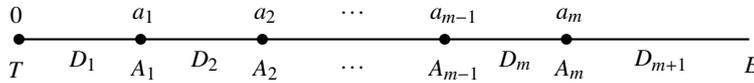
\begin{center}
\begin{tikzpicture}[scale=.8]

\draw[thick] (0,0) -- (12,0); 

\node at (0,-.01) {\small $\bullet$};
\node at (2,-.01) {\small $\bullet$};
\node at (4,-.01) {\small $\bullet$};
\node at (7,-.01) {\small $\bullet$};
\node at (9,-.01) {\small $\bullet$};

%\node at (1.3,.4) {\scriptsize $\cdots$};
%\node at (4.7,.4) {\scriptsize $\cdots$};

\node at (-.05,.4) {\scriptsize $0$};
\node at (2.05,.4) {\scriptsize $a_1$};
\node at (4.05,.4) {\scriptsize $a_2$};
\node at (5.5,.4) {\scriptsize $\cdots$};
\node at (7.15,.4) {\scriptsize $a_{m-1}$};
\node at (9.05,.4) {\scriptsize $a_m$};

\node at (-.05,-.5) {\scriptsize $T$};
\node at (2.05,-.5) {\scriptsize $A_1$};
\node at (4.05,-.5) {\scriptsize $A_2$};
\node at (5.5,-.5) {\scriptsize $\cdots$};
\node at (7.15,-.5) {\scriptsize $A_{m-1}$};
\node at (9.05,-.5) {\scriptsize $A_m$};
\node at (1.05,-.35) {\scriptsize $D_1$};
\node at (3.05,-.35) {\scriptsize $D_2$};
\node at (8.15,-.35) {\scriptsize $D_{m}$};
\node at (10.55,-.35) {\scriptsize $D_{m+1}$};
\node at (12.05,-.425) {\scriptsize $E$};

\end{tikzpicture}
\end{center}

For every $B \in \B(F_C)$, there is a homeomorphism $\phi_B: I \to B$ mapping $1$ to the tip of $B$ such that $\phi_B[X] = B \cap \closure{E(F_C) \setminus B}$. Partition $F_C$ into $n$ classes of points by taking
$$P_Z \,=\, \bigcup_{B \in \B(F_C)} \phi_B^{-1}[Z]$$
where $Z$ is any of the classes listed above, $T$, $E$, $A_1,\dots,A_m$, $D_1, \dots, D_{m+1}$. 
It is clear that the $P_Z$ partition $F_C$ into $n$ sets, and we claim that this partition witnesses the $\frac{1}{n}$-homogeneity of $F_C$.

First, it is not difficult to see that these classes can be defined topologically. 
For example, $T = \{t\}$ (the top of $F_C$) and $E = E(F_C)$. Given $i \leq m$, $A_i$ consists of all those points $x \in F_C$ that are limits of endpoints, and such that the arc $[t,x]$ contains $i+1$ limits of endpoints (counting $t$ and $x$). Given $i \leq m+1$, $D_i$ consists of all those points $x \in F_C$ that are not limits of endpoints, and such that the arc $[t,x]$ contains $i$ limits of endpoints.

Because the $P_Z$ admit this kind of topological definition, every homeomorphism $F_C \to F_C$ preserves each of the $P_Z$. It remains to show that if $x,y \in P_Z$ for some particular $Z$, then there is a homeomorphism $h: F_C \to F_C$ with $h(x) = y$.

This is obvious for $Z=T$, and it is true for $Z=E$ because $F_C$ is endpoint-homogeneous. 

Fix $i \leq m$, and let $x,y \in A_i$. Because $F_C$ is endpoint-homogeneous, there is a homeomorphism $h: F_C \to F_C$ mapping the blade containing $x$ to the blade containing $y$. But $P_{A_i}$ contains exactly one point of every blade, and $h$ preserves $P_{A_i}$, so this implies $h(x) = y$.

It remains to consider the $D_i$. Fix $i \leq m+1$. 
Instead of considering two arbitrary members of $D_i$, we begin by consider two members of $P_{D_i}$ on $B_0$, the leftmost blade of $C$.
Fix $z_1 = (0,y_1)$ and $z_2 = (0,y_2)$ in $P_{D_i} \cap B_0$. 
Because $[z_1,z_2] \sub P_{D_i}$ and $P_{D_i} \cap \closure{\mathcal E(C)} = \0$, there is an open $U \supseteq [z_1,z_2]$ such that $U \cap \mathcal E(C) = \0$. Shrinking $U$ if necessary (and relabelling $z_1$ and $z_2$ so that $z_1$ is between $t$ and $z_2$), we may (and do) assume that
$$U \,=\, V \times (y_1-\e \,,\, y_2+\e)$$
for some clopen $V \sub K$ containing $0$. 
There is a homeomorphism $h_0: I \to I$ such that $h_0(y_1) = y_2$ and $h_0(y) = y$ for all $y \notin (y_1-\e,y_2+\e)$. But then
$$h(z) \,=\, \begin{cases}
\big( x,h_0(y) \big) &\text{ if } z = (x,y) \in U, \\
z &\text{ if not}
\end{cases}$$
is a homeomorphism $F_C \to F_C$ mapping $z_1$ to $z_2$. 

Thus for any two points of $P_{D_i} \cap B_0$, there is a homeomorphism $F_C \to F_C$ mapping one to the other. 
But also, because $F_C$ is endpoint-homogeneous and $P_{D_i}$ is homeomorphism-invariant, for every $x \in P_{D_i}$ there is a homeomorphism $F_C \to F_C$ mapping it to a point of $B_0$. By composing homeomorphisms, it follows that if $x,y \in P_{D_i}$ then there is a homeomorphism $F_C \to F_C$ mapping $x$ to $y$.
\end{proof}

Observe that not all $\frac{1}{n}$-homogeneous fans need be smooth, and not all $\frac{1}{n}$-homogeneous smooth fans need be endpoint-homogeneous. 
The picture following Definition~\ref{def:smooth} shows two non-smooth fans; the one on the left is $\frac{1}{7}$-homogeneous and the one on the right is $\frac{1}{5}$-homogeneous. 
Drawn below is a $\frac{1}{5}$-homogeneous smooth fan that is not endpoint-homogeneous.

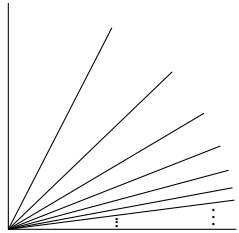
\vspace{1mm}

\begin{center}
\begin{tikzpicture}[scale=1.5]

\begin{scope}[rotate=90*.7^0]
\draw (0,0) -- (2,0); \end{scope}
\begin{scope}[rotate=90*.7^1]
\draw (0,0) -- (2,0); \end{scope}
\begin{scope}[rotate=90*.7^2]
\draw (0,0) -- (2,0); \end{scope}
\begin{scope}[rotate=90*.7^3]
\draw (0,0) -- (2,0); \end{scope}
\begin{scope}[rotate=90*.7^4]
\draw (0,0) -- (2,0); \end{scope}
\begin{scope}[rotate=90*.7^5]
\draw (0,0) -- (2,0); \end{scope}
\begin{scope}[rotate=90*.7^6]
\draw (0,0) -- (2,0); \end{scope}
\begin{scope}[rotate=90*.7^7]
\draw (0,0) -- (2,0); \end{scope}
\draw (0,0) -- (2,0); 

\node at (.95,.027) {\tiny $.$};
\node at (.95,.06) {\tiny $.$};
\node at (.95,.093) {\tiny $.$};

\node at (1.8,.05) {\tiny $.$};
\node at (1.8,.11) {\tiny $.$};
\node at (1.8,.17) {\tiny $.$};

\end{tikzpicture}
\end{center}

%We now prove the main theorem of this section.

\begin{theorem}\label{thm:1/6}
There is an endpoint-homogeneous, $\frac{1}{n}$-homogeneous smooth fan for all $n \geq 5$.
\end{theorem}
\begin{proof}
In light of the previous theorem, it suffices to show that there is an endpoint-homogeneous, $\frac{1}{n}$-homogeneous smooth fan for all even $n \geq 6$. 

Let $\set{p_i}{i \in \Z}$ be a subset of $[0,1]$ such that $p_i < p_j$ whenever $i < j$, and $\lim_{i \to -\infty} p_i = 0$, and $\lim_{i \to \infty} p_i < 1$. Let $a_1 = \lim_{i \to \infty} p_i$, and if $n \geq 8$ fix some $a_2 < \dots < a_m < 1$, where $m = \frac{n-4}{2}$. Let
$$X \,=\, \set{p_i}{i \in \Z} \cup \{0,a_1\} \cup \{a_2,\dots,a_m\}.$$
This set $X$ is illustrated below for $n = 10$.

\begin{center}
\begin{tikzpicture}[scale=.8]

\draw[thick] (0,0) -- (12,0); 

\node at (0,-.01) {\small $\bullet$};
\node at (32/729,-.01) {\scriptsize $\bullet$};
\node at (64/729,-.01) {\scriptsize $\bullet$};
\node at (96/729,-.01) {\scriptsize $\bullet$};
\node at (128/729,-.01) {\scriptsize $\bullet$};
\node at (64/243,-.01) {\scriptsize $\bullet$};
\node at (32/81,-.01) {\scriptsize $\bullet$};
\node at (16/27,-.01) {\scriptsize $\bullet$};
\node at (8/9,-.01) {\scriptsize $\bullet$};
\node at (4/3,-.01) {\scriptsize $\bullet$};
\node at (2,-.01) {\scriptsize $\bullet$};
\node at (3,-.01) {\scriptsize $\bullet$};
\node at (4,-.01) {\scriptsize $\bullet$};
\node at (6,-.01) {\small $\bullet$};
\node at (6-32/729,-.01) {\scriptsize $\bullet$};
\node at (6-64/729,-.01) {\scriptsize $\bullet$};
\node at (6-96/729,-.01) {\scriptsize $\bullet$};
\node at (6-128/729,-.01) {\scriptsize $\bullet$};
\node at (6-64/243,-.01) {\scriptsize $\bullet$};
\node at (6-32/81,-.01) {\scriptsize $\bullet$};
\node at (6-16/27,-.01) {\scriptsize $\bullet$};
\node at (6-8/9,-.01) {\scriptsize $\bullet$};
\node at (6-4/3,-.01) {\scriptsize $\bullet$};
\node at (6-2,-.01) {\scriptsize $\bullet$};
\node at (6-3,-.01) {\scriptsize $\bullet$};

\node at (3.05,.4) {\scriptsize $p_{{}_0}$};
\node at (2.15,.4) {\scriptsize $p_{{}_{-1}}$};
\node at (4.05,.4) {\scriptsize $p_{{}_1}$};
\node at (1.45,.4) {\scriptsize $p_{{}_{-2}}$};
\node at (4.7,.4) {\scriptsize $p_{{}_2}$};
\node at (6.15,.4) {\scriptsize $a_1$};
\node at (-.06,.4) {\scriptsize $0$};
\node at (.7,.4) {\scriptsize $\cdots$};
\node at (5.3,.4) {\scriptsize $\cdots$};

\node at (8,-.01) {\small $\bullet$};
\node at (8.05,.4) {\scriptsize $a_2$};

\node at (10,-.01) {\small $\bullet$};
\node at (10.05,.4) {\scriptsize $a_3$};

\end{tikzpicture}
\end{center}

We claim that the $X$-EPG fan constructed in the proof of Theorem~\ref{thm:EPG} is $\frac{1}{n}$-homogeneous. 
Let $C$ denote the comb of this fan, as constructed in the proof of Theorem~\ref{thm:EPG}, and let $F_C$ denote the fan itself. 
We know already that $F_C$ is an endpoint-homogeneous smooth fan, so we need only show it is $\frac{1}{n}$-homogeneous. 

Not unlike in the previous proof, $X$ naturally divides $[0,1]$ into $n$ sets:
\begin{align*}
T &\,=\, \{0\}, \\
E & \,=\, \{1\}, \\
L &\,=\, \set{p_i}{i \in \Z}, \\
G &\,=\, \textstyle \bigcup \set{(p_i,p_{i+1})}{i \in \Z} \\
A_1 &\,=\, \{a_1\}, \ A_2 \,=\, \{a_2\}, \ \dots, \ A_{m-1} \,=\, \{a_{m-1}\}, \ A_m \,=\, \{a_m\}, \\
D_1 &\,=\, (a_1,a_2), \ D_2 \,=\, (a_2,a_3), \ \dots, \ D_{m-1} \,=\, (a_{m-1},a_m), \ D_m \,=\, (a_m,1).  
\end{align*}

For every $B \in \B(F_C)$, there is a homeomorphism $\phi_B: I \to B$ mapping $1$ to the tip of $B$ such that $\phi_B[X] = B \cap \closure{E(F_C) \setminus B}$. Partition $F_C$ into $n$ classes of points by taking
$$P_Z \,=\, \bigcup_{B \in \B(F_C)} \phi_B^{-1}[Z],$$
where $Z$ is any of the classes listed above, $T$, $E$, $P$, $I$, $A_1,\dots,A_m$, $D_1,\dots,D_m$.
It is clear that these sets $P_Z$ form a partition $F_C$ into $n$ sets, and we claim that this partition witnesses the $\frac{1}{n}$-homogeneity of $F_C$.

As in the previous proof, it is not difficult to see that each of the partition pieces $P_Z$ admits a purely topological definition. Explicitly: $T = \{t\}$ is just the top of $F_C$; $E = E(F_C)$; $L$ is the set of all those $x$ that are a limit of endpoints, and such that $[t,x] \cap \closure{\mathcal E(C) \setminus [t,x]}$ is a decreasing convergent sequence;  $G$ is defined similarly, except that $x$ is not a limit of endpoints; $A_i$ is all those $x$ that are a limit of endpoints, and such that, if $e$ denotes the endpoint on the blade of $x$, then $[x,e]$ contains $m-i$ limits of endpoints; $D_i$ is defined similarly, except that $x$ is not a limit of endpoints.

Because each of the $P_Z$ admits a purely topological definition, every homeomorphism $F_C \to F_C$ preserves each of the $P_Z$. It remains to show that if $x,y \in P_Z$ for some particular $Z$, then there is a homeomorphism $h: F_C \to F_C$ with $h(x) = y$.

This is obvious for $Z = T$, and it is true for $Z=E$ because $F_C$ is endpoint-homogeneous. 
The case $Z = A_i$ (for any $i \leq m$) follows from the endpoint-homogeneity of $F_C$ as in the previous proof. 
To be explicit, fix $i \leq m$ and let $x,y \in P_{A_i}$. 
Because $F_C$ is endpoint-homogeneous, there is a homeomorphism $F_C \to F_C$ mapping the blade containing $x$ to the blade containing $y$. Because $P_{A_i}$ contains exactly one point of every blade, and $h$ preserves $P_{A_i}$, this homeomorphism must map $x$ to $y$.

Like with the $A_i$, the argument for the $D_i$ is essentially the same as in the proof of the previous theorem. 
Just as we did there, we begin here by consider two members of $P_{D_i}$ on $B_0$, the leftmost blade of $C$. 
Fix $z_1 = (0,y_1)$ and $z_2 = (0,y_2)$ in $P_{D_i} \cap B_0$. 
Because $[z_1,z_2] \sub P_{D_i}$ and $P_{D_i} \cap \closure{\mathcal E(C)} = \0$, there is an open neighborhood $U \sub C$ of $[z_1,z_2]$ such that $U \cap \mathcal E(C) = \0$. Shrinking $U$ if necessary (and relabelling $z_1$ and $z_2$ so that $z_1$ is between $t$ and $z_2$), we may (and do) assume that
$$U \,=\, V \times (y_1-\e \,,\, y_2+\e)$$
for some clopen $V \sub K$ containing $0$. 
There is a homeomorphism $h_0: I \to I$ such that $h_0(y_1) = y_2$ and $h_0(y) = y$ for all $y \notin (y_1-\e,y_2+\e)$. But then
$$h(z) \,=\, \begin{cases}
\big( x,h_0(y) \big) &\text{ if } z = (x,y) \in U, \\
z &\text{ if not}
\end{cases}$$
is a homeomorphism $F_C \to F_C$ mapping $z_1$ to $z_2$. 

Thus for any two points of $P_{D_i} \cap B_0$, there is a homeomorphism $F_C \to F_C$ mapping one to the other. 
But also, because $F_C$ is endpoint-homogeneous and $P_{D_i}$ is homeomorphism-invariant, for every $x \in P_{D_i}$ there is a homeomorphism $F_C \to F_C$ mapping $x$ to a point of $B_0$. By composing homeomorphisms, it follows that if $x,y \in P_{D_i}$ then there is a homeomorphism $F_C \to F_C$ mapping $x$ to $y$.

Observe that the first part of this argument for the $P_{D_i}$ (not the previous paragraph, but the one before it) shows a little more than what is claimed. An essentially identical argument shows: 
\begin{itemize}
\item[$(*)$] If $i \in \Z$ and $z_1,z_2 \in \{0\} \times (p_i,p_{i+1})$, then there is a homeomorphism $F_C \to F_C$ mapping $z_1$ to $z_2$. 
\end{itemize}
This fact will be used to examine $P_G$ later, but next we look at $P_L$. 

Like with the $D_i$, we begin here by consider two members of $P_L$ on $B_0$. 
Fix $(0,p_i)$ and $(0,p_j)$ in $P_L \cap B_0$, with $i < j$. 
Let $\phi: I \to I$ be the homeomorphism such that 
\begin{itemize}
\item[$\circ$] $\phi(0) = 0$ and $\phi(z) = z$ for all $z \geq a_1$.
\item[$\circ$] $\phi(p_k) = p_{k+j-i}$ for all $k \in \Z$, and
\item[$\circ$] $\phi$ maps $(p_k,p_{k+1})$ linearly onto $(p_{k+j-i},p_{k+j-i+1})$ for all $k \in \Z$.
\end{itemize}
Note, in particular, that $\phi[X] = X$ and $\phi(p_i) = p_j$. We aim to find a homeomorphism $h: F_C \to F_C$ such that $h(0,y) = (0,\phi(y))$ for all $y > 0$ (in other words, $h$ restricts to $\phi$ on $B_0 \setminus \{t\}$).

In the proof of Theorem~\ref{thm:EPG}, our construction begins with a countable dense $D \sub X$. 
Because each of the $p_k$ is isolated in $X$, $\set{p_k}{k \in \Z} \sub D$. 
The construction proceeds by fixing an infinite sequence $\seq{y_n}{n \in \N}$ such that each member of $D$ appears infinitely often in the sequence. In particular, 
$$Y_k \,=\, \set{n}{y_n = p_k}$$
is infinite for every $k \in \N$. Let $\psi: \N \to \N$ be a permutation of $\N$ such that 
\begin{itemize}
\item[$\circ$] $\psi[Y_k] = Y_{k+j-i}$ for all $k \in \Z$, and
\item[$\circ$] $\psi$ acts as the identity map on $\N \setminus \bigcup_{k \in \Z}Y_k$.
\end{itemize}
Observe that $\phi(y_n) = y_{\psi(n)}$ for all $n \in \N$.

Recalling more notation from the proof of Theorem~\ref{thm:EPG}, $C_{>0} = C \cap (0,1]$, and for each $n \in \N$ there is a relatively clopen subset $V^1_n$ of $C_{>0}$ containing the blade at $x^1_n = \nicefrac{2}{3^n}$. Furtheremore, the $V^1_n$ form a partition of $C_{>0} \setminus B_0$. 

Recall that $M$ denotes the largest member of $X$, and $M^{n+1}y_n$ is the $Y$-coordinate of the highest point on the leftmost blade of $V^1_n$ that is a limit of endpoints. We divide $V^1_n$ into two pieces overlapping at one point: 

\begin{center}
\begin{tikzpicture}[xscale=4,yscale=3]

%\node at (.3,.85) {\footnotesize $V^k_{n_1,\dots,n_k}$};

\draw[thin] (2/9,0) -- (2/9,2/3);

\draw[very thin] (2/9+2/27,0) -- (2/9+2/27,.75^2*.5*.75);
\draw[very thin] (2/9+2/81,0) -- (2/9+2/81,.75^2*.5*.5);
\draw[very thin] (2/9+2/243,0) -- (2/9+2/243,.75^2*.5*.75);
\draw[very thin] (2/9+2/729,0) -- (2/9+2/729,.75^2*.5*.5);

\draw[very thin] (2/9+2/27+2/81,0) -- (2/9+2/27+2/81,.75^3*.5*.75*.75);
\draw[very thin] (2/9+2/27+2/243,0) -- (2/9+2/27+2/243,.75^3*.5*.75*.5);
\draw[very thin] (2/9+2/27+2/729,0) -- (2/9+2/27+2/729,.75^3*.5*.75*.75);

\draw[very thin] (2/9+2/81+2/243,0) -- (2/9+2/81+2/243,.75^4*.5*.5*.75);
\draw[very thin] (2/9+2/81+2/729,0) -- (2/9+2/81+2/729,.75^4*.5*.5*.5);

\draw[very thin] (2/9+2/243+2/729,0) -- (2/9+2/243+2/729,.75^5*.5*.75*.75);

\draw[very thin] (2/9+2/81+2/243+2/729,0) -- (2/9+2/81+2/243+2/729,.75^5*.5*.5*.75*.75);

\draw[very thin] (2/9+2/27+2/243+2/729,0) -- (2/9+2/27+2/243+2/729,.75^5*.5*.75*.5*.75);
\draw[very thin] (2/9+2/27+2/81+2/729,0) -- (2/9+2/27+2/81+2/729,.75^4*.5*.75*.75*.5);
\draw[very thin] (2/9+2/27+2/81+2/243+2/729,0) -- (2/9+2/27+2/81+2/243+2/729,.75^5*.5*.75*.75*.75*.75);
\draw[very thin] (2/9+2/27+2/81+2/243,0) -- (2/9+2/27+2/81+2/243,.75^4*.5*.75*.75*.75);

\node at (1.25,.45) {$V^{1+}_n \,=\, \set{(x,y) \in V^1_n}{y \geq M^{n+1}y_n},$};
\node at (1.25,.1) {$V^{1-}_n \,=\, \set{(x,y) \in V^1_n}{y \leq M^{n+1}y_n}.$};

\draw[thin,densely dotted] (.12,.5*.75^3) -- (.45,.5*.75^3);

\end{tikzpicture}
\end{center}

\noindent In other words, $V^{1+}_n$ is just the top part of the leftmost blade of $V^1_n$, and $V^{1-}_n$ is everything else. In particular, note that $V^{1+}_n$ is a vertical line segment.

In the proof of the Theorem~\ref{thm:EPG}, we show that for any $m,n \in \N$ there is a homeomorphism $V^1_n \to V^1_m$, which in the proof is denoted $\big( h^1_m \big)^{-1} \circ h^1_n$. 
This homeomorphism scales and translates $V^{1+}_n$ onto $V^{1+}_m$ and (with a different scaling factor) $V^{1-}_n$ onto $V^{1-}_m$. 
For the present proof, we need a homeomorphism $V_n^1 \to V_{\psi(n)}^1$, defined by modifying $\big( h^1_{\psi(n)} \big)^{-1} \circ h^1_n$ as follows. 

%Fix $n \in \N$, and let $\phi_n = \phi \rest [0,y_n]$. 
The idea is that we would like to adjust the map $\big( h^1_{\psi(n)} \big)^{-1} \circ h^1_n: V_n^1 \to V_{\psi(n)}^1$ on the vertical segment $V_n^{1+}$ only, so that instead of a simple linear mapping onto $V_{\psi(n)}^{1+}$, it mimics $\phi \rest [0,y_n]$. 
More specifically, for each $n$ let $\ell_n$ denote the linear mapping $V^{1+}_n \to [0,y_n]$ sending the topmost point of $V^{1+}_n$ to $y_n$:
$$\ell_n(\nicefrac{2}{3^n},y) \,=\, \textstyle \frac{y_n}{y_n-M^{n+1}y_n}\big( y-M^{n+1}y_n \big).$$
Observe that these mappings are invertible, and $\big( \ell_{\psi(n)} \big)^{-1}$ is the linear mapping $[0,y_{\psi(n)}] \to V^{1+}_{\psi(n)}$ that sends $y_{\psi(n)}$ to the topmost point of $V^{1+}_{\psi(n)}$:
$$\big( \ell_{\psi(n)} \big)^{-1}(y) \,=\, \textstyle \Big( \frac{2}{3^{\psi(n)}} \,,\, \frac{y_{\psi(n)}-M^{\psi(n)+1}y_{\psi(n)}}{y_{\psi(n)}}y + M^{\psi(n)+1}y_{\psi(n)} \Big).$$
Define $h_n: V^1_n \to V^1_{\psi(n)}$ by setting
$$h_n(x) \,=\, \begin{cases}
\big( h^1_{\psi(n)} \big)^{-1} \circ h^1_n (z) &\text{ if } z \in V^{1-}_n, \\
\big( \ell_{\psi(n)} \big)^{-1} \circ \phi \circ \ell_n (z) & \text{ if } z \in V^{1+}_n.
\end{cases}$$ 
Note that $\ell_n$ is a linear mapping $V^{1+}_n \to [0,y_n]$, $\phi$ maps $[0,y_n]$ homeomorphically to $[0,y_{\psi(n)}]$, and $\big( \ell_{\psi(n)} \big)^{-1}$ is a linear mapping $[0,y_{\psi(n)}] \to V^{1+}_{\psi(n)}$. Thus the composition $\big( \ell_{\psi(n)} \big)^{-1} \circ \phi \circ \ell_n$ is a homeomorphism $V^{1+}_n \to V^{1+}_{\psi(n)}$. 
Furthermore, $\big( h^1_{\psi(n)} \big)^{-1} \circ h^1_n (z)$ is a homeomorphism $V^{1-}_n \to V^{1-}_{\psi(n)}$, and it is straightforward to check that each of these homeomorphisms sends the common point of $V^{1-}_n \cap V^{1+}_n$ to the common point of $V^{1-}_{\psi(n)} \cap V^{1+}_{\psi(n)}$. 
It follows that $h_n$ is a homeomorphism $V^1_n \to V^1_{\psi(n)}$.

Define a mapping $h: F_C \to F_C$ as follows:
$$h(z) \,=\, \begin{cases}
t &\text{ if } z = t, \\
\big( 0,\phi(y) \big) &\text{ if } z = (0,y) \in B_0 \text{ and } y > 0, \\
h_n(z) &\text{ if } z \in V^1_n \text{ for some } n.
\end{cases}$$
Observe that $F_C = \{t\} \cup C_{>0} = \{t\} \cup B_0 \cup \bigcup_{n \in \N}V^1_n$, so $h$ is defined on all of $F_C$. 
Because $\psi$ is a permutation of $\N$, and each $h_{n,\psi(n)}$ is a bijection $V^1_n \to V^1_{\psi(n)}$, and because $\phi: I \to I$ is a bijection with $\phi(0) = 0$, $h$ is a bijection. 
We claim that $h$ is also continuous. 

Because the $V^1_n$ are all clopen and the $h_{n,\psi(n)}$ are all homeomorphisms, it is clear that $h$ is continuous at every $z \in \bigcup_{n \in \N}V^1_n$. 

Let $z = (0,y) \in B_0 \setminus \{t\}$ (so $y \in (0,1]$). We show next that $h$ is continuous at $z$ in $\{z\} \cup \bigcup_{n \in \N}V^1_n$. 

To this end, let $\seq{z_n}{n \in \N}$ be a sequence of points in $\bigcup_{n \in \N}V^1_n$ converging to $z$. 
For each $n$ let $\bar x_n = \pi_1(z_n)$ and $\bar y_n = \pi_2(z_n)$, so that $z_n = (\bar x_n,\bar y_n)$. 

For each $n$, let $k_n$ denote the integer such that $z_n \in V^1_k$. 
Because each $V^1_k$ is clopen in $C$ and $\lim_{n \to \infty} z_n = (0,y) \notin \bigcup_{k \in \N}V^1_k$, no single $V^1_k$ can contain infinitely many of the $z_n$. Consequently, $\lim_{n \to \infty}k_n = \infty$. 

Next, observe that the maximum $Y$-coordinate of $V^{1-}_k$ is $M^{k+1}y_n$. Because $M = a_m < 1$, this goes to $0$ as $k$ increases. Because $\lim_{n \to \infty}k_n = \infty$ and $\lim_{n \to \infty} \bar y_n = y > 0$, it follows that $z_n \notin V^{1-}_{k_n}$ for sufficiently large values of $n$. Consequently, $z_n \in V^{1+}_{k_n}$ for all sufficiently large $n$. 

Thus, for sufficiently large $n$, we have
$$\bar x_n \,=\, \textstyle \frac{2}{3^{k_n}} \qquad \text{and} \qquad \bar y_n \in \big[ M^{k_n+1}y_n \,,\, y_n \big].$$
Plugging these values into the formulas above, we get $h(\bar x_n,\bar y_n) \,=\, $
$$\textstyle \Bigg( \frac{2}{3^{{\psi(k_n)}}} \,,\, \frac{y_{\psi(k_n)}-M^{\psi(k_n)+1}y_{\psi(k_n)}}{y_{\psi(k_n)}} \phi \bigg( \frac{y_{k_n}}{y_{k_n}-M^{k_n+1}y_{k_n}}\big( y-M^{k_n+1}y_{k_n} \big) \bigg) + M^{\psi(k_n)+1}y_{\psi(k_n)} \Bigg)$$
for all sufficiently large $n$. Recall that $\lim_{n \to \infty}k_n = \infty$, and (because $\psi$ is a bijection) this also implies $\lim_{n \to \infty} \psi(k_n) = \infty$. 
In particular, $\lim_{n \to \infty} \frac{2}{3^{\psi(k_n)}} = 0$ and, because $M < 1$, $\lim_{n \to \infty}M^{k_n+1} = 0$ and $\lim_{n \to \infty} M^{\psi(k_n)+1} = 0$. 
Therefore, taking the limit of the above expression for $h(\bar x_n,\bar y_n)$ as $n \to \infty$ gives 
$$\textstyle \lim_{n \to \infty} h(\bar x_n,\bar y_n) \,=\, (0,\phi(y)) \,=\, \phi(z).$$
Thus $h$ is continuous at $z$ in the space $\{z\} \cup \bigcup_{n \in \N}V^1_n$. 

But now observe that $h$ is clearly continuous at $z$ in $F_C \setminus \bigcup_{n \in \N}V^1_n$, because this space is just a copy of $[0,1]$ on which $h$ acts like $\phi$. 

Because $h$ is continuous at $z$ both in $\{z\} \cup \bigcup_{n \in \N}V^1_n$ and in $F_C \setminus \bigcup_{n \in \N}V^1_n$, it follows that $h$ is continuous at $z$ in $F_C$. Because $z$ was an arbitrary point of $B_0 \setminus \{t\}$, and because we already observed $h$ is continuous on $F_C \setminus B_0 = \bigcup_{n \in \N}V^1_n$, this shows that $h$ is continuous on $F_C \setminus \{t\}$. 

But if $h$ is continuous on $F_C \setminus \{t\}$ then it is continuous. Generally speaking, any continuous bijection $F_C \setminus \{t\} \to F_C \setminus \{t\}$ extends to a continuous bijection $F_C \to F_C$ by mapping $t$ to $t$, because $F_C$ is the one-point compactification of $F_C \setminus \{t\}$. 

Hence $h$ is a continuous bijection $F_C \to F_C$. Because $F_C$ is compact, $h$ is a homeomorphism. Furthermore, 
$h(0,p_i) = (0,\phi(p_i)) = (0,p_j)$ and so, as desired, we have found a homeomorphism sending one arbitrary point of $B_0 \cap P_L$ to another.

But, because $F_C$ is endpoint-homogeneous and $P_{L}$ is homeomorphism-invariant, for every $x \in P_{L}$ there is a homeomorphism $F_C \to F_C$ mapping it to a point of $B_0$. By composing homeomorphisms, it follows that if $x,y \in P_{L}$ then there is a homeomorphism $F_C \to F_C$ mapping $x$ to $y$. 

To complete the proof of the theorem, it now remains only to show that the same is true for the final piece of our partition, $P_G$. 
Thankfully, this follows relatively easily from the work we have already done for $P_L$, plus the statement $(*)$ from earlier in the proof. 

Observe that with $p_i,p_j$ as before, 
the homeomorphism $h$ described above does not only map $(0,p_i)$ to $(0,p_j)$; it maps the entire vertical line segment $\{0\} \times (p_i,p_{i+1})$ onto $\{0\} \times (\phi(p_i),\phi(p_{i+1})) = \{0\} \times (p_j,p_{j+1})$. 

Fix $x,y \in B_0 \cap P_G$. There are some $i,j \in \Z$ such that $x \in \{0\} \times (p_i,p_{i+1})$ and $y \in \{0\} \times (p_j,p_{j+1})$. As mentioned in the previous paragraph, there is a homeomorphism $F_C \to F_C$ that sends $\{0\} \times (p_i,p_{i+1})$ to $\{0\} \times (p_j,p_{j+1})$. 
But also, by the statement $(*)$, there is a homeomorphism mapping any given point of $\{0\} \times (p_i,p_{i+1})$ to any other point of $\{0\} \times (p_i,p_{i+1})$. By composing homeomorphisms, it follows that there is a homeomorphism $F_C \to F_C$ that maps $x$ to $y$. 

Thus for any two points of $B_0 \cap P_G$, there is a homeomorphism $F_C \to F_C$ mapping one to the other. 
But also, because $F_C$ is endpoint-homogeneous and $P_G$ is homeomorphism-invariant, for every $x \in P_G$ there is a homeomorphism $F_C \to F_C$ mapping it to a point of $B_0$. 
By composing homeomorphisms, it follows that if $x,y \in P_G$ then there is a homeomorphism $F_C \to F_C$ mapping $x$ to $y$. 
\end{proof}

%%%%%%%%%%%%
\section{Open problems, and some non-smooth fans}
%%%%%%%%%%%%

In this final section of the paper we state some open questions relating to endpoing-homogeneous and endpoint-generated fans. We also describe some examples of non-smooth fans relevant to these questions. These examples are meant more as illustrations than formal theorems. Accordingly, we do not include full proofs of some of our claims about these examples. 

\begin{problem}\label{q:classification}
Suppose $F$ and $G$ are smooth, endpoint-homogeneous fans with $EPG(F) = EPG(G)$. Are they homeomorphic?
\end{problem}

\noindent \textbf{Conjecture:} 
Yes, except for the degenerate case of $\0$-EPG fans (the $n$-ods).

\vspace{3mm}

%We conjecture that the answer to this question is yes, except for the degenerate case of the $\0$-EPG fans (the $n$-ods). 
\noindent Some partial results in this direction are known already:
\begin{itemize}
\item[$\circ$] The Cantor fan is the only $\{1\}$-EPG smooth fan (Corollary~\ref{cor:Cantor}).
\item[$\circ$] The Lelek fan is the only $[0,1]$-EPG smooth fan (Charatonik).
\item[$\circ$] The star is the only $\{0\}$-EPG smooth fan. (implicit in \cite{AHJ})
\item[$\circ$] The shrinking Cantor fan (i.e., the $\{0,1\}$-EPG fan constructed in Section 5; see the picture at the beginning of Section 6) is the only $\{0,1\}$-EPG smooth fan. (implicit in \cite[Theorem 2.2(2)]{AHJ})
\item[$\circ$] The $\set{\nicefrac{1}{n}}{n \geq 1}$-EPG smooth fans construction in Section 5 are the only $\set{\nicefrac{1}{n}}{n \in \N}$-EPG smooth fans. (see \cite{BGR})
\item[$\circ$] The $\set{\nicefrac{1}{n}}{n \geq 2}$-EPG smooth fans construction in Section 5 are the only $\set{\nicefrac{1}{n}}{n \in \N}$-EPG smooth fans. (see \cite{BGR})
\end{itemize}

\noindent It is worth pointing out that the above conjecture becomes false if the smoothness assumption is removed. The following example shows that if we allow non-smooth fans, then there are distinct homeomorphism types of $\{0,\nicefrac{1}{2}\}$-EPG fans, and there are even distinct homeomorphism types of non-smooth $\{0,\nicefrac{1}{2}\}$-EPG endpoint-homogeneous fans.

\begin{example}
Consider the following picture, which shows our construction of the comb for a $\{0,\nicefrac{1}{2}\}$-EPG endpoint-homogeneous smooth fan:

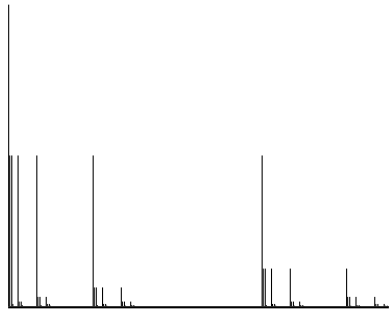
\vspace{1mm}

\begin{center}
\begin{tikzpicture}[xscale=5,yscale=4]

\draw[thick] (-.0014,-.0033) -- (1,-.0033);

\draw (0,0) -- (0,1);

\draw (2/3,0) -- (2/3,.5);
\draw (2/9,0) -- (2/9,.5);
\draw (2/27,0) -- (2/27,.5);
\draw (2/81,0) -- (2/81,.5);
\draw (2/243,0) -- (2/243,.5);
\draw (2/729,0) -- (2/729,.5);

\draw[thin] (2/3+2/9,0) -- (2/3+2/9,.5^3);
\draw[thin] (2/3+2/27,0) -- (2/3+2/27,.5^3);
\draw[thin] (2/3+2/81,0) -- (2/3+2/81,.5^3);
\draw[thin] (2/3+2/243,0) -- (2/3+2/243,.5^3);
\draw[thin] (2/3+2/729,0) -- (2/3+2/729,.5^3);

\draw[thin] (2/9+2/27,0) -- (2/9+2/27,.5^4);
\draw[thin] (2/9+2/81,0) -- (2/9+2/81,.5^4);
\draw[thin] (2/9+2/243,0) -- (2/9+2/243,.5^4);
\draw[thin] (2/9+2/729,0) -- (2/9+2/729,.5^4);

\draw[thin] (2/27+2/81,0) -- (2/27+2/81,.5^5);
\draw[thin] (2/27+2/243,0) -- (2/27+2/243,.5^5);
\draw[thin] (2/27+2/729,0) -- (2/27+2/729,.5^5);

\draw[thin] (2/81+2/243,0) -- (2/81+2/243,.5^6);
\draw[thin] (2/81+2/729,0) -- (2/81+2/729,.5^6);

\draw[thin] (2/243+2/729,0) -- (2/243+2/729,.5^7);

\draw[thin] (2/3+2/9+2/27,0) -- (2/3+2/9+2/27,.5^5);
\draw[thin] (2/3+2/9+2/81,0) -- (2/3+2/9+2/81,.5^5);
\draw[thin] (2/3+2/9+2/243,0) -- (2/3+2/9+2/243,.5^5);
\draw[thin] (2/3+2/9+2/729,0) -- (2/3+2/9+2/729,.5^5);
\draw[thin] (2/3+2/27+2/81,0) -- (2/3+2/27+2/81,.5^6);
\draw[thin] (2/3+2/27+2/243,0) -- (2/3+2/27+2/243,.5^6);
\draw[thin] (2/3+2/27+2/729,0) -- (2/3+2/27+2/729,.5^6);
\draw[thin] (2/3+2/81+2/243,0) -- (2/3+2/81+2/243,.5^7);
\draw[thin] (2/3+2/81+2/729,0) -- (2/3+2/81+2/729,.5^7);
\draw[thin] (2/3+2/243+2/729,0) -- (2/3+2/243+2/729,.5^8);

\draw[thin] (2/9+2/27+2/81,0) -- (2/9+2/27+2/81,.5^6);
\draw[thin] (2/9+2/27+2/243,0) -- (2/9+2/27+2/243,.5^6);
\draw[thin] (2/9+2/27+2/729,0) -- (2/9+2/27+2/729,.5^6);
\draw[thin] (2/9+2/81+2/243,0) -- (2/9+2/81+2/243,.5^7);
\draw[thin] (2/9+2/81+2/729,0) -- (2/9+2/81+2/729,.5^7);
\draw[thin] (2/9+2/243+2/729,0) -- (2/9+2/243+2/729,.5^8);

\draw[thin] (2/27+2/81+2/243,0) -- (2/27+2/81+2/243,.5^7);
\draw[thin] (2/27+2/81+2/729,0) -- (2/27+2/81+2/729,.5^7);
\draw[thin] (2/27+2/243+2/729,0) -- (2/27+2/243+2/729,.5^8);
\draw[thin] (2/81+2/243+2/729,0) -- (2/81+2/243+2/729,.5^8);

\draw[thin] (2/3+2/9+2/27+2/81,0) -- (2/3+2/9+2/27+2/81,.5^7);
\draw[thin] (2/3+2/9+2/27+2/243,0) -- (2/3+2/9+2/27+2/243,.5^7);
\draw[thin] (2/3+2/9+2/27+2/729,0) -- (2/3+2/9+2/27+2/729,.5^7);
\draw[thin] (2/3+2/9+2/81+2/243,0) -- (2/3+2/9+2/81+2/243,.5^8);
\draw[thin] (2/3+2/9+2/81+2/729,0) -- (2/3+2/9+2/81+2/729,.5^8);
\draw[thin] (2/3+2/9+2/243+2/729,0) -- (2/3+2/9+2/243+2/729,.5^9);
\draw[thin] (2/3+2/27+2/81+2/243,0) -- (2/3+2/27+2/81+2/243,.5^8);
\draw[thin] (2/3+2/27+2/81+2/729,0) -- (2/3+2/27+2/81+2/729,.5^8);
\draw[thin] (2/3+2/27+2/243+2/729,0) -- (2/3+2/27+2/243+2/729,.5^9);
\draw[thin] (2/3+2/81+2/243+2/729,0) -- (2/3+2/81+2/243+2/729,.5^9);
\draw[thin] (2/9+2/27+2/81+2/243,0) -- (2/9+2/27+2/81+2/243,.5^8);
\draw[thin] (2/9+2/27+2/81+2/729,0) -- (2/9+2/27+2/81+2/729,.5^8);
\draw[thin] (2/9+2/27+2/243+2/729,0) -- (2/9+2/27+2/243+2/729,.5^9);
\draw[thin] (2/9+2/81+2/243+2/729,0) -- (2/9+2/81+2/243+2/729,.5^9);
\draw[thin] (2/27+2/81+2/243+2/729,0) -- (2/27+2/81+2/243+2/729,.5^9);

\draw[thin] (2/3+2/9+2/27+2/81+2/243,0) -- (2/3+2/9+2/27+2/81+2/243,.5^9);
\draw[thin] (2/3+2/9+2/27+2/81+2/729,0) -- (2/3+2/9+2/27+2/81+2/729,.5^9);
\draw[thin] (2/3+2/9+2/27+2/243+2/729,0) -- (2/3+2/9+2/27+2/243+2/729,.5^10);
\draw[thin] (2/3+2/9+2/81+2/243+2/729,0) -- (2/3+2/9+2/81+2/243+2/729,.5^10);
\draw[thin] (2/3+2/27+2/81+2/243+2/729,0) -- (2/3+2/27+2/81+2/243+2/729,.5^10);
\draw[thin] (2/9+2/27+2/81+2/243+2/729,0) -- (2/9+2/27+2/81+2/243+2/729,.5^10);

\draw[thin] (2/3+2/9+2/27+2/81+2/243+2/729,0) -- (2/3+2/9+2/27+2/81+2/243+2/729,.5^10);

\end{tikzpicture}
\end{center}

\vspace{1mm}

\noindent Here are two ways of modifying this picture to obtain planar dendroids that, while no longer combs over $K$, still have the property that collapsing the bottom line to a point produces a (non-smooth) fan:

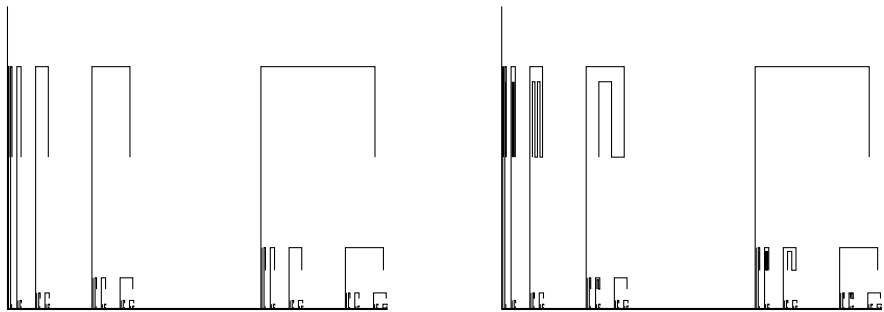
\vspace{1mm}

\begin{center}
\begin{tikzpicture}[xscale=5,yscale=4]

\draw[thick] (-.0014,-.0033) -- (1,-.0033);

\draw (0,0) -- (0,1);

\draw (2/3,0) -- (2/3,.8) -- (2/3+.3,.8) -- (2/3+.3,.5);
\draw (2/9,0) -- (2/9,.8) -- (2/9+.3/3,.8) -- (2/9+.3/3,.5);
\draw (2/27,0) -- (2/27,.8) -- (2/27+.3/9,.8) -- (2/27+.3/9,.5);
\draw (2/81,0) -- (2/81,.8) -- (2/81+.3/27,.8) -- (2/81+.3/27,.5);
\draw (2/243,0) -- (2/243,.8) -- (2/243+.3/81,.8) -- (2/243+.3/81,.5);
\draw (2/729,0) -- (2/729,.8) -- (2/729+.3/243,.8) -- (2/729+.3/243,.5);

\draw[thin] (2/3+2/9,0) -- (2/3+2/9,.5^2*.8) -- (2/3+2/9+.3/3,.5^2*.8) -- (2/3+2/9+.3/3,.5^2*.5);
\draw[thin] (2/3+2/27,0) -- (2/3+2/27,.5^2*.8) -- (2/3+2/27+.3/9,.5^2*.8) -- (2/3+2/27+.3/9,.5^2*.5);
\draw[thin] (2/3+2/81,0) -- (2/3+2/81,.5^2*.8) -- (2/3+2/81+.3/27,.5^2*.8) -- (2/3+2/81+.3/27,.5^2*.5);
\draw[thin] (2/3+2/243,0) -- (2/3+2/243,.5^2*.8) -- (2/3+2/243+.3/81,.5^2*.8) -- (2/3+2/243+.3/81,.5^2*.5);
\draw[thin] (2/3+2/729,0) -- (2/3+2/729,.5^2*.8);

\draw[thin] (2/9+2/27,0) -- (2/9+2/27,.5^3*.8) -- (2/9+2/27+.3/9,.5^3*.8) -- (2/9+2/27+.3/9,.5^3*.5);
\draw[thin] (2/9+2/81,0) -- (2/9+2/81,.5^3*.8) -- (2/9+2/81+.3/27,.5^3*.8) -- (2/9+2/81+.3/27,.5^3*.5);
\draw[thin] (2/9+2/243,0) -- (2/9+2/243,.5^3*.8) -- (2/9+2/243+.3/81,.5^3*.8) -- (2/9+2/243+.3/81,.5^3*.5);
\draw[thin] (2/9+2/729,0) -- (2/9+2/729,.5^3*.8);

\draw[thin] (2/27+2/81,0) -- (2/27+2/81,.5^4*.8) -- (2/27+2/81+.3/27,.5^4*.8) -- (2/27+2/81+.3/27,.5^4*.5);
\draw[thin] (2/27+2/243,0) -- (2/27+2/243,.5^4*.8) -- (2/27+2/243+.3/81,.5^4*.8) -- (2/27+2/243+.3/81,.5^4*.5);
\draw[thin] (2/27+2/729,0) -- (2/27+2/729,.5^4*.8);

\draw[thin] (2/81+2/243,0) -- (2/81+2/243,.5^5*.8) -- (2/81+2/243+.3/81,.5^5*.8) -- (2/81+2/243+.3/81,.5^5*.5);
\draw[thin] (2/81+2/729,0) -- (2/81+2/729,.5^5*.8);

\draw[thin] (2/243+2/729,0) -- (2/243+2/729,.5^6*.8);

\draw[thin] (2/3+2/9+2/27,0) -- (2/3+2/9+2/27,.5^4*.8) -- (2/3+2/9+2/27+.3/9,.5^4*.8) -- (2/3+2/9+2/27+.3/9,.5^4*.5);
\draw[thin] (2/3+2/9+2/81,0) -- (2/3+2/9+2/81,.5^4*.8) -- (2/3+2/9+2/81+.3/27,.5^4*.8) -- (2/3+2/9+2/81+.3/27,.5^4*.5);
\draw[thin] (2/3+2/9+2/243,0) -- (2/3+2/9+2/243,.5^4*.8) -- (2/3+2/9+2/243+.3/81,.5^4*.8) -- (2/3+2/9+2/243+.3/81,.5^4*.5);
\draw[thin] (2/3+2/9+2/729,0) -- (2/3+2/9+2/729,.5^4*.8);
\draw[thin] (2/3+2/27+2/81,0) -- (2/3+2/27+2/81,.5^5*.8) -- (2/3+2/27+2/81+.3/27,.5^5*.8) -- (2/3+2/27+2/81+.3/27,.5^5*.5);
\draw[thin] (2/3+2/27+2/243,0) -- (2/3+2/27+2/243,.5^5*.8) -- (2/3+2/27+2/243+.3/81,.5^5*.8) -- (2/3+2/27+2/243+.3/81,.5^5*.5);
\draw[thin] (2/3+2/27+2/729,0) -- (2/3+2/27+2/729,.5^5*.8);
\draw[thin] (2/3+2/81+2/243,0) -- (2/3+2/81+2/243,.5^6*.8) -- (2/3+2/81+2/243+.3/81,.5^6*.8) -- (2/3+2/81+2/243+.3/81,.5^6*.5);
\draw[thin] (2/3+2/81+2/729,0) -- (2/3+2/81+2/729,.5^6*.8);
\draw[thin] (2/3+2/243+2/729,0) -- (2/3+2/243+2/729,.5^7*.8);

\draw[thin] (2/9+2/27+2/81,0) -- (2/9+2/27+2/81,.5^5*.8) -- (2/9+2/27+2/81+.3/27,.5^5*.8) -- (2/9+2/27+2/81+.3/27,.5^5*.5);
\draw[thin] (2/9+2/27+2/243,0) -- (2/9+2/27+2/243,.5^5*.8) -- (2/9+2/27+2/243+.3/81,.5^5*.8) -- (2/9+2/27+2/243+.3/81,.5^5*.5);
\draw[thin] (2/9+2/27+2/729,0) -- (2/9+2/27+2/729,.5^5*.8);
\draw[thin] (2/9+2/81+2/243,0) -- (2/9+2/81+2/243,.5^6*.8) -- (2/9+2/81+2/243+.3/81,.5^6*.8) -- (2/9+2/81+2/243+.3/81,.5^6*.5);
\draw[thin] (2/9+2/81+2/729,0) -- (2/9+2/81+2/729,.5^6*.8);
\draw[thin] (2/9+2/243+2/729,0) -- (2/9+2/243+2/729,.5^7*.8);

\draw[thin] (2/27+2/81+2/243,0) -- (2/27+2/81+2/243,.5^6*.8) -- (2/27+2/81+2/243+.3/81,.5^6*.8) -- (2/27+2/81+2/243+.3/81,.5^6*.5);
\draw[thin] (2/27+2/81+2/729,0) -- (2/27+2/81+2/729,.5^6*.8);
\draw[thin] (2/27+2/243+2/729,0) -- (2/27+2/243+2/729,.5^7*.8);
\draw[thin] (2/81+2/243+2/729,0) -- (2/81+2/243+2/729,.5^7*.8);

\draw[thin] (2/3+2/9+2/27+2/81,0) -- (2/3+2/9+2/27+2/81,.5^6*.8) -- (2/3+2/9+2/27+2/81+.3/27,.5^6*.8) -- (2/3+2/9+2/27+2/81+.3/27,.5^6*.5);
\draw[thin] (2/3+2/9+2/27+2/243,0) -- (2/3+2/9+2/27+2/243,.5^6*.8) -- (2/3+2/9+2/27+2/243+.3/81,.5^6*.8) -- (2/3+2/9+2/27+2/243+.3/81,.5^6*.5);
\draw[thin] (2/3+2/9+2/27+2/729,0) -- (2/3+2/9+2/27+2/729,.5^6*.8);
\draw[thin] (2/3+2/9+2/81+2/243,0) -- (2/3+2/9+2/81+2/243,.5^7*.8) -- (2/3+2/9+2/81+2/243+.3/81,.5^7*.8) -- (2/3+2/9+2/81+2/243+.3/81,.5^7*.5);
\draw[thin] (2/3+2/9+2/81+2/729,0) -- (2/3+2/9+2/81+2/729,.5^7*.8);
\draw[thin] (2/3+2/9+2/243+2/729,0) -- (2/3+2/9+2/243+2/729,.5^8*.8);
\draw[thin] (2/3+2/27+2/81+2/243,0) -- (2/3+2/27+2/81+2/243,.5^7*.8) -- (2/3+2/27+2/81+2/243+.3/81,.5^7*.8) -- (2/3+2/27+2/81+2/243+.3/81,.5^7*.5);
\draw[thin] (2/3+2/27+2/81+2/729,0) -- (2/3+2/27+2/81+2/729,.5^7*.8);
\draw[thin] (2/3+2/27+2/243+2/729,0) -- (2/3+2/27+2/243+2/729,.5^8*.8);
\draw[thin] (2/3+2/81+2/243+2/729,0) -- (2/3+2/81+2/243+2/729,.5^8*.8);
\draw[thin] (2/9+2/27+2/81+2/243,0) -- (2/9+2/27+2/81+2/243,.5^7*.8) -- (2/9+2/27+2/81+2/243+.3/81,.5^7*.8) -- (2/9+2/27+2/81+2/243+.3/81,.5^7*.5);
\draw[thin] (2/9+2/27+2/81+2/729,0) -- (2/9+2/27+2/81+2/729,.5^7*.8);
\draw[thin] (2/9+2/27+2/243+2/729,0) -- (2/9+2/27+2/243+2/729,.5^8*.8);
\draw[thin] (2/9+2/81+2/243+2/729,0) -- (2/9+2/81+2/243+2/729,.5^8*.8);
\draw[thin] (2/27+2/81+2/243+2/729,0) -- (2/27+2/81+2/243+2/729,.5^8*.8);

\draw[thin] (2/3+2/9+2/27+2/81+2/243,0) -- (2/3+2/9+2/27+2/81+2/243,.5^8*.8) -- (2/3+2/9+2/27+2/81+2/243+.3/81,.5^8*.8) -- (2/3+2/9+2/27+2/81+2/243+.3/81,.5^8*.8);
\draw[thin] (2/3+2/9+2/27+2/81+2/729,0) -- (2/3+2/9+2/27+2/81+2/729,.5^8*.8);
\draw[thin] (2/3+2/9+2/27+2/243+2/729,0) -- (2/3+2/9+2/27+2/243+2/729,.5^9*.8);
\draw[thin] (2/3+2/9+2/81+2/243+2/729,0) -- (2/3+2/9+2/81+2/243+2/729,.5^9*.8);
\draw[thin] (2/3+2/27+2/81+2/243+2/729,0) -- (2/3+2/27+2/81+2/243+2/729,.5^9*.8);
\draw[thin] (2/9+2/27+2/81+2/243+2/729,0) -- (2/9+2/27+2/81+2/243+2/729,.5^9*.8);

\draw[thin] (2/3+2/9+2/27+2/81+2/243+2/729,0) -- (2/3+2/9+2/27+2/81+2/243+2/729,.5^9*.8);

\begin{scope}[shift={(1.3,0)}]

\draw[thick] (-.0014,-.0033) -- (1,-.0033);

\draw (0,0) -- (0,1);

\draw (2/3,0) -- (2/3,.8) -- (2/3+.3,.8) -- (2/3+.3,.5);
\draw (2/9,0) -- (2/9,.8) -- (2/9+.3/3,.8) -- (2/9+.3/3,.5) -- (2/9+.2/3,.5) -- (2/9+.2/3,.75) -- (2/9+.1/3,.75) -- (2/9+.1/3,.5);
\draw (2/27,0) -- (2/27,.8) -- (2/27+.3/9,.8) -- (2/27+.3/9,.5) -- (2/27+.24/9,.5) -- (2/27+.24/9,.75) -- (2/27+.18/9,.75) -- (2/27+.18/9,.5) -- (2/27+.12/9,.5) -- (2/27+.12/9,.75) -- (2/27+.06/9,.75) -- (2/27+.06/9,.5);
\draw (2/81,0) -- (2/81,.8) -- (2/81+.3/27,.8) -- (2/81+.3/27,.5); \draw[very thick] (2/81+.15/27,.75) -- (2/81+.15/27,.5);
\draw (2/243,0) -- (2/243,.8) -- (2/243+.3/81,.8) -- (2/243+.3/81,.5); \draw (2/243+.15/81,.75) -- (2/243+.15/81,.5);
\draw (2/729,0) -- (2/729,.8) -- (2/729+.3/243,.8) -- (2/729+.3/243,.5);

\draw[thin] (2/3+2/9,0) -- (2/3+2/9,.5^2*.8) -- (2/3+2/9+.3/3,.5^2*.8) -- (2/3+2/9+.3/3,.5^2*.5);
\draw[thin] (2/3+2/27,0) -- (2/3+2/27,.5^2*.8) -- (2/3+2/27+.3/9,.5^2*.8) -- (2/3+2/27+.3/9,.5^2*.5) -- (2/3+2/27+.2/9,.5^2*.5) -- (2/3+2/27+.2/9,.5^2*.75) -- (2/3+2/27+.1/9,.5^2*.75) -- (2/3+2/27+.1/9,.5^2*.5);
\draw[thin] (2/3+2/81,0) -- (2/3+2/81,.5^2*.8) -- (2/3+2/81+.3/27,.5^2*.8) -- (2/3+2/81+.3/27,.5^2*.5); \draw[very thin]  (2/3+2/81+.24/27,.5^2*.5) -- (2/3+2/81+.24/27,.5^2*.75) -- (2/3+2/81+.18/27,.5^2*.75) -- (2/3+2/81+.18/27,.5^2*.5) -- (2/3+2/81+.12/27,.5^2*.5) -- (2/3+2/81+.12/27,.5^2*.75) -- (2/3+2/81+.06/27,.5^2*.75) -- (2/3+2/81+.06/27,.5^2*.5);
\draw[thin] (2/3+2/243,0) -- (2/3+2/243,.5^2*.8) -- (2/3+2/243+.3/81,.5^2*.8) -- (2/3+2/243+.3/81,.5^2*.5);
\draw[thin] (2/3+2/243+.15/81,.5^2*.75) -- (2/3+2/243+.15/81,.5^2*.5);
\draw[thin] (2/3+2/729,0) -- (2/3+2/729,.5^2*.8);

\draw[thin] (2/9+2/27,0) -- (2/9+2/27,.5^3*.8) -- (2/9+2/27+.3/9,.5^3*.8) -- (2/9+2/27+.3/9,.5^3*.5);
\draw[thin] (2/9+2/81,0) -- (2/9+2/81,.5^3*.8) -- (2/9+2/81+.3/27,.5^3*.8) -- (2/9+2/81+.3/27,.5^3*.5) -- (2/9+2/81+.2/27,.5^3*.5) -- (2/9+2/81+.2/27,.5^3*.75) -- (2/9+2/81+.1/27,.5^3*.75) -- (2/9+2/81+.1/27,.5^3*.5);
\draw[thin] (2/9+2/243,0) -- (2/9+2/243,.5^3*.8) -- (2/9+2/243+.3/81,.5^3*.8) -- (2/9+2/243+.3/81,.5^3*.5);
\draw[thin] (2/9+2/243+.15/81,.5^3*.75) -- (2/9+2/243+.15/81,.5^3*.5);
\draw[thin] (2/9+2/729,0) -- (2/9+2/729,.5^3*.8);

\draw[thin] (2/27+2/81,0) -- (2/27+2/81,.5^4*.8) -- (2/27+2/81+.3/27,.5^4*.8) -- (2/27+2/81+.3/27,.5^4*.5);
\draw[thin] (2/27+2/243,0) -- (2/27+2/243,.5^4*.8) -- (2/27+2/243+.3/81,.5^4*.8) -- (2/27+2/243+.3/81,.5^4*.5);
\draw[thin] (2/27+2/243+.15/81,.5^4*.75) -- (2/27+2/243+.15/81,.5^4*.5);
\draw[thin] (2/27+2/729,0) -- (2/27+2/729,.5^4*.8);

\draw[thin] (2/81+2/243,0) -- (2/81+2/243,.5^5*.8) -- (2/81+2/243+.3/81,.5^5*.8) -- (2/81+2/243+.3/81,.5^5*.5);
\draw[thin] (2/81+2/729,0) -- (2/81+2/729,.5^5*.8);

\draw[thin] (2/243+2/729,0) -- (2/243+2/729,.5^6*.8);

\draw[thin] (2/3+2/9+2/27,0) -- (2/3+2/9+2/27,.5^4*.8) -- (2/3+2/9+2/27+.3/9,.5^4*.8) -- (2/3+2/9+2/27+.3/9,.5^4*.5);
\draw[thin] (2/3+2/9+2/81,0) -- (2/3+2/9+2/81,.5^4*.8) -- (2/3+2/9+2/81+.3/27,.5^4*.8) -- (2/3+2/9+2/81+.3/27,.5^4*.5) -- (2/3+2/9+2/81+.2/27,.5^4*.5) -- (2/3+2/9+2/81+.2/27,.5^4*.75) -- (2/3+2/9+2/81+.1/27,.5^4*.75) -- (2/3+2/9+2/81+.1/27,.5^4*.5);
\draw[thin] (2/3+2/9+2/243,0) -- (2/3+2/9+2/243,.5^4*.8) -- (2/3+2/9+2/243+.3/81,.5^4*.8) -- (2/3+2/9+2/243+.3/81,.5^4*.5);
\draw[thin] (2/3+2/9+2/243+.15/81,.5^4*.8) -- (2/3+2/9+2/243+.15/81,.5^4*.5);
\draw[thin] (2/3+2/9+2/729,0) -- (2/3+2/9+2/729,.5^4*.8);
\draw[thin] (2/3+2/27+2/81,0) -- (2/3+2/27+2/81,.5^5*.8) -- (2/3+2/27+2/81+.3/27,.5^5*.8) -- (2/3+2/27+2/81+.3/27,.5^5*.5);
\draw[thin] (2/3+2/27+2/243,0) -- (2/3+2/27+2/243,.5^5*.8) -- (2/3+2/27+2/243+.3/81,.5^5*.8) -- (2/3+2/27+2/243+.3/81,.5^5*.5);
\draw[thin] (2/3+2/27+2/243+.15/81,.5^5*.8) -- (2/3+2/27+2/243+.15/81,.5^5*.5);
\draw[thin] (2/3+2/27+2/729,0) -- (2/3+2/27+2/729,.5^5*.8);
\draw[thin] (2/3+2/81+2/243,0) -- (2/3+2/81+2/243,.5^6*.8) -- (2/3+2/81+2/243+.3/81,.5^6*.8) -- (2/3+2/81+2/243+.3/81,.5^6*.5);
\draw[thin] (2/3+2/81+2/729,0) -- (2/3+2/81+2/729,.5^6*.8);
\draw[thin] (2/3+2/243+2/729,0) -- (2/3+2/243+2/729,.5^7*.8);

\draw[thin] (2/9+2/27+2/81,0) -- (2/9+2/27+2/81,.5^5*.8) -- (2/9+2/27+2/81+.3/27,.5^5*.8) -- (2/9+2/27+2/81+.3/27,.5^5*.5);
\draw[thin] (2/9+2/27+2/243,0) -- (2/9+2/27+2/243,.5^5*.8) -- (2/9+2/27+2/243+.3/81,.5^5*.8) -- (2/9+2/27+2/243+.3/81,.5^5*.5);
\draw[thin] (2/9+2/27+2/243+.15/81,.5^5*.8) -- (2/9+2/27+2/243+.15/81,.5^5*.5);
\draw[thin] (2/9+2/27+2/729,0) -- (2/9+2/27+2/729,.5^5*.8);
\draw[thin] (2/9+2/81+2/243,0) -- (2/9+2/81+2/243,.5^6*.8) -- (2/9+2/81+2/243+.3/81,.5^6*.8) -- (2/9+2/81+2/243+.3/81,.5^6*.5);
\draw[thin] (2/9+2/81+2/729,0) -- (2/9+2/81+2/729,.5^6*.8);
\draw[thin] (2/9+2/243+2/729,0) -- (2/9+2/243+2/729,.5^7*.8);

\draw[thin] (2/27+2/81+2/243,0) -- (2/27+2/81+2/243,.5^6*.8) -- (2/27+2/81+2/243+.3/81,.5^6*.8) -- (2/27+2/81+2/243+.3/81,.5^6*.5);
\draw[thin] (2/27+2/81+2/729,0) -- (2/27+2/81+2/729,.5^6*.8);
\draw[thin] (2/27+2/243+2/729,0) -- (2/27+2/243+2/729,.5^7*.8);
\draw[thin] (2/81+2/243+2/729,0) -- (2/81+2/243+2/729,.5^7*.8);

\draw[thin] (2/3+2/9+2/27+2/81,0) -- (2/3+2/9+2/27+2/81,.5^6*.8) -- (2/3+2/9+2/27+2/81+.3/27,.5^6*.8) -- (2/3+2/9+2/27+2/81+.3/27,.5^6*.5);
\draw[thin] (2/3+2/9+2/27+2/243,0) -- (2/3+2/9+2/27+2/243,.5^6*.8) -- (2/3+2/9+2/27+2/243+.3/81,.5^6*.8) -- (2/3+2/9+2/27+2/243+.3/81,.5^6*.5);
\draw[thin] (2/3+2/9+2/27+2/243+.15/81,.5^6*.8) -- (2/3+2/9+2/27+2/243+.15/81,.5^6*.5);
\draw[thin] (2/3+2/9+2/27+2/729,0) -- (2/3+2/9+2/27+2/729,.5^6*.8);
\draw[thin] (2/3+2/9+2/81+2/243,0) -- (2/3+2/9+2/81+2/243,.5^7*.8) -- (2/3+2/9+2/81+2/243+.3/81,.5^7*.8) -- (2/3+2/9+2/81+2/243+.3/81,.5^7*.5);
\draw[thin] (2/3+2/9+2/81+2/729,0) -- (2/3+2/9+2/81+2/729,.5^7*.8);
\draw[thin] (2/3+2/9+2/243+2/729,0) -- (2/3+2/9+2/243+2/729,.5^8*.8);
\draw[thin] (2/3+2/27+2/81+2/243,0) -- (2/3+2/27+2/81+2/243,.5^7*.8) -- (2/3+2/27+2/81+2/243+.3/81,.5^7*.8) -- (2/3+2/27+2/81+2/243+.3/81,.5^7*.5);
\draw[thin] (2/3+2/27+2/81+2/729,0) -- (2/3+2/27+2/81+2/729,.5^7*.8);
\draw[thin] (2/3+2/27+2/243+2/729,0) -- (2/3+2/27+2/243+2/729,.5^8*.8);
\draw[thin] (2/3+2/81+2/243+2/729,0) -- (2/3+2/81+2/243+2/729,.5^8*.8);
\draw[thin] (2/9+2/27+2/81+2/243,0) -- (2/9+2/27+2/81+2/243,.5^7*.8) -- (2/9+2/27+2/81+2/243+.3/81,.5^7*.8) -- (2/9+2/27+2/81+2/243+.3/81,.5^7*.5);
\draw[thin] (2/9+2/27+2/81+2/729,0) -- (2/9+2/27+2/81+2/729,.5^7*.8);
\draw[thin] (2/9+2/27+2/243+2/729,0) -- (2/9+2/27+2/243+2/729,.5^8*.8);
\draw[thin] (2/9+2/81+2/243+2/729,0) -- (2/9+2/81+2/243+2/729,.5^8*.8);
\draw[thin] (2/27+2/81+2/243+2/729,0) -- (2/27+2/81+2/243+2/729,.5^8*.8);

\draw[thin] (2/3+2/9+2/27+2/81+2/243,0) -- (2/3+2/9+2/27+2/81+2/243,.5^8*.8) -- (2/3+2/9+2/27+2/81+2/243+.3/81,.5^8*.8) -- (2/3+2/9+2/27+2/81+2/243+.3/81,.5^8*.8);
\draw[thin] (2/3+2/9+2/27+2/81+2/729,0) -- (2/3+2/9+2/27+2/81+2/729,.5^8*.8);
\draw[thin] (2/3+2/9+2/27+2/243+2/729,0) -- (2/3+2/9+2/27+2/243+2/729,.5^9*.8);
\draw[thin] (2/3+2/9+2/81+2/243+2/729,0) -- (2/3+2/9+2/81+2/243+2/729,.5^9*.8);
\draw[thin] (2/3+2/27+2/81+2/243+2/729,0) -- (2/3+2/27+2/81+2/243+2/729,.5^9*.8);
\draw[thin] (2/9+2/27+2/81+2/243+2/729,0) -- (2/9+2/27+2/81+2/243+2/729,.5^9*.8);

\draw[thin] (2/3+2/9+2/27+2/81+2/243+2/729,0) -- (2/3+2/9+2/27+2/81+2/243+2/729,.5^9*.8);

\end{scope}

\end{tikzpicture}
\end{center}

\vspace{1mm}

\noindent These fans are still endpoint-homogeneous, with essentially the same proof as for the original version. (All the pieces of this picture are non-uniformly scaled versions of the whole, after making a topological adjustment to the leftmost blade.) 
This shows that there are at least two homeomorphism types of non-smooth, endpoint-homogeneous, $\{\nicefrac{1}{2}\}$-EPG fans. %Of course, one can image how to modify these combs in order to obtain many homeomorphism types.
\hfill$\dashv$
\end{example}

It should be clear that the two pictures shown here can be modified in many other ways, producing a zoo of non-smooth, endpoint-homogeneous, $\{0,\nicefrac{1}{2}\}$-EPG fans. And there is nothing special about the set $\{0,\nicefrac{1}{2}\}$: the same thing works for any of the sets considered in the proof of Theorem~\ref{thm:EPG}. In light of this, we do not expect a nice classification of endpoint-homogeneous non-smooth fans, along the lines of what the conjecture following Problem~\ref{q:classification} proposes for smooth fans.

Considering non-smooth fans raises further questions. The results in Sections 4 and 5 characterize precisely those $X$ such that there is a smooth $X$-EPG fan. But what if we remove the smoothness requirement? 

\begin{problem}\label{q:WhichX}
For which $X \sub [0,1]$ is there an $X$-EPG fan?
\end{problem}

The following example shows that it is possible to have a non-smooth $[\nicefrac{1}{2},1]$-EPG fan, despite the fact that there are no smooth $[\nicefrac{1}{2},1]$-EPG fans. 

\begin{example}\label{ex:3d}
The Cantor space $K$ has a basis of clopen sets, each homeomorphic to $K$, and also contains many nowhere dense sets homeomorphic to $K$ (because every nonempty closed subset of $K$ is homeomorphic to $K$). Using these two facts, a straightforward recursive argument can be used to find a sequence $\seq{K_n}{n \in \N}$ of nowhere dense subsets of $K$, each homeomorphic to $K$, such that for any sequence of points $\seq{d_n}{n \in \N}$ such that $d_n \in K_n$ for all $n$, the set $\set{d_n}{n \in \N}$ is dense in $K$. 

Let $C_K$ denote the Cantor comb, viewed as a subset of the $X$-$Y$ plane in $3$-dimensional space: i.e., $C_K = (K \times [0,1] \times \{0\}) \cup ([0,1] \times \{0\} \times \{0\})$. 

Let $C_L$ be a comb whose corresponding fan is the Lelek fan (that is, $E(C_L)$ is dense in $C_L$). 
For each $n \in \N$, define $C^n_L \sub \R^3$ as follows. 
First, fix a homeomorphism $\phi_n: K \to K_n$. 
Next, define a function $C_L \to \R^3$ by
$$(x,y) \textstyle \,\mapsto\, \Big( \phi_n(x) \,,\, 1-\frac{y}{2} \,,\, \frac{y}{2n} \Big).$$
Finally, let $C_L^n$ be the image of $C_L$ under this mapping. Roughly speaking, $C_L^n$ is a copy of the Lelek comb, with the bottom line removed, but (1) situated over $K_n$ instead of over $K$, and (2) beginning at the top of the Cantor comb $C_K$, instead of on the $X$-axis, and $(3)$ tilted down towards the Cantor comb, so that $C^n_L$ approaches $C_K$ as $n$ increases. 
Let
$$C \,=\, C_K \cup \textstyle \bigcup_{n \in \N}C^n_L.
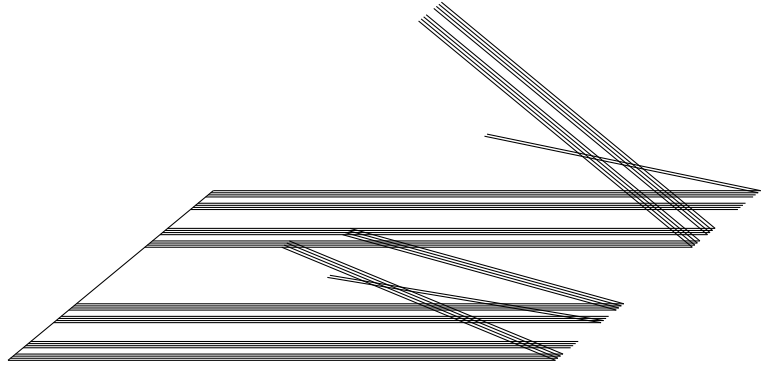
$$

\begin{center}
\begin{tikzpicture}[xscale=.9,yscale=.75]

\draw (0,0) -- (3,3);

\draw (0,0) -- (8,0);
\draw (1/27,1/27) -- (8+1/27,1/27);
\draw (2/27,2/27) -- (8+2/27,2/27);
\draw (1/9,1/9) -- (8+1/9,1/9);
\draw (2/9,2/9) -- (8+2/9,2/9);
\draw (7/27,7/27) -- (8+7/27,7/27);
\draw (8/27,8/27) -- (8+8/27,8/27);
\draw (1/3,1/3) -- (8+1/3,1/3);
\begin{scope}[shift={(2/3,2/3)}]
\draw (0,0) -- (8,0);
\draw (1/27,1/27) -- (8+1/27,1/27);
\draw (2/27,2/27) -- (8+2/27,2/27);
\draw (1/9,1/9) -- (8+1/9,1/9);
\draw (2/9,2/9) -- (8+2/9,2/9);
\draw (7/27,7/27) -- (8+7/27,7/27);
\draw (8/27,8/27) -- (8+8/27,8/27);
\draw (1/3,1/3) -- (8+1/3,1/3);\end{scope}
\begin{scope}[shift={(2,2)}]
\draw (0,0) -- (8,0);
\draw (1/27,1/27) -- (8+1/27,1/27);
\draw (2/27,2/27) -- (8+2/27,2/27);
\draw (1/9,1/9) -- (8+1/9,1/9);
\draw (2/9,2/9) -- (8+2/9,2/9);
\draw (7/27,7/27) -- (8+7/27,7/27);
\draw (8/27,8/27) -- (8+8/27,8/27);
\draw (1/3,1/3) -- (8+1/3,1/3);
\begin{scope}[shift={(2/3,2/3)}]
\draw (0,0) -- (8,0);
\draw (1/27,1/27) -- (8+1/27,1/27);
\draw (2/27,2/27) -- (8+2/27,2/27);
\draw (1/9,1/9) -- (8+1/9,1/9);
\draw (2/9,2/9) -- (8+2/9,2/9);
\draw (7/27,7/27) -- (8+7/27,7/27);
\draw (8/27,8/27) -- (8+8/27,8/27);
\draw (1/3,1/3) -- (8+1/3,1/3);\end{scope}\end{scope}

\begin{scope}[shift={(2,2)}]
\draw (8,0) -- (4,4);
\draw (8+1/27,1/27) -- (4+1/27,4+1/27);
\draw (8+2/27,2/27) -- (4+2/27,4+2/27);
\draw (8+1/9,1/9) -- (4+1/9,4+1/9);
\draw (8+2/9,2/9) -- (4+2/9,4+2/9);
\draw (8+7/27,7/27) -- (4+7/27,4+7/27);
\draw (8+8/27,8/27) -- (4+8/27,4+8/27);
\draw (8+1/3,1/3) -- (4+1/3,4+1/3);\end{scope}

\begin{scope}[shift={(2/3,2/3)}]
\draw (8+2/9,2/9) -- (4+2/9,4/3+2/9);
\draw (8+7/27,7/27) -- (4+7/27,4/3+7/27);
\draw (8+8/27,8/27) -- (4+8/27,4/3+8/27);
\draw (8+1/3,1/3) -- (4+1/3,4/3+1/3);\end{scope}

\draw (8,0) -- (4,2);
\draw (8+1/27,1/27) -- (4+1/27,2+1/27);
\draw (8+2/27,2/27) -- (4+2/27,2+2/27);
\draw (8+1/9,1/9) -- (4+1/9,2+1/9);

\begin{scope}[shift={(26/9,26/9)}]
\draw (8+2/27,2/27) -- (4+2/27,1+2/27);
\draw (8+1/9,1/9) -- (4+1/9,1+1/9);\end{scope}

\begin{scope}[shift={(16/27,16/27)}]
\draw (8+2/27,2/27) -- (4+2/27,4/5+2/27);
\draw (8+1/9,1/9) -- (4+1/9,4/5+1/9);\end{scope}

\end{tikzpicture}
\end{center}

\vspace{1mm}

\noindent Here we have illustrated $C_K \cup \bigcup_{n=1}^5 C^n_L$ (but with the copies of the Lelek fan drawn more like the Cantor fan, so that they show up better; and we have made no attempt to show that the $K_n$ are nowhere dense).

It is not difficult to check that $C$ is a dendroid. Furthermore (like with combs) collapsing $[0,1] \times \{0\} \times \{0\}$ to a point yields a (non-smooth) fan. We claim it is a $[\nicefrac{1}{2},1]$-EPG fan. 

Now let $B$ be a blade of the fan. These come in two types: some lie flat in the $X$-$Y$ plane (for example, the co-meagerly many blades above points in $K \setminus \bigcup_{n=1}^\infty K_n$ in $C_K$), and some have a piece angled out of the $X$-$Y$ plane. Let us called these the type I and type II blades, respectively. 

Because of how we chose the $K_n$, and because the endpoints of the Lelek comb are dense in the Lelek comb, the closure of the endpoints of the type II blades is all of $K \times [\nicefrac{1}{2},1] \times \{0\}$, i.e., the top half of the comb $C_K$.

If $B$ is a type I blade, then the top half of $B$ is contained in $\closure{E(C) \setminus B}$, by the previous paragraph. But no end points of $C$ have $X$-coordinate below $\nicefrac{1}{2}$, so $B \cap \closure{E(C) \setminus B}$ is precisely the top half of $B$. 

If $B$ is a type II blade, then $B = B' \cup B''$, where $B'$ lies flat in the $X$-$Y$ plane and $B''$ is angled up from the $X$-$Y$ plane. As in the previous paragraph, $B' \cap \closure{E(C) \setminus B}$ is precisely the top half of $B'$. But observe that $B''$ is a blade in a copy of the Lelek fan attached to $C_K$, which means that all of $B''$ is contained in $\closure{E(C) \setminus B}$. Hence $B \cap \closure{E(C) \setminus B}$ is the top half of $B'$ plus all of $B''$, which is a closed final segment of $B$.
\hfill$\dashv$
\end{example}

%Observe that this example shows the smoothness assumption in Theorem~\ref{thm:BCT1} is necessary. 

\begin{problem}
Is the fan from Example~\ref{ex:3d} planar?
\end{problem}

Observe that this example of a non-smooth $[\nicefrac{1}{2}]$-EPG fan is {not} endpoint-homogeneous. (This is because only some of the blades of the fan witness its non-smoothness. More precisely, say $B \in \B(F)$ is \emph{smooth} if for every sequence $\seq{x_n}{n \in \N}$ in $F \setminus B$ converging to a point $x \in B$, we have $\lim_{n \to \infty}[t,x_n] = [t,x]$. Some of the blades in our $[\nicefrac{1}{2}]$-EPG fan are smooth (the type II blades), while some are not (the type I blades), and no homeomorphism of the fan can map a smooth blade onto a non-smooth blade.)

\begin{problem}\label{q:WhichX}
For which $X \sub [0,1]$ is there an endpoint-homogeneous $X$-EPG (but not necessarily smooth) fan?
\end{problem}

While the previous example shows there is an endpoint-generated fan that is not endpoint-homogeneous, we do not have a smooth example.

\begin{problem}
Is every $X$-EPG smooth fan endpoint-homogeneous?
\end{problem}

\noindent If the answer to this question is yes, then our conjecture above (if true) would provide a classification of all endpoint-generated smooth fans. 

\begin{problem}
Up to homeomorphism, is there only one $\frac{1}{6}$-homogeneous smooth fan?
\end{problem} 

As mentioned already, Acosta, Hoehn, and Ju\'arez prove in \cite{AHJ} that there is no $\frac{1}{4}$-homogeneous smooth fan. They do not rule out the possibility of a $\frac{1}{4}$-homogeneous non-smooth fan, however; they even prove some results about what such a fan would look like (if it exists). 

\begin{problem}
Is there a $\frac{1}{4}$-homogeneous (non-smooth) fan?
\end{problem} 

For the next question, recall that $X,Y \sub [0,1]$ are \emph{equivalently embedded} if there is an order-preserving homeomorphism $[0,1] \to [0,1]$ mapping $X$ onto $Y$. Given $X \sub [0,1]$, let us say $x,x' \in X$ are \emph{equivalently embedded} if there is an order-preserving homeomorphism $\phi: [0,1] \to [0,1]$ mapping $X$ onto $X$ such that $\phi(x) = x'$.

\begin{problem}
Suppose $X \sub [0,1]$ has $n$ equivalence classes of points with respect to the ``equivalently embedded'' relation described above. If there is an $X$-EPG fan, must it be $\frac{1}{n}$-homogeneous? More specifically, is this true for the smooth fans constructed in Section 5?
\end{problem}

We end with a question less directly related to the material of the paper. Endpoint-homogeneous fans have rich automorphism groups in one sense. But there are other senses in which a space can have a rich automorphism group, for example, the automorphism group's universal minimal flow can be extremely amenable. The Lelek fan has this property \cite{BK}, and the fans constructed in Section 5 are, in a sense, generalizations of the Lelek fan. This suggests the following question.

\begin{problem}
Except for the degenerate cases of the simple $n$-ods and the star, is the universal minimal flow of automorphism group of an endpoint-homogeneous fan extremely amenable? More specifically, is this true for the smooth fans constructed in Section 5?
\end{problem}


\begin{thebibliography}{9}

\bibitem{AO} J. M. Aarts and L. G. Oversteegen, ``The homeomorphism group of the hairy arc,'' \emph{Compositio Math.} \textbf{96} (1995), pp. 283--292.

\bibitem{AHJ} G. Acosta, L. Hoehn, and Y. Pacheco-Ju\'arez, ``Homogeneity degree of fans,'' \emph{Topology and its Applications} \textbf{231} (2017), pp. 320--328.

\bibitem{APJ} G. Acosta and Y. Pacheco-Ju\'arez, ``$\frac{1}{3}$-homogeneous dendrites,'' \emph{Topology and its Applications} \textbf{219} (2017), pp. 55--77.

\bibitem{BK0} D. Barto\v{s}ov\'a and A. Kwiatkowska, ``Lelek fan from a projective Fra\"{i}ss\'e limit,'' \emph{Fundamenta Mathematicae} \textbf{231} (2015), pp. 57--79.

\bibitem{BK} D. Barto\v{s}ov\'a and A. Kwiatkowska, ``The universal minimal flow of the group of homeomorphisms of the Lelek fan,'' \emph{Transactions of the American Mathematical Society} \textbf{371} (2019), pp. 6995--7027.

%\bibitem{BPV} J. Bobok, P. Pyrih, and B. Vejnar, ``Half-homogeneous chainable continua with end points,'' \emph{Topology and its Applications} \textbf{160} (2013), pp. 1066--1073.

%\bibitem{Bor} J. Boro\'nski, ``On indecomposable $\frac{1}{2}$-homogeneous circle-like continua,'' \emph{Topology and its Applications} \textbf{160} (2013), pp. 59--62.

%\bibitem{Bor2} J. Boro\'nski, ``On the number of orbits of the homeomorphism group of solenoidal spaces,'' \emph{Topology and its Applications} \textbf{182} (2015), pp. 98--106.

%\bibitem{BGK} J. Boro\'nski, G. Gruenhage, and G. Kozlowski, ``$\frac{1}{k}$-homogeneous long solenoids,'' \emph{Monatshefte Math.} \textbf{180} (2016), pp. 171--192.

%\cite{JHMPC}, \cite{NPCP}, and \cite{VLPCP}

\bibitem{BGR} W. Brian and R. Gril Rogina, ``A complete classification of the two lines inverse limit,'' forthcoming.

%\bibitem{Brouwer} L.E.J. Brouwer, On the structure of perfect sets of points, in: KNAW, Proceedings, 12, 1909-1910, Amsterdam, 1910, pp. 785-794

\bibitem{lelekUniqueBO} W.~Bula, L.~Oversteegen, ``A characterization of smooth Cantor bouquets," \emph{Proceedings of the American Mathematical Society} \textbf{108} (1990), pp. 529--534.

\bibitem{CC} J. J. Charatonik and W. J. Charatonik, ``Images of the Cantor fan," \emph{Topology and its Applications} \textbf{33} (1989), pp. 163--172.

\bibitem{C} W. J. Charatonik, ``On fans," \emph{Dissertationes Math. Rozprawy Mat.} \textbf{54} (1967), no. 39.

\bibitem{lelekUniqueChar} Charatonik, Wlodzimierz. (1989). The Lelek fan is unique. Houston Journal of Mathematics. 15.

\bibitem{E} C. Eberhart, ``A note on smooth fans,'' \emph{Colloq. Math.} \textbf{20} (1969), pp. 89--90.

\bibitem{Connecting} R. Gril Rogina, ``Towards the Complete Classification of the Two Lines Inverse Limit,"  \emph{Journal of Dynamical and Control Systems} \textbf{29} (2023), pp. 1525--1545.

\bibitem{HGH} R. Hern\'andez-Guti\'errez and L. Hoehn, ``Smoth fans that are endpoint rigid,'' \emph{Applied General Topology} \textbf{24} (2023), pp. 407--422.

\bibitem{Nadler} S.~B.~Nadler, Jr., \emph{Continuum Theory, an Introduction}, CRC Press, Oregon, 1992.

\end{thebibliography}
\end{document}